\documentclass[preprint,nonatbib]{elsarticle}
\usepackage{fullpage} 

\makeatletter
\let\c@author\relax
\makeatother

\usepackage[table]{xcolor}
\usepackage{amsmath}
\usepackage{amsthm}
\usepackage{amssymb} 
\usepackage{hyperref}
\usepackage{graphicx}
\usepackage{float}
\usepackage{tikz-cd}
\usepackage{caption}
\usepackage{subcaption}
\usepackage{array}
\usepackage{extarrows}
\usepackage{multirow}
\usepackage{tikz}
\usepackage{makecell}
\usepackage{hhline}
\usepackage{cleveref}
\usepackage{boldline}
\usepackage{booktabs}
\usetikzlibrary{cd}
\usetikzlibrary{math}

\DeclareMathOperator{\LS}{2-RP}
\DeclareMathOperator{\LC}{3-RP}

\DeclareMathOperator{\OA}{OA}

\allowdisplaybreaks[3]

\newtheorem{theorem}{Theorem}[section]
\newtheorem{lemma}[theorem]{Lemma}

\newtheorem{corollary}[theorem]{Corollary}

\theoremstyle{definition}
\newtheorem{definition}[theorem]{Definition}
\newtheorem{example}[theorem]{Example}
\newtheorem{construction}[theorem]{Construction}
\newtheorem{append}[theorem]{}

\numberwithin{equation}{section}

\usepackage{biblatex} %Imports biblatex package
\hypersetup{pdfauthor=author}

\makeatletter
\def\ps@pprintTitle{%
  \let\@oddhead\@empty
  \let\@evenhead\@empty
  \def\@oddfoot{\reset@font\hfil\thepage\hfil}
  \let\@evenfoot\@oddfoot
}
\makeatother

\begin{document}

\title{Latin cubes with disjoint subcubes of two orders}

\author[]{Tara Kemp\corref{cor1}}
\ead{t.kemp@uq.net.au}
\author[]{James Lefevre}
\ead{j.lefevre@uq.edu.au}

\address{School of Mathematics and Physics, ARC Centre of Excellence, Plant Success in Nature and Agriculture, The University of Queensland, Brisbane, Queensland, 4072, Australia}

\cortext[cor1]{Corresponding author}

\begin{abstract}
Given a partition $h_1+h_2+\dots+h_k = n$, a latin square of order $n$ with pairwise disjoint subsquares of orders $h_1\dots h_k$ is called a realization. When the values $h_i$ are of at most two sizes, the existence of a realization has been completely determined. However, the existence of a latin cube with pairwise disjoint subcubes of two orders is only partially solved. In this paper, we determine existence for such latin cubes in almost all cases.
\end{abstract}


\maketitle

\section{Preliminaries}

The study of latin squares with sets of pairwise disjoint subsquares was originally motivated by a question posed by L.~Fuchs \cite{keedwell2015latin} about quasigroups with disjoint subquasigroups. Although much work has been done, the problem of existence of such latin squares remains unsolved. The original problem requires that the orders of the subsquares partition the order of the latin square, and in that case existence has been completely determined when the subsquares are of at most two different orders or there are at most five subsquares. An analogous problem has been defined for latin cubes and latin hypercubes in \cite{donovan2025latin}, and the same first case, where the subcubes are of at most two orders, is only partially solved.

A \emph{latin square} of order $n$ is an $n\times n$ array, $L$, with symbols from $[n] = \{1,2,\dots,n\}$ such that each symbol occurs exactly once in each row and column. The notation $a+[n]$ represents the set $\{a+i\mid i\in[n]\}$. A \emph{subsquare} of order $m$ is an $m\times m$ subarray of $L$ that is itself a latin square of order $m$. It is well-known that for any subsquare, if $m\neq n$ then $m\leq \frac{n}{2}$ \cite{evans1960embedding}.

An \emph{orthogonal array} with strength $t$, $k$ constraints, $n$ levels and index $\lambda$, denoted $\OA_\lambda(t,k,n)$, is a $k\times (\lambda n^t)$ array $A$ with $n$ symbols such that each possible $t$-tuple of symbols occurs exactly $\lambda$ times as a column in any $t$-rowed subarray of $A$. If $\lambda=1$, then it is often omitted. A latin square of order $n$ is equivalent to an $\OA(2,3,n)$, and so orthogonal arrays can be used to generalise latin squares into higher dimensions.

There are competing definitions for latin cubes, equivalent to either an $\OA_n(2,4,n)$ or $\OA(3,4,n)$. The latter are sometimes called \emph{permutation cubes} and that is the type we consider here. Thus, a \emph{latin cube} of order $n$ is an $n\times n\times n$ array $C$ with $n$ symbols such that each symbol occurs exactly once in each line of $C$, where a line is a set of $n$ cells found by fixing two of the three coordinates. Lines are denoted by $(i,j,[n])$, $(i,[n],j)$ or $([n],i,j)$ and referred to respectively as files, rows and columns. A layer is a set of cells with one fixed coordinate and represented by $(i,[n],[n])$, $([n],i,[n])$ or $([n],[n],i)$. The notation $(A,B,C)$ represents the set of cells $\{(i,j,k)\mid i\in A, j\in B, k\in C\}$, and a $\cdot$ in place of any of the sets represents $[n]$.

A \emph{subcube} of order $m$ is defined similarly to a subsquare; an $m\times m\times m$ subarray of $C$ that is itself a latin cube of order $m$. Two subcubes are disjoint if they share no layer or symbol.

\begin{example}
\label{ex: cube}
The arrays in \Cref{fig: 221 cube} are the layers of an order 5 latin cube with three pairwise disjoint subcubes. The subcubes are highlighted.
    \begin{figure}[h]
    \centering
    $\arraycolsep=4pt\begin{array}{|ccccc|} \hline
    \cellcolor{lightgray}1 & \cellcolor{lightgray}2 & 3 & 5 & 4\\
    \cellcolor{lightgray}2 & \cellcolor{lightgray}1 & 5 & 4 & 3\\
    3 & 5 & 4 & 2 & 1\\
    4 & 3 & 2 & 1 & 5\\
    5 & 4 & 1 & 3 & 2\\ \hline \end{array}$
    \quad
    $\arraycolsep=4pt\begin{array}{|ccccc|} \hline
    \cellcolor{lightgray}2 & \cellcolor{lightgray}1 & 5 & 4 & 3\\ 
    \cellcolor{lightgray}1 & \cellcolor{lightgray}2 & 4 & 3 & 5\\ 
    4 & 3 & 1 & 5 & 2\\ 
    5 & 4 & 3 & 2 & 1\\ 
    3 & 5 & 2 & 1 & 4\\ \hline \end{array}$
    \quad
    $\arraycolsep=4pt\begin{array}{|ccccc|} \hline
    5 & 3 & 4 & 1 & 2\\ 
    3 & 4 & 2 & 5 & 1\\ 
    1 & 2 & \cellcolor{lightgray}3 & 4 & \cellcolor{lightgray}5\\ 
    2 & 5 & 1 & 3 & 4\\ 
    4 & 1 & \cellcolor{lightgray}5 & 2 & \cellcolor{lightgray}3\\ \hline \end{array}$
    \quad
    $\arraycolsep=4pt\begin{array}{|ccccc|} \hline
    4 & 5 & 2 & 3 & 1\\ 
    5 & 3 & 1 & 2 & 4\\ 
    2 & 4 & \cellcolor{lightgray}5 & 1 & \cellcolor{lightgray}3\\ 
    3 & 1 & 4 & 5 & 2\\ 
    1 & 2 & \cellcolor{lightgray}3 & 4 & \cellcolor{lightgray}5\\ \hline \end{array}$
    \quad
    $\arraycolsep=4pt\begin{array}{|ccccc|} \hline
    3 & 4 & 1 & 2 & 5\\ 
    4 & 5 & 3 & 1 & 2\\ 
    5 & 1 & 2 & 3 & 4\\ 
    1 & 2 & 5 & \cellcolor{lightgray}4 & 3\\ 
    2 & 3 & 4 & 5 & 1\\ \hline \end{array}$
    \caption{An order 5 latin cube with disjoint subcubes.}
    \label{fig: 221 cube}
    \end{figure}
    
\end{example}

We define partial latin squares and cubes here, as they are needed in later sections. A \emph{partial latin square} of order $n$ is an $n\times n$ array in which each cell contains at most one symbol from $[n]$ and each symbol occurs at most once in each row and column. An $r\times s$ array with symbols from $[n]$, where $n\geq r,s$, is a \emph{latin rectangle} if each symbol occurs at most once in each row and column. There are analogous definitions for latin cubes. A \emph{partial latin cube} is an $n\times n\times n$ array such that each cell contains at most one symbol from $[n]$ and each symbol occurs at most once in each line. A \emph{latin box} is an $r\times s\times t$ array with $n\geq r,s,t$ symbols such that each symbol occurs at most once in each line.

Similarly, a \emph{partial $\OA_\lambda(t,k,n)$} is a $k\times m$ array $A$, for some $m\leq \lambda n^t$, with $n$ symbols such that each possible $t$-tuple of symbols occurs at most $\lambda$ times as a column in any $t$-rowed subarray of $A$.

If the empty cells of a partial latin square (or cube) $L$ can be filled with symbols from $[n]$ to form a latin square (cube) $L'$, then we say that $L$ can be completed and $L'$ is a completion of $L$.

Throughout, we use $P = (h_1\dots h_k)$ or $(h_1^{\alpha_1}h_2^{\alpha_2}\dots h_m^{\alpha_m})$ to represent a partition of $n$ with $\sum_{i=1}^kh_i = n$ or $\sum_{i=1}^m\alpha_ih_i = n$.

Returning to latin squares and the problem introduced by L.~Fuchs, a \emph{realization} of a partition $P = (h_1\dots h_k)$ of $n$, denoted $\LS(h_1\dots h_k)$ or just $\text{RP}(h_1\dots h_k)$, is a latin square of order $n$ with pairwise disjoint subsquares of orders $h_1,h_2,\dots,h_k$. The subsquares must be row, column and symbol disjoint.

Relevant to this work are the results of D{\'e}nes and P{\'a}sztor \cite{denes1963some}, Heinrich \cite{heinrich2006latin,heinrich1982disjoint} and Kuhl and Schroeder \cite{kuhl2018latin}.

\begin{theorem}[\cite{denes1963some}]
\label{thm: a^k square}
    For $k\geq 1$ and $a\geq 1$, a $\LS(a^k)$ exists if and only if $k\neq 2$.
\end{theorem}

\begin{theorem}[\cite{heinrich2006latin},\cite{heinrich1982disjoint}, Theorem 5.1 of \cite{kuhl2018latin}]
\label{thm: squaresatmost2}
    For $a>b>0$ and $k\geq 4$, a $\LS(a^ub^{k-u})$ exists if and only if $u\geq 3$, or $u\in[2]$ and $a\leq (k-2)b$ except when $k=4$ and $u=2$.
\end{theorem}

Given a partition $P$ of $n$, a \emph{3-realization} of $P$, denoted $\LC(h_1\dots h_k)$, is a latin cube of order $n$ with pairwise disjoint subcubes of orders $h_1,\dots,h_k$. Unless otherwise stated, we assume that $h_1\geq h_2\geq\dots\geq h_k>0$. The latin cube in \Cref{ex: cube} is a $\LC(2^21^1)$.

A 3-realization is in \emph{normal form} if for each $m\in[k]$ and $i,j,\ell\in H_m$, the cell $(i,j,\ell)$ contains a symbol from $H_m$, where $H_m = \sum_{i=1}^{m-1}h_i +[h_m]$. A 3-realization in normal form has the subcubes along the main diagonal, with the orders of the subcubes in decreasing order and the symbols used in each subcube in increasing order (note that this is a stronger condition than the one given in \cite{donovan2025latin}). Any 3-realization can be put into normal form via an appropriate permutation of layers and symbols.

The results in \cite{donovan2025latin} focus on partitions with at most two distinct integers and correspond to \Cref{thm: a^k square,thm: squaresatmost2}. Throughout, we assume that $a>b$.

\begin{theorem}[Theorem 3.4 of \cite{donovan2025latin}]
\label{a^n}
    A $\LC(a^k)$ exists for all $a,k\in\mathbb{Z}^+$.
\end{theorem}

\begin{theorem}[Theorem 4.4 of \cite{donovan2025latin}]
\label{ab^n-1}
    For $k\geq 3$, a $\LC(ab^{k-1})$ exists if and only if $a\leq (k-1)b$.
\end{theorem}

These results do not cover partitions of the form $(a^ub^{k-u})$ for $u\geq 2$, though some partial results are provided. The following is a strong result which provides a method for constructing many 3-realizations.

\begin{theorem}[Theorem 3.1 of \cite{donovan2025latin}]
\label{thm: CubefromSquare}
    If there exists a $\LS(h_1\dots h_k)$, then there exists a $\LC(h_1\dots h_k)$.
\end{theorem}

Using \Cref{thm: CubefromSquare} with \Cref{thm: squaresatmost2}, it follows directly that a $\LC(a^ub^{k-u})$ always exists when $u\geq 3$ for all $k\geq 4$. Thus, the only case to be settled is $u=2$. By adding in \Cref{a^2b^2}, existence is obvious when $u=2$ and $a\leq (k-2)b$.

\begin{lemma}[Lemma 4.6 of \cite{donovan2025latin}]
\label{a^2b^2}
    A $\LC(a^2b^2)$ exists for all $a\leq 2b$.
\end{lemma}

\begin{corollary}
\label{lemma: a<=(k-2)b}
    For $k\geq 4$, a $\LC(a^2b^{k-2})$ exists for all $a\leq (k-2)b$.
\end{corollary}

Therefore, we need only consider $u=2$ with $a>(k-2)b$. The following results provide an upper bound on the size of $a$ and reduce the problem to a single case: $\LC(a^2b^1)$

\begin{lemma}[Lemma 4.7 of \cite{donovan2025latin}]
\label{a2b(n-2) cond}
    If a $\LC(a^2b^{k-2})$ exists for $k\geq 3$, then $a\leq 2(k-2)b$.
\end{lemma}

\begin{lemma}[Lemma 4.8 of \cite{donovan2025latin}]
\label{lemma: greater k from a^2b^1}
    If a $\LC(a^2b^1)$ exists for all $\frac{a}{2}\leq c\leq b< a$ then a $\LC(a^2d^{k-2})$ exists for all $d$ with $c\leq (k-2)d< a$, where $k\geq 4$.
\end{lemma}

Thus, by proving existence of a $\LC(a^2b^1)$ for all $\frac{a}{2}\leq b<a$, the existence of 3-realizations for partitions with at most two distinct parts is completely solved. In this paper we use two different constructions to prove such a result with the exception of $a\equiv 3\pmod 6$.

By combining the above results with \Cref{thm: even a final,thm: odd a final}, we obtain our main result.

\begin{theorem}
\label{thm: main result}
    For all $k\geq 3$, a $\LC(a^ub^{k-u})$ exists if and only if
    \begin{itemize}
        \item $u\geq 3$,
        \item $u=1$ and $a\leq (k-1)b$, or
        \item $u=2$ and $a\leq 2(k-2)b$,
    \end{itemize}
    except possibly when $u=2$ with $a\equiv 3\pmod6$ and $\frac{a}{2}<(k-2)b< \frac{2}{3}a$.
\end{theorem}

\section{Inflation construction}

The first construction involves a modified version of a method known as inflation. Given a latin cube $C$ of order $n$ and a latin cube $T$ of order $t$, an \emph{inflation} of $C$ is a latin cube of order $tn$ made by replacing each entry $C(i,j,\ell)$ of $C$ with a copy of $T$ on the symbols of $t(C(i,j,\ell)-1)+[t]$.

\begin{lemma}[\cite{donovan2025latin}]
    If a $\LC(h_1\dots h_k)$ exists then a $\LC((th_1)\dots (th_k))$ exists for any integer $t\geq 1$.
\end{lemma}

As part of the construction, we need latin cubes with some specific properties to use in the inflation. The following results, from Donovan et.~al \cite{donovan2025latin} and Evans \cite{evans1960embedding} respectively, are combined to give an analogue of \Cref{thm: single subsquare} for latin cubes.

\begin{construction}[Construction 2.1 of \cite{donovan2025latin}]
\label{constr: cube from square construction}
    Given any order $n$ latin square $L$, with entries determined by $L(r,c)$, we can construct an order $n$ latin cube $C$ as $C(r,c,\ell) = L(L(r,\ell),c)$ for all $r,c,\ell\in[n]$.
\end{construction}

\begin{theorem}[\cite{evans1960embedding}]
\label{thm: single subsquare}
    There exists a latin square of order $n$ with a subsquare of order $m<n$ if and only if $m\leq \frac{n}{2}$.
\end{theorem}

\begin{lemma}
\label{lemma: cube with subcube}
    For any integers $t,c\geq 0$, there exists a latin cube of order $t+c$ with a subcube of order $c$ if and only if $t\geq c$.
\end{lemma}
\begin{proof}
    Assume that the subcube of order $c$ is located in the cells $(i,j,\ell)$ for $i,j,\ell\in[c]$. The layer $([t+c],[t+c],1)$ is by definition a latin square $L$ and the cells $(i,j)$ for all $i,j\in[c]$ contain only the symbols of $[c]$. Thus, $L$ is a latin square of order $t+c$ with a subsquare of order $c$. Therefore, $t\geq c$ by \Cref{thm: single subsquare}.
    
    Take a latin square of order $t+c$ with a subsquare of order $c$, which always exists by \Cref{thm: single subsquare}. Permute the rows, columns and symbols so that the subsquare occurs in the first $c$ rows and columns and contains only the symbols of $[c]$. Construct a latin cube $C$ as in \Cref{constr: cube from square construction}.

    In the latin square we have that $L(i,j)\in[c]$ for all $i,j\in[c]$. Thus, $C(i,j,\ell) = L(g,j)$, where $g = L(i,\ell)$, and so for all $i,j,\ell\in[c]$, $g\in[c]$ and also $C(i,j,\ell)\in[c]$.
\end{proof}

As established earlier, latin squares and cubes are specific cases of orthogonal arrays. Thus, the following, from Ji and Yin \cite{ji2010constructions}, is useful for constructing latin cubes.

\begin{theorem}[Theorem 2.5 of \cite{ji2010constructions}]
\label{thm: orthogonal arrays}
    For all integers $n\geq 4$, if $n\not\equiv 2\pmod 4$, then an $\OA(3,5,n)$ exists.
\end{theorem}

We introduce a new definition for a partial latin cube with certain useful properties and then show a method for constructing these objects.

\begin{definition}
    A partial latin cube $A_2$ of order $t+c$ is an \emph{extension by $c$} of a latin cube $A_1$ of order $t$ if
    \begin{itemize}
        \item all cells $(i,j,\ell)$ of $A_2$ where $i,j,\ell\in [t]$ are filled with symbols of $[t+c]$,
        \item all cells $(i,j,\ell)$ of $A_2$ where exactly one of $i$, $j$ or $\ell$ is in $t+[c]$, are filled with symbols from $[t]$,
        \item if $i,j,\ell\in [t]$ and cell $(i,j,\ell)$ of $A_2$ has symbol $s\in [t]$, then $A_1(i,j,\ell) = s$ also, and
        \item all cells $(i,j,\ell)$ with at least two of $i$, $j$ and $\ell$ in $t+[c]$ are empty.
    \end{itemize}
\end{definition}

\begin{lemma}
\label{lemma: orthogonal is extension}
    If an orthogonal array $\OA(3,5,t)$ exists, then for any $c\leq t$ there exists a latin cube $A_1$ of order $t$ and a partial latin cube $A_2$ of order $t+c$ such that $A_2$ is an extension by $c$ of $A_1$.
\end{lemma}
\begin{proof}
    Let $O$ be an $\OA(3,5,t)$. Take $A_1$ to be the latin cube of order $t$ which is formed by taking $A_1(o_1,o_2,o_3) = o_4$ for all columns $(o_1,o_2,o_3,o_4,o_5)^T$ of $O$.

    Given a column $(o_1,o_2,o_3,o_4,o_5)^T$ of $O$, fill $A_2$ as follows:
    $$\begin{cases}
        A_2(o_1,o_2,o_3) = o_4 & \text{if $o_5 > c$},\\
        A_2(o_1,o_2,o_3) = t+o_5 & \text{if $o_5 \leq c$},\\
        A_2(t+o_5,o_2,o_3) = o_4 & \text{if $o_5 \leq c$},\\
        A_2(o_1,t+o_5,o_3) = o_4 & \text{if $o_5 \leq c$},\\
        A_2(o_1,o_2,t+o_5) = o_4 & \text{if $o_5 \leq c$}.
    \end{cases}$$

    It is clear that all cells $(i,j,\ell)$ with $i,j,\ell\in[t]$ are filled and if $A_2(i,j,\ell)\in[t]$ then $A_2(i,j,\ell) = A_1(i,j,\ell)$. Also, it is easily seen that any cell $(i,j,\ell)$ with at least two of $i$, $j$ and $\ell$ in $t+[c]$ must be empty.

    For all $i,j\in[t]$ and $\ell\in[c]$, the triple $(i,j,\ell)$ must occur within a column of the orthogonal array in rows 1, 2, and 5. Therefore, cell $(i,j,t+\ell)$ is filled for all such values and it is obvious that it is filled with a value from $[t]$. Similarly, cells $(t+\ell,i,j)$ and $(i,t+\ell,j)$ are filled as required. Since $(i,j,\ell)$ must also occur exactly once as a column across rows 1, 4 and 5, there exists exactly one cell in the line $(i,t+\ell,[t+c])$ with $j$ as the entry. The same is true for the lines $(t+\ell,i,\cdot)$, $(i,\cdot,t+\ell)$, $(t+\ell,\cdot,i)$, $(\cdot,i,t+\ell)$ and $(\cdot,t+\ell,i)$.

    Therefore, $A_2$ is a extension by $c$ of $A_1$.
\end{proof}

Transversals in latin squares are well-studied and useful for constructing 2-realizations. A similar object has been defined for latin cubes and this is important for our construction.

\begin{definition}
    In a latin cube $C$ of order $n$, a \emph{transversal} $T$ is a set of $n$ cells $\{(i_{11},i_{12},i_{13}), (i_{21},i_{22},i_{23}), \dots, (i_{n1},i_{n2},i_{n3})\}$ with distinct symbols such that for any fixed $k\in[3]$ the values $i_{jk}$ for $j\in[n]$ are distinct.

    A \emph{partial transversal} of size $m$ is a set of $m< n$ cells satisfying the same conditions.
\end{definition}

Our main construction relies on the existence of two 3-realizations with a shared partial transversal that has specific properties.

\begin{definition}
    Take integers $h_k>0$ and $h_1\geq h_2\geq \dots\geq h_{k-1}$, and for each $i\in[k]$ let $H_i = \sum_{j=1}^{i-1}h_j + [h_i]$ with $H_k' = \sum_{j=1}^{k-1}h_j + [h_k+1]$. Given a $\LC(h_1\dots h_k)$, denoted $L_1$, and a $\LC(h_1\dots h_{k-1}(h_k+1))$, denoted $L_2$ and of order $n$, both in normal form, we say that $L_1$ and $L_2$ are \emph{paired} if there exists a partial transversal $T$ of size $m = \sum_{i=1}^{k-1}h_i$ such that
    \begin{itemize}
        \item $i_{j\ell}\notin H'_k$ for all $\ell\in[3]$ and all $(i_{j1},i_{j2},i_{j3})\in T$,
        \item $L_1(i_{j1},i_{j2},i_{j3}) = L_2(i_{j1},i_{j2},i_{j3})\notin H'_k$ for all $j\in[m]$,
        \item for all $a\in[k]$ and all $j\in[m]$, there exists $\ell\in[3]$ such that $i_{j\ell}\notin H_a$, and
        \item for $j\in[m]$, $L_2(i_{j1},i_{j2},i_{j3}) = L_2(n,n,i_{j3}) = L_2(n,i_{j2},n) = L_2(i_{j1},n,n)$ and $L_2(n,i_{j2},i_{j3}) = L_2(i_{j1},n,i_{j3}) = L_2(i_{j1},i_{j2},n) = n$.
    \end{itemize}
\end{definition}

\begin{example}
    The latin cubes in \Cref{fig:2-2-1 & 2-2-2} are two paired realizations. The shared transversal (and corresponding forced cells) are highlighted.
\end{example}

\begin{figure}[h]
    \centering
$\arraycolsep=3pt\begin{array}{|ccccc|} \hline
1 & \multicolumn{1}{c|}{2} & 3 & 4 & 5\\
2 & \multicolumn{1}{c|}{1} & 4 & 5 & 3\\ \cline{1-2}
5 & 3 & 1 & \cellcolor{lightgray}2 & 4\\
3 & 4 & 5 & 1 & 2\\
4 & 5 & 2 & 3 & 1\\
\hline \end{array}$\quad
$\arraycolsep=3pt\begin{array}{|ccccc|} \hline
2 & \multicolumn{1}{c|}{1} & 5 & 3 & 4\\
1 & \multicolumn{1}{c|}{2} & 3 & 4 & 5\\ \cline{1-2}
3 & 4 & 2 & 5 & 1\\
4 & 5 & \cellcolor{lightgray}1 & 2 & 3\\
5 & 3 & 4 & 1 & 2\\
\hline \end{array}$\quad
$\arraycolsep=3pt\begin{array}{|ccccc|} \hline
5 & 4 & 2 & 1 & 3\\
4 & \cellcolor{lightgray}3 & 5 & 2 & 1\\ \cline{3-4}
1 & 5 & \multicolumn{1}{|c}{3} & \multicolumn{1}{c|}{4} & 2\\
2 & 1 & \multicolumn{1}{|c}{4} & \multicolumn{1}{c|}{3} & 5\\ \cline{3-4}
3 & 2 & 1 & 5 & 4\\
\hline \end{array}$\quad
$\arraycolsep=3pt\begin{array}{|ccccc|} \hline
\cellcolor{lightgray}4 & 3 & 1 & 5 & 2\\
3 & 5 & 2 & 1 & 4\\ \cline{3-4}
2 & 1 & \multicolumn{1}{|c}{4} & \multicolumn{1}{c|}{3} & 5\\
5 & 2 & \multicolumn{1}{|c}{3} & \multicolumn{1}{c|}{4} & 1\\ \cline{3-4}
1 & 4 & 5 & 2 & 3\\
\hline \end{array}$\quad
$\arraycolsep=3pt\begin{array}{|ccccc|} \hline
3 & 5 & 4 & 2 & 1\\
5 & 4 & 1 & 3 & 2\\
4 & 2 & 5 & 1 & 3\\
1 & 3 & 2 & 5 & 4\\ \cline{5-5}
2 & 1 & 3 & 4 & \multicolumn{1}{|c|}{5}\\
\hline \end{array}$\quad

\vspace{0.25cm}

$\arraycolsep=3pt\begin{array}{|cccccc|} \hline
1 & \multicolumn{1}{c|}{2} & 3 & 4 & 6 & 5\\
2 & \multicolumn{1}{c|}{1} & 6 & 5 & 3 & 4\\ \cline{1-2}
5 & 3 & 1 & \cellcolor{lightgray}2 & 4 & \cellcolor{lightgray!50}6\\
6 & 4 & 2 & 1 & 5 & 3\\
4 & 6 & 5 & 3 & 2 & 1\\
3 & 5 & 4 & \cellcolor{lightgray!50}6 & 1 & \cellcolor{lightgray!50}2\\
\hline \end{array}$\quad
$\arraycolsep=3pt\begin{array}{|cccccc|} \hline
2 & \multicolumn{1}{c|}{1} & 5 & 6 & 4 & 3\\
1 & \multicolumn{1}{c|}{2} & 3 & 4 & 6 & 5\\ \cline{1-2}
3 & 6 & 2 & 1 & 5 & 4\\
4 & 5 & \cellcolor{lightgray}1 & 2 & 3 & \cellcolor{lightgray!50}6\\
6 & 3 & 4 & 5 & 1 & 2\\
5 & 4 & \cellcolor{lightgray!50}6 & 3 & 2 & \cellcolor{lightgray!50}1\\
\hline \end{array}$\quad
$\arraycolsep=3pt\begin{array}{|cccccc|} \hline
3 & 4 & 6 & 1 & 5 & 2\\
4 & \cellcolor{lightgray}3 & 5 & 2 & 1 & \cellcolor{lightgray!50}6\\ \cline{3-4}
6 & 5 & \multicolumn{1}{|c}{3} & \multicolumn{1}{c|}{4} & 2 & 1\\
2 & 1 & \multicolumn{1}{|c}{4} & \multicolumn{1}{c|}{3} & 6 & 5\\ \cline{3-4}
5 & 2 & 1 & 6 & 3 & 4\\
1 & \cellcolor{lightgray!50}6 & 2 & 5 & 4 & \cellcolor{lightgray!50}3\\
\hline \end{array}$\quad
$\arraycolsep=3pt\begin{array}{|cccccc|} \hline
\cellcolor{lightgray}4 & 3 & 1 & 5 & 2 & \cellcolor{lightgray!50}6\\
3 & 4 & 2 & 6 & 5 & 1\\ \cline{3-4}
2 & 1 & \multicolumn{1}{|c}{4} & \multicolumn{1}{c|}{3} & 6 & 5\\
5 & 6 & \multicolumn{1}{|c}{3} & \multicolumn{1}{c|}{4} & 1 & 2\\ \cline{3-4}
1 & 5 & 6 & 2 & 4 & 3\\
\cellcolor{lightgray!50}6 & 2 & 5 & 1 & 3 & \cellcolor{lightgray!50}4\\
\hline \end{array}$\quad
$\arraycolsep=3pt\begin{array}{|cccccc|} \hline
5 & 6 & 4 & 2 & 3 & 1\\
6 & 5 & 1 & 3 & 4 & 2\\
4 & 2 & 6 & 5 & 1 & 3\\
1 & 3 & 5 & 6 & 2 & 4\\ \cline{5-6}
3 & 4 & 2 & 1 & \multicolumn{1}{|c}{5} & 6\\
2 & 1 & 3 & 4 & \multicolumn{1}{|c}{6} & 5\\
\hline \end{array}$\quad
$\arraycolsep=3pt\begin{array}{|cccccc|} \hline
\cellcolor{lightgray!50}6 & 5 & 2 & 3 & 1 & \cellcolor{lightgray!50}4\\
5 & \cellcolor{lightgray!50}6 & 4 & 1 & 2 & \cellcolor{lightgray!50}3\\
1 & 4 & 5 & \cellcolor{lightgray!50}6 & 3 & \cellcolor{lightgray!50}2\\
3 & 2 & \cellcolor{lightgray!50}6 & 5 & 4 & \cellcolor{lightgray!50}1\\ \cline{5-6}
2 & 1 & 3 & 4 & \multicolumn{1}{|c}{6} & 5\\
\cellcolor{lightgray!50}4 & \cellcolor{lightgray!50}3 & \cellcolor{lightgray!50}1 & \cellcolor{lightgray!50}2 & \multicolumn{1}{|c}{5} & 6\\
\hline \end{array}$
    \caption{Two paired realizations: a $\LC(2^21^1)$ and $\LC(2^3)$}
    \label{fig:2-2-1 & 2-2-2}
\end{figure}

Using all of these tools, we give a construction for 3-realizations that are almost inflations. Given two 3-realizations with similar subcubes, we inflate both realizations and combine the entries to create a new latin cube.

\begin{lemma}
\label{lemma: paired cube construction}
    A $\LC((h_1t)\dots(h_{k-1}t)(h_kt+c))$ exists for all $0\leq c\leq t$ if the following exist:
    \begin{itemize}
        \item two paired realizations, $L_1$ and $L_2$, which are a $\LC(h_1\dots h_k)$ and $\LC(h_1\dots h_{k-1}(h_k+1))$ respectively, with partial transversal $T$, and
        \item a partial latin cube $A_2$ which is an extension by $c$ of a latin cube $A_1$ of order $t$.
    \end{itemize}
\end{lemma}
\begin{proof}
    We construct a latin cube $C$ of order $c+t\sum_{i=1}^kh_i$ by filling subarrays of $C$. Let $n = 1+\sum_{i=1}^kh_i$ and let the subarrays be denoted by $C_{i_1,i_2,i_3}$ for $i_\ell\in[n]$, where the subarray $C_{i_1,i_2,i_3}$ contains the cells $(j_1,j_2,j_3)$ of $C$ for all $j_\ell\in\begin{cases}
        t(i_\ell-1)+[t], & \text{if $i_\ell< n$,}\\
        t(n-1)+[c], & \text{if $i_\ell=n$.}
    \end{cases} $

    We assume that both $L_1$ and $L_2$ are realizations in normal form, and let $H'_k = \sum_{i=1}^{k-1}h_i+[h_k+1]$. The subarrays are filled in four cases.

    \vspace{0.25cm}
    \noindent\textbf{Case 1:} $C_{i,j,\ell}$ where $i,j,\ell\in[n-1]$ and $(i,j,\ell)\notin T$.
    
    All arrays in this case are $t\times t\times t$ arrays and are filled as follows. For all $u,v,w\in[t]$, 
    $$C_{i,j,\ell}(u,v,w) = \begin{cases}
        t(L_1(i,j,\ell)-1)+A_1(u,v,w), & \text{if $A_2(u,v,w)\leq t$,}\\
        t(L_2(i,j,\ell)-1)+A_1(u,v,w), & \text{if $A_2(u,v,w)> t$ and $L_2(i,j,\ell)\neq n$,}\\
        t(n-2)+A_2(u,v,w), & \text{if $A_2(u,v,w)> t$ and $L_2(i,j,\ell)=n$.}
    \end{cases}$$

    \vspace{0.25cm}
    \noindent\textbf{Case 2:} $C_{i,j,\ell}$ where exactly one of $i$, $j$ or $\ell$ is $n$ and $L_2(i,j,\ell)\neq n$.

    For $j,\ell\in[n-1]$, the subarray $C_{n,j,\ell}$ is $c\times t\times t$, so for $u\in[c]$, $v,w\in[t]$, take
    $$C_{n,j,\ell}(u,v,w) = t(L_2(n,j,\ell)-1)+A_2(t+u,v,w).$$ 

    Similarly, for $C_{i,n,\ell}$ where $i,\ell\in[n-1]$, for all $v\in[c]$, $u,w\in[t]$, take
    $$C_{i,n,\ell}(u,v,w) = t(L_2(i,n,\ell)-1)+A_2(u,t+v,w).$$

    And finally to fill $C_{i,j,n}$ where $i,j\in[n-1]$, for $w\in[c]$, $u,v\in[t]$, let the entries of the subarray be
    $$C_{i,j,n}(u,v,w) = t(L_2(i,j,n)-1)+A_2(u,v,t+w).$$

    \vspace{0.25cm}
    \noindent\textbf{Case 3:} For $(i,j,\ell)\in T$, $C_{i,j,\ell}$, $C_{n,j,\ell}$, $C_{i,n,\ell}$, $C_{i,j,n}$, $C_{n,n,\ell}$, $C_{n,j,n}$, $C_{i,n,n}$ and $C_{n,n,n}$.

    Let $B$ be a latin cube of order $t+c$ with a subcube of order $c$, which exists by \Cref{lemma: cube with subcube}. We assume that the subcube of order $c$ in $B$ is in the cells $(i,j,\ell)$ for $i,j,\ell\in t+[c]$ with the symbols of $t+[c]$ (permute the rows, columns and symbols if necessary). For each cell $(i,j,\ell)$ in the transversal, observe that the subarrays $C_{i,j,\ell}$, $C_{n,j,\ell}$, $C_{i,n,\ell}$, $C_{i,j,n}$, $C_{n,n,\ell}$, $C_{n,j,n}$ and $C_{i,n,n}$ form a $(t+c)\times(t+c)\times(t+c)$ subarray when combined with $C_{n,n,n}$; label this larger subarray $C'_{i,j,\ell}$. Note that $C_{i,j,\ell}$ is not filled in case 1, and $C_{n,j,\ell}$, $C_{i,n,\ell}$ and $C_{i,j,n}$ are not filled in case 2 since the corresponding value of $L_2$ is $n$ by the requirements of the paired realizations, and hence no part of $C'_{i,j,\ell}$ is filled in the previous cases. For all $u,v,w\in[t+c]$, let
    $$C'_{i,j,\ell}(u,v,w) = \begin{cases}
        t(L_1(i,j,\ell)-1) + B(u,v,w), & \text{if $B(u,v,w)\leq t$,}\\
        t(n-2) + B(u,v,w), & \text{otherwise.}
    \end{cases}$$

    By the properties of a transversal, these subarrays $C'_{i,j,\ell}$ intersect only in $C_{n,n,n}$ which is filled by the subcube of $B$ in all cases with the same symbol set.

    \vspace{0.25cm}
    \noindent\textbf{Case 4:} $C_{i,j,\ell}$ for $i,j,\ell\in H'_k$.

    The subarrays $C_{i,j,\ell}$ where $i,j,k\in H'_k\setminus\{n\}$ were filled in case 1. Since $L_1(i,j,\ell)\in H'_k\setminus\{n\}$ and $L_2(i,j,\ell)\in H'_k$, all entries in those subarrays are symbols from $t\sum_{m=1}^{k-1}h_m+[th_k+c]$. The only other subarray in this case which already has symbols is $C_{n,n,n}$. This subarray has symbols from $t(n-1)+[c]\subseteq t\sum_{m=1}^{k-1}h_m+[th_k+c]$.

    Thus the union of all of the subarrays forms a $(th_k+c)\times(th_k+c)\times(th_k+c)$ array, partially filled with symbols from $t\sum_{m=1}^{k-1}h_m+[th_k+c]$. We replace these symbols and fill the array with an arbitrary latin cube of order $th_k+c$ on the symbols of $t\sum_{m=1}^{k-1}h_m+[th_k+c]$.

    \vspace{0.25cm}

    We now check that $C$ is a latin cube. Observe that each subarray is either a subset of a latin cube or a combination of two latin cubes on disjoint symbol sets. Thus, we need only check instances between separate subarrays. 
    
    For the subarrays not involved in case 4, note that if symbol $s\in tg+[t]$ occurs in subarray $C_{i,j,\ell}$ then either
    \begin{itemize}
        \item $L_1(i,j,\ell) = g+1$,
        \item $L_2(i,j,\ell) = g+1$,
        \item or $C_{i,j,\ell}$ is part of $C'_{u,v,w}$ for some $(u,v,w)\in T$, and $L_1(u,v,w) = g+1$ or $s\in t(n-1)+[c]$.
    \end{itemize}

    The subarrays in case 4 contain only symbols of $t\sum_{m=1}^{k-1}h_m+[th_k+c]$ and each symbol clearly occurs exactly once in each line. Suppose that for some $i,j\in H'_k$ and $\ell\notin H'_k$, the subarray $C_{i,j,\ell}$ contains a symbol from $t\sum_{m=1}^{k-1}h_m+[th_k+c]$. Note that $C_{i,j,\ell}$ is not filled under case 4, and since $L_1$ and $L_2$ are realizations in normal form, $L_1(i,j,\ell)\notin H'_k$ and $L_2(i,j,\ell)\notin H'_k$. Therefore, the symbols of $t\sum_{m=1}^{k-1}h_m+[th_k+c]$ can only occur in $C_{i,j,\ell}$ if it is contained in $C'_{u,v,\ell}$ for some $u,v\in[n-1]$. Considering the subarrays used in case 3, $i\in\{u,n\}$ and $j\in\{v,n\}$. However, $u,v\notin H'_k$ by definition of paired realizations, and so $i=j=n$. Thus, the entries of $C_{n,n,\ell}$ must come from $t(L_1(u,v,\ell)-1)+[t]$ or $t(n-1)+[c]$. Since $L_1(i,j,\ell)\notin H'_k$ for all $(i,j,\ell)\in T$, it follows that there exists an entry in the subarray from $t(n-1)+[c]$. Thus, there exists an entry $C(x,y,z)$ in $C_{n,n,\ell}$, where $x,y\in t+[c]$ and $z\in[t]$, such that $B(x,y,z) \in t+[c]$. This contradicts the construction of $B$, which assumes that $B$ has a subcube of order $c$ such that $B(x,y,z)\in t+[c]$ for all $x,y,z\in t+[c]$. Therefore, for $i,j\in H'_k$, the subarrays $C_{i,j,\ell}$ with $\ell\notin H'_k$ do not contain symbols of $t\sum_{m=1}^{k-1}h_m+[th_k+c]$ and the subarrays with $\ell\in H'_k$ form a subcube on that set of symbols.

    \medskip

    Consider the line $(u,v,[t(n-1)+c])$ for some $u,v\in[t]$ within the subarrays $C_{i,j,\ell}$ for $\ell\in[n]$. Suppose that symbol $s\in tg+[t]$, for some $g$, occurs more than once in the line. Then $s$ occurs in subarrays $C_{i,j,\ell_1}$ and $C_{i,j,\ell_2}$ for some $\ell_1,\ell_2\in[n]$. We show that this leads to a contradiction in all cases. We begin with the case $i,j\in[n-1]$. Within this, we have separate cases to consider the different parts of the construction; case 3, case 1, and a combination of cases 1 and 2.

    \smallskip

    If $C_{i,j,\ell_1}$ is a part of $C'_{i,j,\ell}$ for $(i,j,\ell)\in T$ and $s\leq t(n-1)$, then $L_1(i,j,\ell) = L_2(i,j,\ell) = g+1$. Since $g+1$ cannot occur again in the same line of $L_1$ or $L_2$, the only occurrences of $s$ are in $C_{i,j,\ell}$ and $C_{i,j,n}$. As $C'_{i,j,\ell}$ is a copy of the latin cube $B$, the symbol $s$ cannot occur more than once in a line.
    
    If $\ell_1,\ell_2\in[n-1]$, then we consider only case 1, and since $L_1(i,j,\ell_1)$ and $L_1(i,j,\ell_2)$ cannot both be $g+1$, nor can both $L_2(i,j,\ell_1)$ and $L_2(i,j,\ell_2)$, we assume without loss of generality that $L_1(i,j,\ell_1) = g+1$ and $L_2(i,j,\ell_2) = g+1$. Thus, $g+1\leq n-1$ and so there exists $w_1,w_2\in [t]$ such that $A_1(u,v,w_1) = s-gt = A_1(u,v,w_2)$. But since $A_1$ is a latin cube, $w_1=w_2$, however, this implies that $t\geq A_2(u,v,w_1) = A_2(u,v,w_2)>t$. Therefore, one of $\ell_1$ or $\ell_2$ must be $n$.

    If the subarrays are $C_{i,j,\ell}$ and $C_{i,j,n}$, then since case 3 has already been covered, $C_{i,j,\ell}$ must be filled under case 1 and $C_{i,j,n}$ is under case 2. Thus, $L_2(i,j,n) = g+1 < n$ and so $L_2(i,j,\ell)\neq g+1$ and $L_1(i,j,\ell) = g+1$. Also, there exists $w\in[t]$ and $w'\in[c]$ such that $A_2(u,v,w)\leq t$ and $A_1(u,v,w) = s-tg = A_2(u,v,t+w')$. Since $A_2(u,v,w)\leq t$ and $u,v,w\in[t]$, by definition of an extension we have that $A_1(u,v,w) = A_2(u,v,w) = s-tg$. But this places the same symbol twice in the line $(u,v,[t+c])$ of $A_2$.

    Therefore, if $i,j\in[n-1]$, then no symbol occurs more than once in the line $(u,v,[t(n-1)+c])$ contained within $C_{i,j,\ell}$, for $u,v\in[t]$ and $\ell\in[n]$.

    \medskip

    Instead of $i,j\in[n-1]$, now suppose that $i\in[n-1]$ but $j=n$. We consider two cases dependent on the value of $s$.

    \smallskip
    
    If $s\in[t(n-1)]$ occurs in both $C_{i,n,\ell_1}$ and $C_{i,n,\ell_2}$, then $L_2(i,n,\ell_1)$ and $L_2(i,n,\ell_2)$ are either $g+1$ or $n$. Since $g+1$ cannot occur twice in the same line of $L_2$, suppose that $L_2(i,n,\ell_1) = g+1$ and $L_2(i,n,\ell_2) = n$. For $s$ to appear in $C_{i,n,\ell_2}$, it must be part of $C'_{i,j,\ell_2}$ for some $j\in[n-1]$ with $L_2(i,j,\ell_2) = g+1$ and $(i,j,\ell_2)\in T$. This forces $L_2(i,n,n) = g+1$ and so $\ell_1 = n$. Thus, $C_{i,n,\ell_1}$ is also a part of $C'_{i,j,\ell_2}$. Therefore, both subarrays are subarrays of a latin cube (under case 3), and so $s$ cannot occur multiple times in the same line.

    If $s\in t(n-1)+[c]$ occurs in both $C_{i,n,\ell_1}$ and $C_{i,n,\ell_2}$, then both subarrays must be filled under case 3. There is only one cell in $T$ that is in layer $(i,\cdot,\cdot)$; let this cell be $(i,j',\ell)$. Then $C_{i,n,\ell_1}$ and $C_{i,n,\ell_2}$ are in $C'_{i,j',\ell}$ with $\{\ell_1,\ell_2\} = \{\ell,n\}$. As before, this larger array is a latin cube and so $s$ cannot occur more than once in any line.

    Therefore, if $i\in[n-1]$, then no symbol occurs more than once in the line $(u,v,[t(n-1)+c])$ contained within $C_{i,n,\ell}$, for $u\in[t]$, $v\in[c]$ and $\ell\in[n]$. Lines with $u\in t(n-1)+[c]$ and $v\in t(i-1)+[t]$ are shown similarly.

    \medskip

    Finally, suppose that $i=j=n$. As shown earlier, the symbols of $t\sum_{m=1}^{k-1}h_m+[th_k+c]$ only occur in the subarrays $C_{n,n,\ell}$ where $\ell\in \sum_{m=1}^{k-1}h_m+[h_k+1]$ which are part of the latin cube placed in case 4, and so we only check the symbols $s\in[t\sum_{m=1}^{k-1}h_m]$. Each subarray $C_{n,n,\ell}$ is filled in case 3 and so each array is part of $C'_{i,j,\ell}$ for some $i,j\in[n-1]$. Thus, $L_2(i,j,\ell) = g+1 = L_2(n,n,\ell)$ since the realizations are paired. It follows that $s$ can only appear in a subarray $C_{n,n,\ell}$ where $L_2(n,n,\ell) = g+1$. Since $L_2$ is a latin cube, it is clear that $s$ can only occur in one such subarray and so only once in any line.

    \medskip

    Therefore, for all $i,j\in[n]$, no symbol occurs more than once in the line $(u,v,[t(n-1)+c])$ contained within $C_{i,j,\ell}$, for $u,v\in[t]$ and $\ell\in[n]$. The lines $(u,[t(n-1)+c],v)$ and $([t(n-1)+c],u,v)$ are checked in a similar manner. Thus, $C$ is a latin cube.

    \medskip

    For any $m\in[k]$, let $H_m = \sum_{i=1}^{m-1}h_i + [h_m]$ and $J_m = t\sum_{i=1}^{m-1}h_i + [th_m]$. If $u,v,w\in J_m$ for some $m\in[k-1]$, then $C(u,v,w)$ is a cell in $C_{i,j,\ell}$ for $i,j,\ell\in H_m$ where $i=\lceil\frac{u}{t}\rceil$, $j=\lceil\frac{v}{t}\rceil$ and $\ell=\lceil\frac{w}{t}\rceil$. Since $L_1$ and $L_2$ are realizations in normal form, $L_1(i,j,\ell),L_2(i,j,\ell)\in H_m$ also. Note that $i,j,\ell\in[n-1]$, $L_1(i,j,\ell),L_2(i,j,\ell)\in[n-1]$, and $(i,j,\ell)\notin T$ by definition, so this subarray $C_{i,j,\ell}$ is filled in case 1. Thus, $C(u,v,w)\in t(L_1(i,j,\ell)-1) + [t]$ or $C(u,v,w)\in t(L_2(i,j,\ell)-1) + [t]$. In either case, $t\sum_{i=1}^{m-1}h_i+1\leq C(u,v,w)\leq t\sum_{i=1}^{m}h_i$, and so $C(u,v,w)\in J_m$. Thus, if $u,v,w\in J_m$ for $m\in[k-1]$, then $C(u,v,w)\in J_m$.

    Therefore, since case 4 provides the $k^{th}$ subcube, $C$ is a $\LC((h_1t)\dots(h_{k-1}t)(h_kt+c))$.
\end{proof}

\Cref{lemma: orthogonal is extension} shows the existence of an extension for many orders by using orthogonal arrays, and so this construction gives 3-realizations for many partitions of the form $(a^2b^1)$ with only a few sets of paired 3-realizations (given in \Cref{app: paired 4-4-b} and \Cref{app: paired 3-3-b}).

\begin{theorem}
\label{thm: even a final}
    A $\LC(a^2b^1)$ exists for all even $a$ where $b\geq \frac{a}{2}$ and all $a\equiv 3\pmod6$ where $b\geq \frac{2}{3}a$.
\end{theorem}
\begin{proof}
    For $a\equiv 0,2,6\pmod8$, $a=2t$ where $t\equiv0,1,3\pmod 4$. By \Cref{thm: orthogonal arrays}, an orthogonal array $\OA(3,5,t)$ exists when $t\geq 4$, and by \Cref{lemma: orthogonal is extension} there is an extension by $c$ for any $c\leq t$. The realizations $\LC(2^21^1)$ and $\LC(2^3)$ in \Cref{fig:2-2-1 & 2-2-2} are paired. Thus, by \Cref{lemma: paired cube construction}, a $\LC((2t)^2(t+c)^1)$ exists for all $0\leq c\leq t$ when $a=2t\geq 8$.

    If $a\equiv 4\pmod 8$, then $a=4t$ for some odd $t$. Thus, an orthogonal array $\OA(3,5,t)$ exists for $t\geq 4$. The figures in \Cref{app: paired 4-4-b} show two realizations which are paired; a $\LC(4^22^1)$ and $\LC(4^23^1)$. Using \Cref{lemma: paired cube construction}, a $\LC((4t)^2b^1)$ exists for all $2t\leq b\leq 3t$ with $a=4t\geq 20$.

    A $\LC(4^3)$ is given in \Cref{app: paired 4-4-b}, which is paired with the previous $\LC(4^23^1)$. Thus, for all $3t\leq b\leq 4t$, a $\LC((4t)^2b^1)$ exists for $a=4t\geq 20$.

    It is left to prove existence in the cases where $a\in\{2,4,6,12\}$. A $\LC(2^21^1)$ and $\LC(2^3)$ are given in \Cref{fig:2-2-1 & 2-2-2}. A $\LC(4^22^1)$, $\LC(4^23^1)$ and $\LC(4^3)$ are given in \Cref{app: paired 4-4-b}. A $\LC(6^23^1)$, $\LC(6^24^1)$ and $\LC(6^3)$ are found by inflating realizations of $(2^21^1)$, $(3^22^1)$ and $(2^3)$ respectively, with a $\LC(3^22^1)$ given in \Cref{app: paired 3-3-b}, and a $\LC(6^25^1)$ is given in \Cref{app: 6-6-5}.
    
    The realizations of $(12^26^1)$, $(12^28^1)$, $(12^29^1)$, $(12^210^1)$ and $(12^3)$ are inflations of $(2^21^1)$, $(3^22^1)$, $(4^23^1)$, $(6^25^1)$ and $(2^3)$ respectively. A $\LC(12^27^1)$ is given in \Cref{app: 12-12-7}.

    For $a=3t$ and $b\geq 2t$, if $t\geq 5$ then we need only consider odd $t$ and so an $\OA(3,5,t)$ exists. The realizations in \Cref{app: paired 3-3-b} are a paired $\LC(3^22^1)$ and $\LC(3^3)$, so by \Cref{lemma: paired cube construction} there exists a $\LC((3t)^2(2t+c)^1)$ for all $0\leq c\leq t$ and $a=3t\geq 15$. Using an $\OA(3,5,4)$ and the same realizations, we construct a $\LC(12^211^1)$. A $\LC(9^26^1)$ and $\LC(9^3)$ are inflations of realizations of $(3^22^1)$ and $(3^3)$, and a $\LC(9^27^1)$ and a $\LC(9^28^1)$ are given in \Cref{app: 9-9-7} and \Cref{app: 9-9-8}
\end{proof}

\section{Construction from orthogonal arrays}

The construction in this section pieces together latin cubes that have been extended or permuted and orthogonal arrays are important in creating those pieces. We first define a more intricate form of extension than that used in the previous construction.

\begin{definition}
\label{def: extension by B with S}
    Take integers $a$ and $b$ with $a\geq b$. Let $A$ be a latin cube of order $a$ and let $B$ be a partial latin cube of order $a$ with each symbol of $[b]$ occurring $a^2$ times (and no restriction on the other symbols) and such that the elements of $\{(i,j,k,A(i,j,k),B(i,j,k))^T\mid i,j,k\in[a],B(i,j,k)\text{ is non-empty}\}$ form the columns of a partial $\OA(3,5,a)$.
    Take $S\subseteq [a]$ and let $S_k = \{s+k-1\pmod{a}\mid s\in S\}$ for all $k\in[a]$. Then the \emph{extension of $A$ by $B$ with $S$} is the partial latin cube $A_S$ of order $a+b$ where for $i,j,k\in[a]$,
    $$A_S(i,j,k) = \begin{cases}
        a+B(i,j,k), & \text{if $B(i,j,k)\in[b]$ and $A(i,j,k)\in S_k$,}\\
        A(i,j,k), & \text{otherwise,}
    \end{cases}$$
    and for all $m\in[b]$,
    $$A_S(a+m,j,k) = \begin{cases}
        A(i,j,k), & \text{if $B(i,j,k)=m$ and $A(i,j,k)\in S_k$ for some $i\in[a]$,}
    \end{cases}$$
    $$A_S(i,a+m,k) = \begin{cases}
        A(i,j,k), & \text{if $B(i,j,k)=m$ and $A(i,j,k)\in S_k$ for some $j\in[a]$,}
    \end{cases}$$
    $$A_S(i,j,a+m) = \begin{cases}
        A(i,j,k), & \text{if $B(i,j,k)=m$ and $A(i,j,k)\in S_k$ for some $k\in[a]$.}
    \end{cases}$$

    All other cells are left empty.
\end{definition}

The construction is split into two subsections. The first subsection gives a construction for a $\LC(a^2b^1)$ which relies on the existence of a collection of (partial) latin cubes. The other subsection gives constructions for these substructures in order to complete the cases not considered in the previous section: $a\equiv 1,5\pmod6$.

\subsection{Realization construction}

We begin with some results on completing partial latin cubes. The first relies on a well-known theorem of Ryser \cite{ryser1951combinatorial}, and we provide some background on graph colouring before the second result.

\begin{theorem}[Theorem 2 of \cite{ryser1951combinatorial}]
\label{thm: Ryser}
    Let $L$ be an $r\times s$ latin rectangle with symbols from $[n]$, and let $N(i)$ denote the number of time that symbol $i\in[n]$ occurs in $L$. Then $L$ can be extended to an $n\times n$ latin square if and only if $N(i)\geq r+s-n$ for all $i\in[n]$.
\end{theorem}

\begin{lemma}
\label{lemma: right corner channel}
    Let $C$ be an $n\times n\times r$ partial latin box with only the cells $(i,j,k)$ empty for all $i,j\in n-s+[s]$ and $k\in[r]$ where $s\leq n$, and let $V_{k}$ be the set of symbols in the line $(n-s+1,\cdot,k)$. Then $C$ can be completed to an $n\times n\times r$ latin box if 
    \begin{itemize}
    \item for any fixed $k\in[r]$ the set of symbols occurring in the line $(i,\cdot,k)$ is $V_k$ for all choices of $i\in n-s+[s]$ and the set of symbols occurring in the line $(\cdot,j,k)$ is also $V_k$ for all choices of $j\in n-s+[s]$, and
    \item each symbol $\ell\in[n]$ occurs at least $r-s$ times in layer $(n-s+1,\cdot,\cdot)$ of $C$.
    \end{itemize}
\end{lemma}
\begin{proof}
    If a $k$ is fixed, then the empty cells in the layer $(\cdot,\cdot,k)$ can all take any entry from $[n]\setminus V_k$. Thus, the empty section forms a subsquare on $[n]\setminus V_k$, and if the first row is formed, then the remaining rows are constructed by shifting all symbols in the previous row one cell to the right (with the last entry looping around). Thus, if the layer $(n-s+1,\cdot,\cdot)$ is completed, then the entire latin box is completed.

    The layer $(n-s+1,\cdot,\cdot)$ is an $r\times n$ latin rectangle with only the first $r\times (n-s)$ array filled. By Ryser's \Cref{thm: Ryser} and our stated assumption, this latin rectangle can be extended to an order $n$ latin square. By taking the first $r$ rows of this latin square, we complete the layer and thus the latin box.
\end{proof}

Let $G$ be a graph with vertex set $V$ and edge set $E$, where $E$ may contain repeated edges but no loops.

An \emph{edge-colouring} of $G$ with $k$ colours is a partition of $E$ into mutually disjoint sets $C_1,\dots, C_k$, where edge $e\in E$ has colour $i$ if $e\in C_i$. For $v\in V$ and $i\in[k]$, let $c_i(v)$ be the number of edges adjacent to $v$ in $C_i$. An edge-colouring of $G$ is \emph{equitable} if for all $v\in V$ and all $i,j\in[k]$,
$$|c_i(v)-c_j(v)|\leq 1.$$

The following theorem is from De Werra \cite{de1971balanced}.

\begin{theorem}[Theorem 1 of \cite{de1971balanced}]
\label{thm: de werra}
    For each $k\geq 1$, any finite bipartite graph has an equitable edge-colouring with $k$ colours.
\end{theorem}

\begin{lemma}
\label{lemma: filling back entries}
    Let $C$ be an $r\times r\times n$ partial latin box with symbols from $[n]$ and take $U \subseteq [r]\times[r]$ such that for any $i\in[r]$, $|\{j\mid (i,j)\in U\}| = |\{j\mid (j,i)\in U\}| = n-r$. Let $V$ be a subset of $[n]$ such that $|V| = n-r$. If for all $(i,j)\in U$ the cells $(i,j,k)$ contain symbols from $[n]\setminus V$ when $k\leq r$ and are empty when $k> r$, and all other cells are filled, with no cell of $([r],[r],r+[n-r])$ containing symbols from $V$, then $C$ can be completed to an $r\times r\times n$ latin box.
\end{lemma}
\begin{proof}
    The only cells to fill are $(i,j,k)$ where $(i,j)\in U$ and $k\in r+[n-r]$, and these cells need symbols from $V$. Further, the symbols from $V$ do not appear in any row or column of the layer $(\cdot,\cdot,k)$ for all $k\in r+[n-r]$. If the symbols in $V$ are ordered and layer $(\cdot,\cdot,r+1)$ is filled, then each line $(i,j,\cdot)$ for $(i,j)\in U$ is completed by placing the symbol that is the next in the ordering after the symbol in the layer before. Thus, we must only fill the layer $(\cdot,\cdot,r+1)$.

    Construct a bipartite graph $G$ between $[r]$ and $[r]$, representing the rows and columns of the layer, where there is an edge between $i$ and $j$ when $(i,j)\in U$. Thus, each vertex has degree $n-r$ and so there exists an equitable $(n-r)$-colouring of the edges. Assign a symbol of $V$ to each colour and thus fill the layer by placing symbol $\ell$ in cell $(i,j)$ if the edge between $i$ and $j$ is coloured with the colour that $\ell$ is assigned to. Since each vertex has degree $n-r$ and the colouring is balanced, there will be one edge of each colour adjacent to any vertex, and so each row and column of the layer will have one copy of each symbol in $V$. Therefore, the layer, and thus the latin rectangle, are complete.
\end{proof}

The last object we need for our construction is a permutation of an extension with some additional conditions on where the filled cells are.

\begin{definition}
    For some $A$, $B$ and $S$ as defined in \Cref{def: extension by B with S}, given the extension $A_S$ of order $a+b$, a partial latin cube $A'_S$ is a \emph{shifted partner} of $A_S$ when
    \begin{itemize}
        \item $A'_S$ is a copy of $A_S$ with some permutation applied to the layers $(i,[n],[n])$ and (possibly a different permutation) to $([n],[n],i)$ for $i\in[a]$,
        \item for all $i,k\in[a]$ and $m\in a+[b]$, $A_S(i,m,k)$ is non-empty if and only if $A'_S(i,m,k)$ is non-empty,
        \item for all $i,j\in[a]$ and $m\in a+[b]$, if $A_S(i,j,m)$ is non-empty then $A'_S(i,j,m)$ is empty,
        \item for all $i,k\in[a]$ and $m\in a+[b]$, $A'_S(i,m,k)$ does not contain an element of $S_k$.
    \end{itemize}
\end{definition}

\begin{lemma}
\label{lemma: odd a construction}
    There exists a $\LC(a^2b^1)$ if there exists $A$ and $B$ as defined in \Cref{def: extension by B with S} such that
    \begin{enumerate}[(1)]
        \item there exists a partial latin cube $A'_{[a-b]}$ that is a shifted partner of $A_{[a-b]}$,
        \item there exists a completion $A_{[a-b]}^*$ of $A_{[a-b]}$, and
        \item for all $i,j\in[a]$, $A_{[a-b]}(i,j,a+1)\in[a]$ if and only if $A_{[a-b]}(i,j,m)\in[a]$ for all $m\in\{a+2,\dots,a+b\}$.
    \end{enumerate}
\end{lemma}
\begin{proof}
    Let $L$ be an order $2a+b$ array, and for $i,j,k\in[3]$ let $L_{i,j,k}$ be subarrays of $L$, where $$L_{i,j,k} = \{L(x,y,z)\mid x\in f(i),y\in f(j),z\in f(k)\}$$ for $f(1) = [a]$, $f(2) = a+[a]$ and $f(3) = 2a+[b]$.

    Let $S = [a-b]$ and $T = a-b+[b]$. Construct the extensions $A_S$ and $A_T$ of $A$ by $B$ with $S$ and $T$ respectively, and suppose that the shifted partner $A'_S$ and completion $A^*_S$ exist. Apply the same permutations of layers used for $A'_S$ to obtain $A'^*_S$.

    Note that $A_S$ uses the symbols of $[a+b]$, with only the symbols of $[a]$ occurring in the layers of $a+[b]$. Let $C_S$ be a copy of $A_S$ with the symbol $a+s$ for $s\in[b]$ replaced with $2a+s$. This swaps all occurrences of symbols in $a+[b]$ with symbols in $2a+[b]$, leaving the symbols in $a+[a]$ free. Likewise construct $C'_S$, $C^*_S$ and $C_T$.

    Let $D_S$, $D'_S$, $D^*_S$ and $D_T$ be copies of $C_S$, $C'_S$, $C^*_S$ and $C_T$ respectively, where each symbol $i\in[a]$ is replaced with $a+i\in a+[a]$.

    Place copies of the partial latin cubes as shown in \Cref{fig: odd construction structure}, where each array is placed across multiple subarrays. For example, the partial latin cube $C_T$ of order $a+b$ is placed across $L_{2,2,1}$, $L_{2,3,1}$, $L_{3,2,1}$ and $L_{2,2,3}$. Note that $[a]$, $a+[a]$ and $2a+[b]$ indicate the subarrays that are subcubes and are filled with any latin cube on the given symbol set.

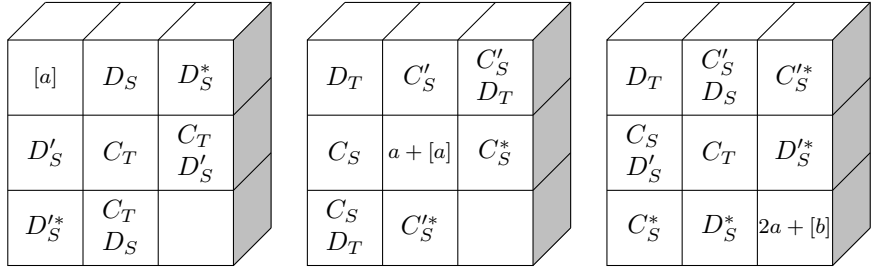
\begin{figure}[h]
    \centering
    \begin{tikzpicture}
        \filldraw[color=lightgray] (3,0) -- (3.5,0.5) -- (3.5,3.5) -- (3,3);
        \draw (0,0) rectangle (3,3);

        \foreach \x in {1,2}{
            \draw (0,\x) -- (3,\x);
            \draw (\x,0) -- (\x,3);}

        \foreach \x [count=\i from 0] in {{\footnotesize{$[a]$},$D_S$,$D^*_S$},{$D'_S$,$C_T$,\thead{$C_T$\\$D'_S$}},{$D'^*_S$,\thead{$C_T$\\$D_S$},}}{
            \foreach \y [count=\j from 0] in \x
                \node at (0.5+ 1*\j,2.5- 1*\i) {\y};
            }

        \foreach \x in {0,1,2,3}
            \draw (\x,3) -- (\x+0.5,3.5);

        \foreach \x in {0,1,2}
            \draw (3,\x) -- (3.5,\x+0.5);

        \draw (3.5,0.5) -- (3.5,3.5);
        \draw (0.5,3.5) -- (3.5,3.5);
    \end{tikzpicture}
    \quad
    \begin{tikzpicture}
        \filldraw[color=lightgray] (3,0) -- (3.5,0.5) -- (3.5,3.5) -- (3,3);
        \draw (0,0) rectangle (3,3);

        \foreach \x in {1,2}{
            \draw (0,\x) -- (3,\x);
            \draw (\x,0) -- (\x,3);}

        \foreach \x [count=\i from 0] in {{$D_T$,$C'_S$,\thead{$C'_S$\\$D_T$}},{$C_S$,\footnotesize{$a+[a]$},$C^*_S$},{\thead{$C_S$\\$D_T$},$C'^*_S$,}}{
            \foreach \y [count=\j from 0] in \x
                \node at (0.5+ 1*\j,2.5- 1*\i) {\y};
            }

        \foreach \x in {0,1,2,3}
            \draw (\x,3) -- (\x+0.5,3.5);

        \foreach \x in {0,1,2}
            \draw (3,\x) -- (3.5,\x+0.5);

        \draw (3.5,0.5) -- (3.5,3.5);
        \draw (0.5,3.5) -- (3.5,3.5);
    \end{tikzpicture}
    \quad
    \begin{tikzpicture}
        \filldraw[color=lightgray] (3,0) -- (3.5,0.5) -- (3.5,3.5) -- (3,3);
        \draw (0,0) rectangle (3,3);

        \foreach \x in {1,2}{
            \draw (0,\x) -- (3,\x);
            \draw (\x,0) -- (\x,3);}

        \foreach \x [count=\i from 0] in {{$D_T$,\thead{$C'_S$\\$D_S$},$C'^*_S$},{\thead{$C_S$\\$D'_S$},$C_T$,$D'^*_S$},{$C^*_S$,$D^*_S$,\footnotesize{$2a+[b]$}}}{
            \foreach \y [count=\j from 0] in \x
                \node at (0.5+ 1*\j,2.5- 1*\i) {\y};
            }

        \foreach \x in {0,1,2,3}
            \draw (\x,3) -- (\x+0.5,3.5);

        \foreach \x in {0,1,2}
            \draw (3,\x) -- (3.5,\x+0.5);

        \draw (3.5,0.5) -- (3.5,3.5);
        \draw (0.5,3.5) -- (3.5,3.5);

        \node at (0.5,0.5) {\phantom{\thead{$C_S$\\$D_T$}}};
    \end{tikzpicture}

    \caption{The subarray structure of $L$}
    \label{fig: odd construction structure}
\end{figure}

    In an order $a+b$ latin cube ($A_S$, $C_S$, $D_S$, etc.), let the $i$-extension, $j$-extension and $k$-extension refer to the cells in $(a+[b],[a],[a])$, $([a],a+[b],[a])$ and $([a],[a],a+[b])$ respectively. Similarly, let the $ij$-, $ik$- and $jk$-extensions refer to the cells in $(a+[b],a+[b],[a])$, $(a+[b],[a],a+[b])$ and $([a],a+[b],a+[b])$ respectively. Also, let the centre refer to the cells in $([a],[a],[a])$.

    We first show that the entries of $L$ so far form a partial latin cube. Using the fact that the entries in the $k$-extensions of $A_S$ and $A'_S$ are in disjoint cells, it is sufficient to establish the following:
    \begin{enumerate}[(i)]
        \item If the $i$-, $j$- or $k$-extension of $C_S$, $C'_S$ or $C_T$ and the same extension of $D_S$, $D'_S$ or $D_T$ do not have entries in the same cells, then the centres and that extension of the relevant arrays can be placed in the same columns, rows or files respectively.
        \item The $i$-, $j$- and $k$-extensions of $A_S$ and $A_T$ can be placed in the same rows, columns and files and together fill the subarray. The same is true for the $j$-extensions of $A'_S$ and $A_T$.
        \item The $k$-extensions of $A'_S$ and $A_T$ can be placed in the same columns, and the $i$-extensions can be placed in the same files.
        \item The $i$- and $j$-extensions of $A_S$ and $A'_S$ can be placed in the same rows and columns respectively.
        \item The $i$-, $j$- and $k$-extensions of $A_S$ contain only symbols from $[a]$ and are completed to $A^*_S$ with only symbols from $a+[b]$.
        \item The $ij$-, $ik$- and $jk$-extensions of $C^*_S$ and $D^*_S$ contain only symbols from $[a]$ and $a+[a]$ respectively.
    \end{enumerate}

    Consider the extensions $A_X$ and $A_Y$ of $A$ by $B$ with $X\subseteq[a]$ and $Y\subseteq[a]$ respectively. Suppose that the entries in the $i$-extensions are in disjoint sets of cells. For any column $(\cdot,j,k)$ where $j,k\in[a]$, if symbol $a+m$ for $m\in[b]$ occurs in this column of $A_X$ then it must occur within the centre and $A_X(a+m,j,k)$ is non-empty. Similarly, if $a+m$ occurs within the same column of $A_Y$, then $A_Y(a+m,j,k)$, which is in the $i$-extension, must be non-empty. This contradicts the assumption that the entries in the $i$-extensions are in disjoint cells, and so no symbol of $a+[b]$ occurs within the same column of both $A_X$ and $A_Y$.

    Taking $C_X$ and $D_Y$, these partial latin cubes share only the symbols of $2a+[b]$ which correspond to the symbols of $a+[b]$ in $A_X$ and $A_Y$. Thus, since those symbols do not occur within the same column of both centres and the $i$-extensions are in disjoint sets of cells, the $i$-extensions can be placed in the same subarray and the centres can be placed in the two subarrays above them (placing them in the same columns).

    Observe that since the layers $(i,\cdot,\cdot)$ and $(\cdot,\cdot,i)$ are only permuted for $i\in[a]$ to get $A'_S$, it follows that $A'_S$ has the same property as $A_S$: if $a+m$ occurs in column $(\cdot,j,k)$ of $A'_S$ then $A'_S(a+m,j,k)$ is non-empty. Thus, the same arguments hold for $C'_S$ and $D'_S$.

    A similar proof is used to show that the $j$- and $k$-extensions can be placed in the same rows and files respectively. Therefore, (i) is shown.

    \smallskip

    Consider the extension $A_{[a]}$ of $A$ by $B$ with $[a]$. Observe that $[a]_k = [a] = S_k\sqcup T_k$ for all $k\in[a]$, so if $A_S(a+m,j,k) = s$ or $A_T(a+m,j,k) = s$ for some $m\in[b]$ and $j,k\in[a]$, then $A_{[a]}(a+m,j,k) = s$ also. The same is true for $A_{[a]}(i,a+m,k)$ and $A_{[a]}(i,j,a+m)$ for $m\in[b]$ and $i,j,k\in[a]$. Thus, the $i$-, $j$- and $k$-extensions of $A_S$ and $A_T$ are a partition of the entries in the same sections of $A_{[a]}$.

    Therefore, if the extensions of $A_S$ and $A_T$ are placed in the same subarray, then they do not overlap and instead fill the entire array as a partial latin cube. They can also be placed in adjacent subarrays with no symbol repeated in a line. Also, since $A'_S$ has the same non-empty cells as $A_S$ in the $j$-extension, that section of $A'_S$ can also be placed with the same section of $A_T$. Thus, (ii) is satisfied.

    \smallskip

    In the shifted partner $A'_S$, the $k$-extension is a copy of the same section of $A_S$ with some permutation applied to the layers $(i,\cdot,\cdot)$. Thus, the symbols in any column of the $k$-extensions of $A'_S$ and $A_S$ are the same. As established earlier, the $k$-extensions of $A_S$ and $A_T$ can be placed in the same columns, meaning that they share no symbols in a column. Therefore, the same is true for the $k$-extensions of $A'_S$ and $A_T$.

    A similar argument shows that the $i$-extensions of $A'_S$ and $A_T$ can be placed in the same files. Thus, (iii) holds.

    \smallskip

    The layer $(\cdot,\cdot,k)$ for $k\in[a]$ of the $i$- and $j$-extensions of $A_S$ contains only symbols of $S_k$. By definition of a shifted partner, the $j$-extension of $A'_S$ contains no symbols from $S_k$. Since each layer must be a permuted version of a layer in $A_S$, and in each layer $(\cdot,\cdot,k)$ of $A_S$ the $i$- and $j$-extensions have the same symbols, the $i$-extension of $A'_S$ also contains no symbols from $S_k$. Thus, the $i$- and $j$-extensions can be placed in the same rows and columns, respectively, of $L$, proving (iv).

    \smallskip

    Observe that the cells in the centre of $A^*_S$ have a combination of symbols from $[a]$ and $a+[b]$. By definition of the extension $A_S$, for any $i,j,k\in[a]$, if $A(i,j,k) = s$ then either $A_S(i,j,k)=s$ or $A_S(i,j,k)=m$ for some $m\in a+[b]$ and $A_S(i,j,a+m) = A(i,a+m,k) = A(a+m,j,k) = s$. In either case, every symbol $s\in[a]$ still occurs once in every line $(i,j,\cdot)$ and $(i,\cdot,k)$ and $(\cdot,j,k)$ for all $i,j,k\in[a]$. Therefore, the empty cells in these lines of $A_S$ are filled with symbols from $a+[b]$ in $A^*_S$.

    For any $k\in[a]$, the filled cells in layer $(\cdot,\cdot,k)$ of the $i$- and $j$-extensions of $A_S$ contain only the symbols of $S_k$. Thus, there are at most $b(a-b)$ filled cells. The empty cells are then filled with symbols from $a+[b]$ to get $A^*_S$, and there are at most $b^2$ such entries. Since there are $ab$ cells in each section, there must be exactly $b(a-b)$ symbols from $S_k$ and exactly $b^2$ from $a+[b]$ ($b$ of each symbol from $S_k\cup (a+[b])$). Therefore, the cells in the same layer of the $ij$-extension must be filled with symbols from $T_k$.

    It follows that all entries in the $ij$-extension of $A^*_S$ are in $[a]$. Note that the entries in the $i$-, $j$- and $k$-extensions are forced by the entries in the centre, and so all three regions contain the same symbols (only in different arrangements). Thus, the $ik$- and $jk$-extensions will also only have symbols from $[a]$. Considering the changed symbol sets, the entries in these extensions of $C^*_S$ and $C'^*_S$ are in $[a]$, and those of $D^*_S$ and $D'^*_S$ are in $a+[a]$. Therefore, (v) and (vi) are true.  

    \smallskip

    Thus, the entries so far form a partial latin cube. We now fill the remaining cells.
    
    Observe that for $k\in[a]$ the cells in $(a+[b],[a],k)$ of the completed latin cube $C^*_S$ contain the symbols of $S_k$ and each symbol of $2a+[b]$ occurs in every line $(i,[a],k)$. It follows that in layer $(\cdot,\cdot,k)$ for $k\in[a]$ of $L$, the entries in any line $(i,\cdot,k)$, where $i\in 2a+[b]$, are those in the set $T_k\cup \{a+\ell\mid \ell\in S_k\}\cup \{a+\ell\mid \ell\in S'_k\}\cup (2a+[b])$ which is a set of $2a$ symbols. The symbols in the lines $(\cdot,j,k)$ for $j\in 2a+[b]$ are from the same set, which we denote $U_k$. The same is true for the layers $(\cdot,\cdot,k)$ of $L$ where $k\in a+[a]$, with the $2a$ symbols of $U_k$ being those of $(2a+[b])\cup \{a+\ell\pmod{2a}\mid\ell\in U_{k-a}\cap [2a]\}$.
    
    The symbols of $2a+[b]$ each occur in every set $U_k$, so appear $2a$ times in the layer $(2a+1,\cdot,[2a])$. The symbols of $[a]$ each occur $b$ times across $U_k$ for $k\in[a]$ and $2(a-b)$ times across $U_k$ for $k\in a+[a]$. The same is true for the symbols of $a+[a]$ but across the opposite ranges of $k$. Thus, every symbol occurs at least $2a-b$ times in the layer, and so by \Cref{lemma: right corner channel} the $(2a+b)\times (2a+b)\times (2a)$ partial latin cube in the cells $(\cdot,\cdot,[2a])$ is completed to fill $L_{3,3,1}$ and $L_{3,3,2}$. Note that $L_{3,3,1}$ and $L_{3,3,2}$ contain only symbols from $[2a]$.
    
    For the partially-filled subarrays, recall that by the Lemma condition (3) $C_S(i,j,a+1)$ is empty for $i,j\in[a]$ if and only if $C_S(i,j,m)$ is also empty for any $m\in a+1+[b-1]$. This will also be true for $C'_S$ and $C_T$. Thus, each line $(i,j,\cdot)$ of $L$ for $i,j\in[2a]$ is either complete or has the last $b$ cells empty and requiring the symbols of $2a+[b]$. Further, there are $a-b+(a-2(a-b)) = b$ incomplete lines $(i,j,\cdot)$ for each $i\in[2a]$ and each $j\in[2a]$, and so each symbol of $2a+[b]$ needs to occur once in every row and column of layer $(\cdot,\cdot,k)$ for $k\in2a+[b]$. Let $U = \{(i,j)\in[2a]\times[2a]\mid \text{cell $(i,j,2a+1)$ is empty}\}$. Therefore, for all $(i,j)\in U$, $L(i,j,k)$ for $k\in 2a+[b]$ is empty and requires a symbol from $2a+[b]$. Thus, by \Cref{lemma: filling back entries}, the $2a\times 2a\times (2a+b)$ array $([2a],[2a],[2a+b])$ is completed.

    Therefore, the latin cube $L$ is completed.
\end{proof}

\subsection{Completion construction}

In this section we construct the components needed to use \Cref{lemma: odd a construction} for $a\equiv 1,5\pmod6$. The constructions are based on cyclic latin cubes, so we begin with a result which gives a method for constructing such latin cubes. When working modulo $a$ we take the residues in $[a]$; that is $ka\equiv a\pmod{a}$ instead of $0$.

\begin{lemma}
\label{lemma: cyclic is cube}
    Let $a$, $a_1$, $a_2$, $a_3$ and $a_4$ be positive integers such that $\gcd(a,a_i) = 1$ for all $i\in[3]$. Then the $a\times a\times a$ array $L$ with entries defined by $L(i,j,k) = a_1i + a_2j + a_3k + a_4\pmod{a}$, for all $i,j,k\in [a]$, is a latin cube.
\end{lemma}
\begin{proof}
    Suppose that $L(i,j,k_1) = L(i,j,k_2)$ for some $i,j,k_1,k_2\in[a]$. Then $a_3k_1 \equiv a_3k_2\pmod{a}$, and so $a_3(k_1-k_2) = am$ for some $m\in\mathbb{Z}$. Since $\gcd(a,a_3) = 1$, $k_1 \equiv k_2\pmod{a}$ and thus $k_1=k_2$. Therefore, each line $(i,j,\cdot)$ has each symbol of $[a]$ occurring once.

    A similar proof is used for the lines $(i,\cdot,k)$ and $(\cdot,j,k)$.
\end{proof}

\begin{corollary}
\label{lemma: cyclic is square}
    If $\gcd(a,a_i) = 1$ for all $i\in[2]$, then an $a\times a$ array $L$ with entries defined by $L(i,j) = a_1i + a_2j + a_3 \pmod{a}$ is a latin square of order $a$.
\end{corollary}

A key component of \Cref{lemma: odd a construction} is the existence of two latin cubes $A$ and $B$ which give an extension with a shifted partner. The following result constructs such latin cubes for $a\equiv 1,5\pmod6$. We use these same latin cubes for the remainder of this section.

\begin{lemma}
\label{lemma: specific odd cubes are orthogonal}
    For all $a\equiv 1,5\pmod6$, let $A$ and $B$ be latin cubes of order $a$ where for all $i,j,k\in[a]$,
    \begin{align*}
        A(i,j,k) &= -i+j+k\pmod{a},\\
        B(i,j,k) &= -i-j+2k+1\pmod{a}.
    \end{align*}
    Then for all $b\in[a]$, $A$ and $B$ satisfy \Cref{def: extension by B with S} and the extension $A_{[a-b]}$ satisfies condition (3) from \Cref{lemma: odd a construction}.
\end{lemma}
\begin{proof}
    Since $A$ and $B$ are latin cubes by \Cref{lemma: cyclic is cube}, we only show that all triples are covered in each of the sets $\{(i, A(i,j,k), B(i,j,k))\mid i,j,k\in[a]\}$, $\{(j, A(i,j,k), B(i,j,k))\mid i,j,k\in[a]\}$ and $\{(k, A(i,j,k), B(i,j,k))\mid i,j,k\in[a]\}$.

    If $(i,A(i,j,k),B(i,j,k)) = (i,A(i,j',k'),B(i,j',k'))$ then $j+k \equiv j'+k'\pmod{a}$ and $-j+2k\equiv -j'+2k'\pmod{a}$. Combining these, we see that $3k\equiv 3k'\pmod{a}$, and since $3\nmid a$, $k=k'$. Thus, $j=j'$ also.

    There is a triple for each combination of $i,j,k\in[a]$, and so each of the $n^3$ possible triples must occur exactly once.

    The proofs for the other sets are similar.

    Let $S = [a-b]$. Observe that $A(i,j,k) = -i+j+k\in S_k$ when $-i+j+k \equiv k-1+s\pmod{a}$ for some $s\in[a-b]$, and thus when $-i+j+1\equiv s\pmod{a}$. Also, for any $i,j\in[a]$ and any $m\in[b]$, there must exist some $k\in[a]$ such that $B(i,j,k) = m$. Therefore, $A_S(i,j,a+m)$ is non-empty if and only if $-i+j+1\equiv s\pmod{a}$ for some $s\in[a-b]$. It follows that $A_S(i,j,a+1)$ is non-empty if and only if $A_S(i,j,a+m)$ is non-empty for all $m\in 1+[b-1]$. Thus, (3) of \Cref{lemma: odd a construction} is satisfied by these latin cubes.
\end{proof}

In order to use these latin cubes $A$ and $B$ with \Cref{lemma: odd a construction}, we need to find a completion and a shifted partner for $A_{[a-b]}$.

\begin{lemma}
\label{lemma: specific odd has shifted partner}
    For all $a\equiv 1,5\pmod6$ and all $b\geq \frac{a}{2}$, there exists a shifted partner for $A_S$, where $S = [a-b]$, with latin cubes $A$ and $B$ as defined in \Cref{lemma: specific odd cubes are orthogonal}.
\end{lemma}
\begin{proof}
    Let $A'_S$ be a partial latin cube of order $a+b$ such that for all $i,j,k\in[a+b]$
    $$A'_S(i,j,k) = \begin{cases}
        A_S\big(i+b-1\pmod{a},j,k+b-1\pmod{a}\big), & \text{if $i,k\in[a]$,}\\
        A_S\big(i+b-1\pmod{a},j,k\big), & \text{if $i\in[a]$ and $k>a$,}\\
        A_S\big(i,j,k+b-1\pmod{a}\big), & \text{if $k\in[a]$ and $i>a$.}\\
    \end{cases}$$
    Note that if $i,k>a$ then $A_S(i,j,k)$ and $A'_S(i,j,k)$ are both empty.
    
    Clearly $A'_S$ is a permuted copy of $A_S$, and $A'_S$ is an extension of $A'$ by $B'$ with $S'$ where $S' = b-1+[a-b]$ and
    \begin{align*}
        A'(i,j,k) &= -i+j+k\pmod{a} = A(i,j,k),\\
        B'(i,j,k) &= -i-j+2k+b\pmod{a}.
    \end{align*}

    If $A_S(i,a+m,k)$ is non-empty for some $i,k\in[a]$ and $m\in[b]$, then there exists a $j\in[a]$ such that $B(i,j,k) = m$ and $A(i,j,k)\in S_k$. Let $j'=j+b-1\pmod{a}$ and observe that $B'(i,j',k) = B(i,j,k) = m$ and $A'(i,j',k) = b-1 + A(i,j,k)\in S'_k$. Thus, $A'_S(i,a+m,k)$ is non-empty. Similarly, if $A_S(i,a+m,k)$ is empty, then $B(i,j,k) = m$ but $A(i,j,k)\notin S_k$ for some $j\in[a]$. Thus, $B'(i,j',k) = m$ and $A'(i,j',k)\notin S'_k$, so $A'_S(i,a+m,k)$ is empty.

    As shown in the proof of \Cref{lemma: specific odd cubes are orthogonal}, $A_S(i,j,a+m)$ for $i,j\in[a]$ and $m\in[b]$ is non-empty if and only if $-i+j+1\equiv s\pmod{a}$ for some $s\in[a-b]$. By a similar proof, $A'_S(i,j,a+m)$ is non-empty if and only if $-i+j+1\equiv b-1+s\pmod{a}$ for $s\in[a-b]$. Suppose there exists $i,j\in[a]$ such that both $-i+j+1\equiv s$ and $-i+j+2-b\equiv s'$ for some $s,s'\in[a-b]$. Then $b-1\equiv s-s'\pmod{a}$. Since $b\geq \frac{a}{2}$ and $a$ is odd, $a-b\leq b-1$ and so $s-s'\equiv d$ where $1\leq d\leq b-2$ or $b+1\leq d\leq a$. Thus, there are no such $i,j\in[a]$ such that both $A_S(i,j,a+m)$ and $A'_S(i,j,a+m)$ are non-empty.

    In layer $(\cdot,\cdot,k)$ for $k\in[a]$, the symbols in the cells of $([a],a+[b],k)$ in $A'_S$ come from $S'_k$. For any $k\in[a]$, if $\ell\in S_k$ and $\ell\in S'_k$ then $\ell\equiv k-1+s\pmod{a}$ and $\ell\equiv k-1+b-1+s'\pmod{a}$ for some $s,s'\in[a-b]$. Thus, $b-1\equiv s-s'\pmod{a}$. As above, this is not possible, and so $S_k$ and $S'_k$ are disjoint sets of symbols.

    Therefore, $A'_S$ is a shifted partner of $A_S$.
\end{proof}

We have shown that our chosen latin cubes $A$ and $B$ satisfy (1) and (3) from \Cref{lemma: odd a construction}, and the last requirement is a completion of $A_S$.

\begin{lemma}
\label{lemma: specific odd cubes rely on last square}
    Let $L$ be a partial latin square of order $a\equiv 1,5\pmod6$ with $a\geq b$ such that
    $$L(i+a-b,j) = \begin{cases}
    \frac{a+3}{2}j + \frac{a+1}{2}i - 1\pmod{a}, & \text{for all $i\in [b]$ and $j\in[a-b]$,}
    \end{cases}$$
    $$L(i,j+a-b) = \begin{cases}
        \frac{a+3}{2}i + \frac{a+1}{2}j - 1\pmod{a}, & \text{for all $i\in [a-b]$ and $j\in[b]$,}
    \end{cases}$$
    and $L(i,j)$ is empty when both $i,j\in[a-b]$ or both $i,j\in a-b+[b]$. If the cells $(i,j)$ where $i,j\in a-b+[b]$ can be filled, then there exists a completion $A^*_S$ of $A_S$, with $A_S$ as defined in \Cref{lemma: specific odd cubes are orthogonal}.
\end{lemma}
\begin{proof}
    Split $A^*_S$ into subarrays $A_{i,j,k}$ for $i,j,k\in[2]$, where $A_{1,j,k}$ is in the layers $([a],\cdot,\cdot)$, $A_{2,j,k}$ is in the layers $(a+[b],j,k)$, and $j$ and $k$ are defined similarly.

    $A_{1,1,1}$ is already filled and in the layer $(\cdot,\cdot,k)$ for $k\in[a]$, the symbols of $S_k$ each occur $a-b$ times, the symbols of $a+[b]$ each occur $a-b$ times, and the symbols of $[a]\setminus S_k$ each occur $a$ times. The same layer of $A_{2,1,1}$ has $b$ copies of each symbol in $S_k$ and is otherwise empty. Clearly, this must be filled with symbols from $a+[b]$.

    \vspace{0.125cm}
    \noindent\textbf{Subarray $\mathbf{A_{2,1,1}}$:}
    
    For $i,j,k\in[a]$, if $A(i,j,k)\notin S_k$ and $B(i,j,k)\in[b]$, then set $$A^*_S\big(a+B(i,j,k),j,k\big) = a+ \big(k-A(i,j,k)\pmod{a}\big).$$
    It is clear that this will fill all cells that are empty in the subarray and that all entries are from $a+[b]$. Further, given $i,j,k\in[a]$, if $B(i,j,k) \equiv -i-j+2k+1\equiv m\pmod{a}$, then the entry in $A^*_S(a+m,j,k)$ is $a+(-2j+2k+1-m\pmod{a})$. This shows that, according to \Cref{lemma: cyclic is cube}, this subarray is a subset of an order $a$ latin cube $C$ defined by $C(i,j,k) = a+(-i-2j+2k+1\pmod{a})$, and so the same symbol will not occur twice in any line.

    If $A^*_S(a+m,j,k) = A^*_S(i,j,k)$ for $i,j,k\in[a]$ and $m\in[b]$, then there exists $i'\in[a]$ such that $B(i',j,k) = m$, while $B(i,j,k) \equiv k-A(i',j,k)\pmod{a}$. Thus, $i'\equiv -j+2k+1-m\pmod{a}$ and $-i-j+2k+1\equiv -i'+j+k\pmod{a}$, and so $i-j\equiv m\pmod{a}$. Since $A(i,j,k)\in S_k$, $-i+j+1\equiv s\pmod{a}$ for some $s\in[a-b]$. Then $m\equiv -s+1\pmod{a}$. However, $m\in[b]$ but $-s+1\equiv s'\pmod{a}$ where $b+1\leq s'\leq a$. Thus, it is not possible that $A^*_S(a+m,j,k) = A^*_S(i,j,k)$.

    Therefore, $A_{2,1,1}$ is filled.

    \vspace{0.125cm}
    \noindent\textbf{Subarray $\mathbf{A_{1,2,1}}$:}

    As with $A_{2,1,1}$, the empty cells in each layer $(\cdot,\cdot,k)$ of $A_{1,2,1}$ need to be filled with symbols from $a+[b]$. Similarly to above, for $i,j,k\in[a]$, if $A(i,j,k)\notin S_k$ and $B(i,j,k)\in[b]$, then set $$A^*_S\big(i,a+B(i,j,k),k\big) = a+ \big(A(i,j,k)-k+1+b\pmod{a}\big).$$
    If $B(i,j,k) = m$, then the entry in $A^*_S(i,a+m,k)$ is $a+(-2i+2k+2+b-m\pmod{a})$. Thus, the new entries are a subset of the order $a$ latin cube $C$ given by $C(i,j,k) = a+(-2i-j+2k+2+b\pmod{a})$, which is a latin cube by \Cref{lemma: cyclic is cube}.

    If $A^*_S(i,a+m,k) = A^*_S(i,j,k)$ for $i,j,k\in[a]$ and $m\in[b]$, then there exists $j'\in[a]$ such that $B(i,j',k) = m$, while $B(i,j,k) \equiv A(i,j',k)-k+1+b\pmod{a}$. It follows that $j'\equiv -i+2k+1-m\pmod{a}$ and $-i-j+2k+1\equiv -i+j'+1+b\pmod{a}$, and so $i-j\equiv b+1-m\pmod{a}$. Since $A(i,j,k)\in S_k$, $-i+j+1\equiv s\pmod{a}$ for some $s\in[a-b]$. Thus, $m-b\equiv s\pmod{a}$. However, $m\in[b]$ and so $m-b\equiv s'$ where $a-b+1\leq s'\leq a$ but $s\in[a-b]$. Thus, it is not possible that $A^*_S(i,a+m,k) = A^*_S(i,j,k)$.

    \vspace{0.125cm}
    \noindent\textbf{Subarrays $\mathbf{A_{1,1,2}}$:}

    As shown earlier, the empty cells of $A_{1,1,2}$ appear such that any line $(i,j,\cdot)$ either has all of the cells in this subarray filled or all cells in the subarray empty. If empty, they require the symbols of $a+[b]$ and these symbols do not already appear anywhere in the subarray. Thus, by \Cref{lemma: filling back entries}, the $a\times a\times (a+b)$ array formed from $A_{1,1,1}$ and $A_{1,1,2}$ is filled.

    \vspace{0.125cm}
    \noindent\textbf{Subarray $\mathbf{A_{2,2,1}}$:}
    
    The symbols in the line $(i,\cdot,k)$ or $(\cdot,j,k)$, for any $i,j\in a+[b]$ and $k\in[a]$, are in the set $V_k = S_k\cup (a+[b])$. Thus, each symbol of $a+[b]$ occurs $a$ times and the symbols of $[a]$ occur $a-b$ times in the layer $(a+1,\cdot,[a])$. Thus, using \Cref{lemma: right corner channel} with $n=a+b$, $r=a$ and $s=b$, $A_{2,2,1}$ is filled with just the symbols of $[a]$.

    \vspace{0.125cm}
    \noindent\textbf{Subarray $\mathbf{A_{2,1,2}}$:}
    
    To fill $A_{2,1,2}$, for $m,n\in [b]$ and $j\in[a]$, let $A^*_S(a+m,j,a+n) = \frac{a+1}{2}(m+n-b)+j-1\pmod{a}$. Since $\gcd(a,\frac{a+1}{2})=1$, this subarray is then, by \Cref{lemma: cyclic is cube}, a subset of a latin cube and only uses the symbols of $[a]$. If $A^*_S(a+m,j,a+n) = A^*_S(i,j,a+n)$ for some $i,j\in[a]$ and $m,n\in[b]$, then there exists some $k\in[a]$ such that $A(i,j,k) = A^*_S(i,j,a+n)$ and $B(i,j,k) = n$. Thus, $-i+j+k\equiv \frac{a+1}{2}(m+n-b)+j-1\pmod{a}$ and $2k\equiv n+i+j-1\pmod{a}$, and so $-i+3j+n-1\equiv m+n-b+2j-2\pmod{a}$ and $-i+j+1\equiv m-b\pmod{a}$. Since $A(i,j,k)\in S_k$, by the same reasoning as with $A_{1,2,1}$ this is not possible.

    If $A^*_S(a+m,j,a+n) = A^*_S(a+m,j,k)$ for some $j,k\in[a]$ and $m,n\in[b]$, then there exists an $i\in[a]$ such that $A(i,j,k) = A^*_S(a+m,j,k)$ and $B(i,j,k) = m$. In the same way as above, it follows that $-i+j+1\equiv n-b\pmod{a}$ and this is not possible.
    
    \vspace{0.125cm}
    \noindent\textbf{Subarray $\mathbf{A_{1,2,2}}$:}

    By assumption, we can complete the cells $(i,j)$ with $i,j\in a-b+[b]$ of the partial latin square $L$. The other entries of $L$ were chosen so that this $b\times b$ subarray of $L$ can be used to fill $A_{1,2,2}$.
    
    Let $A^*_S(i,a+m,a+n) = L(a-b+m,a-b+n)+i-1\pmod{a}$. All entries are in $[a]$ since the symbols of $a+[b]$ occur in every row of $A_{1,1,2}$. In the line $(i,[a],a+m)$ for $i\in[a]$ and $m\in[b]$, if the symbol $\ell\in[a]$ appears then there exists $j,k\in[a]$ such that $B(i,j,k) = m$ and $A(i,j,k)=\ell\in S_k$. Note that there are $a-b$ such entries in each line.

    For all $i\in[a]$ and $m\in[b]$, choose any $x\in[a-b]$ and let $\ell \equiv i + \frac{a+3}{2}x + \frac{a+1}{2}m -2\pmod{a}$, $j \equiv i + x -1\pmod{a}$ and $k\equiv i + \frac{a+1}{2}(x+m) - 1\pmod{a}$. Then $A(i,j,k) \equiv i + \frac{a+3}{2}x + \frac{a+1}{2}m - 2 \equiv \ell\equiv k-1+x\pmod{a}\in S_k$ and $B(i,j,k) \equiv m\pmod{a}$. Since $m\in[b]$, $A^*_S(i,j,a+m) = \ell$ and $A^*_S(i,a+m,k) = \ell$. Therefore, the $a-b$ entries from $[a]$ in the line $(i,[a],a+m)$ and in the line $(i,a+m,[a])$ are $\{i + \frac{a+3}{2}x + \frac{a+1}{2}m -2\pmod{a}\mid x\in[a-b]\}$. This set of symbols is the same as $\{i-1+L(a-b+m,x)\mid x\in[a-b]\}$ and $\{i-1+L(x,a-b+m)\mid x\in[a-b]\}$ which are disjoint sets to the symbols used in the part of the line in $A_{1,2,2}$. Therefore, $A_{1,2,2}$ is filled.
    
    \vspace{0.125cm}
    \noindent\textbf{Subarray $\mathbf{A_{2,2,2}}$:}

    Observe that $A_{2,2,1}$, $A_{2,1,2}$ and $A_{1,2,2}$ contain only symbols from $[a]$, so $A_{2,2,2}$ is a subcube on the symbols of $a+[b]$. Thus, it is easily filled with any order $b$ latin cube on the symbols of $a+[b]$.
\end{proof}

Combining the previous results, we construct a $\LC(a^2b^1)$ for all $b\geq \frac{a}{2}$ with $a\equiv 1,5\pmod6$.

\begin{theorem}
\label{thm: odd a final}
    For all $a\equiv 1,5\pmod6$ and $b\geq\frac{a}{2}$, there exists a $\LC(a^2b^1)$.
\end{theorem}

\begin{proof}
If $b=a$ then the 3-realization is a $\LC(a^3)$ and this exists by \Cref{a^n}. Thus, we assume that $b<a$ and $a\geq 5$.

    We use \Cref{lemma: odd a construction} to construct a $\LC(a^2b^1)$. Therefore, we must fill the latin square $L$ given in \Cref{lemma: specific odd cubes rely on last square}.

    We denote the multiplicative inverse of 3 modulo $a$ by
    \begin{equation*}
    g = \begin{cases} 
    \frac{2a+1}{3} & \text{if } a\equiv 1 \pmod{6}, \\ 
    \frac{a+1}{3} & \text{if } a\equiv 5 \pmod{6}. \\   
    \end{cases}
    \end{equation*}

   Construct a new order $a$ latin square $K$ with $$K(i,j) = \frac{a+3}{2}(i+j+2-2g)-1\pmod{a}$$ for all $i,j\in[a]$, which is a latin square by \Cref{lemma: cyclic is square}. Observe that $3\cdot\frac{a+1}{2} \equiv \frac{a+3}{2}\pmod{a}$ and hence $\frac{a+1}{2} \equiv g\frac{a+3}{2}\pmod{a}$. Thus, for $i\in[a-b]$ and $j\in a-b+[b]$,
    \begin{align*}
        L(i,j) &= \frac{a+3}{2}i + \frac{a+1}{2}(j+b) -1\pmod{a}\\
        &= \frac{a+3}{2}(i+gj+gb) -1\pmod{a}\\
        &= K(i-1,g(j+b+2)-1).
    \end{align*}

Let $d_1$ and $d_2$ be indicator variables which are equal to 1 when $b$ is congruent respectively to 1 or 2 modulo 3, and zero otherwise, and write $b=3y+d_1-d_2$.
Partition $a-b+[b]$ into sequences $U = (u_i)_{i=1}^{y+d_1}$, $V = (v_i)_{i=1}^y$ and $W = (w_i)_{i=1}^{y-d_2}$ with
    $$u_i = a-b-2+3i,\quad v_i = a-b-1+3i,\quad w_i = a-b+3i.$$
Then for $i \in [a-b]$ we have
\begin{align*}
    L(i,u_j) &= L(u_j,i) = K(i-1,j-1),\quad j \in [y+d_1],\\
    L(i,v_j) &= L(v_j,i) =K(i-1,j+g-1),\quad j \in [y],\\
    L(i,w_j) &= L(w_j,i) =K(i-1,j+2g-1),\quad j \in [y-d_2].
\end{align*}

We rearrange the rows and columns of $L$ in the order $[a-b]$, $U$, $V$, $W$, so that the regions of $L$ that are already specified can be filled using subarrays of $K$ on consecutive sequences of rows and columns$\pmod{a}$.
In a similar manner, we complete the portion of $L$ with rows and columns in $a-b+[b]$ with subarrays of $K$. Since $K(i,j)=K(i+d\pmod{a},j-d\pmod{a})$ for any $i,j,d\in[a]$, any subarray of $K$ on rows $i-1+[r]$ and columns $j-1+[c]$ (taken$\pmod{a}$) can be specified by the dimensions $r$ and $c$ and the sum $i+j$, which we refer to as the {\em index} of the subarray. Any two subarrays of $K$ on $c$ columns can be regarded as belonging to the same set of columns, and if adding the index of one subarray to its row count gives the index of the other subarray, then they correspond to consecutive sets of rows within those $c$ columns.

The required constructions are given in \Cref{fig: L construction 2 cases}, with subarrays of $K$ specified by the indices and the dimensions shown in bold at the top and left. Note that when $a\equiv 5\pmod6$ we set $a=3x-1$ with $g=x$, and when $a\equiv 1\pmod6$ we let $a=3x+1$ and $g=2x+1$. The highlighted subarrays are the entries that are specified in the definition of $L$. In the columns, sequences $U$, $V$ and $W$ have been divided into subsequences of the indicated lengths. By the conditions $b<a\leq 2b$ and $a \geq 5$, each of these lengths must be non-negative.

Consider the subarrays in \Cref{fig: L construction 2 cases} (A) corresponding to the rows in $U$. Beginning from any subarray, adding the index value to the width (number of columns) gives the index of another column, working modulo $a$. Continuing in this manner, we cycle through all 8 subarrays, demonstrating that these subarrays correspond to $|U|$ rows of $K$, with columns permuted in blocks, and so these rows of $L$ have each symbol occurring exactly once. The rows of $V$ and $W$, and the columns for each of the subsequences of $U$, $V$ and $W$, as well as case (B), can all be verified in the same way.

    \begin{figure}[h]
    \centering
    \begin{subfigure} {0.98\linewidth}
    \centering
    \scalebox{0.8}{
    $\begin{array}{c!{\vline width 1.5pt}c!{\vline width 1.5pt}c!{\vline width 1.5pt}c|c!{\vline width 1.5pt}c|c!{\vline width 1.5pt}c|c|c!{\vline width 1.5pt}}
         \multicolumn{2}{c}{} & \multicolumn{1}{c}{[a-b]} & \multicolumn{2}{c}{U} & \multicolumn{2}{c}{V} & \multicolumn{3}{c}{W} \\ 
        \cline{3-10}
        \multicolumn{2}{c!{\vline width 1.5pt}}{} & {\bf 3x-3y-1} & {\bf -x+2y} & {\bf x-y} & {\bf -x+2y} & {\bf x-y} & {\bf -x+2y} & {\bf x-y-1} & {\bf 1} \\
        \multicolumn{2}{c!{\vline width 1.5pt}}{} & {\bf -d_1+d_2} & {\bf +d_1} &  & {\bf +d_1-d_2} & {\bf -d_1+d_2} & {\bf +d_1} & {} & {\bf -d_1-d_2} \\
        \cline{2-10}
        [a-b] & {\bf 3x-3y-1} & \cellcolor{lightgray}\text{unfilled} & \cellcolor{lightgray}0	& \cellcolor{lightgray}-x+2y &	\cellcolor{lightgray}x & \cellcolor{lightgray}2y & \cellcolor{lightgray}-x+1 & \cellcolor{lightgray}x+2y & \cellcolor{lightgray}-x+y \\
        & {\bf -d_1+d_2} & \cellcolor{lightgray} & \cellcolor{lightgray}	& \cellcolor{lightgray}+d_1 &\cellcolor{lightgray}	 & \cellcolor{lightgray}+d_1-d_2 & \cellcolor{lightgray} & \cellcolor{lightgray}+d_1 & \cellcolor{lightgray}+d_1 \\
        \cline{2-10} 
        U & {\bf y+d_1} & \cellcolor{lightgray}0 & -3y & -x & x-2y & -y & -x-y+1 & x+y & -x-y \\
        &  & \cellcolor{lightgray} & -d_1+d_2 &  & -d_1+d_2 &  & -d_1 &  & +d_2 \\
        \cline{2-10} 
        V & {\bf y} & \cellcolor{lightgray}x & -y & -x+y & x-3y & d_1 & -x-2y+1 & x-y & -x-2y \\
        & & \cellcolor{lightgray} &  & +d_1 & -d_1+d_2 &  & -d_1 & +d_2 & +d_2 \\
        \cline{2-10} 
        W & {\bf y-d_2} & \cellcolor{lightgray}-x+1 & -2y & -x-y & x-y & y+d_1 & -x-3y+1 & x+d_2 & -x \\
        && \cellcolor{lightgray} & +d_2 & +d_2 & +d_2 &  & -d_1+d_2 &  & +d_1+d_2 \\
        \cline{2-10} 
    \end{array}$
    }
    \caption{Case $a\equiv 5\pmod{6}$; let $a=3x-1$, $b=3y+d_1-d_2$, $g=x$.}
    \end{subfigure}
    
    \vspace{5mm}

    \begin{subfigure} {0.98\linewidth}
    \centering
    \scalebox{0.8}{
    $\begin{array}{c!{\vline width 1.5pt}c!{\vline width 1.5pt}c!{\vline width 1.5pt}c|c!{\vline width 1.5pt}c|c!{\vline width 1.5pt}c|c|c!{\vline width 1.5pt}}
         \multicolumn{2}{c}{} & \multicolumn{1}{c}{[a-b]} & \multicolumn{2}{c}{U} & \multicolumn{2}{c}{V} & \multicolumn{3}{c}{W} \\ 
        \cline{3-10}
        \multicolumn{2}{c!{\vline width 1.5pt}}{} & {\bf 3x-3y+1} & {\bf -x+2y} & {\bf x-y} & {\bf -x+2y} & {\bf x-y} & {\bf -x+2y-1} & {\bf x-y+1} & {\bf d_1} \\
        \multicolumn{2}{c!{\vline width 1.5pt}}{} & {\bf -d_1+d_2} & {\bf -d_2} & {\bf +d_1+d_2} & {\bf -d_2} & {\bf +d_2} & {\bf +d_1} & {\bf -2d_1-d_2} & {} \\
        \cline{2-10}
        [a-b] & {\bf 3x-3y+1} & \cellcolor{lightgray}\text{unfilled} & \cellcolor{lightgray}0 & \cellcolor{lightgray}-x+2y &	\cellcolor{lightgray}-x & \cellcolor{lightgray}x+2y+1 & \cellcolor{lightgray}x+1 & \cellcolor{lightgray}2y+d1 & \cellcolor{lightgray}x+y+1 \\
        & {\bf -d_1+d_2} & \cellcolor{lightgray} & \cellcolor{lightgray} & \cellcolor{lightgray}-d_2 & \cellcolor{lightgray} & \cellcolor{lightgray}-d_2 & \cellcolor{lightgray} & \cellcolor{lightgray} & \cellcolor{lightgray}-d_1-d_2 \\    
        \cline{2-10}     
        U & {\bf y+d_1} & \cellcolor{lightgray}0 & -3y & -x& -x-y& x+y+1 & x-2y+1 & -y+d_2 & x-2y+1 \\
         &  & \cellcolor{lightgray} & -d_1+d_2 & -d_1 & -d_1 & -d_1-d_2 & -d_1 &  & -2d_1 \\    
        \cline{2-10} 
        V & {\bf y} & \cellcolor{lightgray}-x & -2y+d_2 & -x-y & -x-3y & x+1 & x-y+1 & y+d_1 & x-y+1 \\
         &  & \cellcolor{lightgray} &  & -d_1 & -d_1+d_2 & -d_1-d_2 &  &  & -d_1 \\
        \cline{2-10} 
        W & {\bf y-d_2} & \cellcolor{lightgray}x+1 & -y+d_2 & -x+y & -x-2y & x-y+1 & x-3y+1 & d_1+d_2 & x+1 \\
         &  & \cellcolor{lightgray} &  &  & -d_1+d_2 & -d_1 & -d_1+d_2 &  & -d_1 \\
        \cline{2-10} 
    \end{array}$
    }
    \caption{Case $a\equiv 1\pmod{6}$; let $a=3x+1$, $b=3y+d_1-d_2$, $g=2x+1$.}
    \end{subfigure}
    \caption{Construction of $L$, in two cases depending on $a$. Column and row ordering is shown in margins at the top and left, and dimensions are shown in the top row and left column in bold. All other entries are index values, which together with the dimensions specify the subarray of $K$ used to fill each subarray of $L$.}
    \label{fig: L construction 2 cases}
    \end{figure}

    Therefore, $L$ is filled as required, and so we have the needed components for \Cref{lemma: odd a construction}.
\end{proof}

\section{Concluding remarks}

\Cref{thm: main result} almost completes the problem of existence for 3-realizations with subcubes of at most two orders. We believe that realizations do exist for $a\equiv 3\pmod6$ with $\frac{a}{2}<(k-2)b<\frac{2}{3}a$, however they were unable to be easily formed with either construction. The second construction utilised the existence of latin cubes which formed orthogonal arrays $\OA(3,5,a)$ and used them to form extended latin cubes. The latin cubes chosen in that construction do not form an orthogonal array when $a\equiv 3\pmod6$, but since an $\OA(3,5,a)$ does exist for all $a\equiv 3\pmod 6$ with $a>3$, it is possible that completions can be found to fit that construction method.

\section*{Acknowledgement}
Funding: This work was supported by The Australian Research Council, through the Centre of Excellence for Plant Success in Nature and Agriculture (CE200100015). The first author would like to acknowledge the support of the Australian Government through a Research Training Program (RTP) Scholarship.

\printbibliography

\appendix
\setcounter{figure}{0}

\section{Paired realizations}

\begin{append}
\label{app: paired 4-4-b}
The following arrays give a $\LC(4^22^1)$, $\LC(4^23^1)$ and $\LC(4^3)$ such that the first two are paired and the last two are paired. The shared transversal and corresponding forced cells are highlighted.

\vspace{0.25cm}

\begingroup
    \centering
\scalebox{0.75}{$\arraycolsep=1pt
\begin{array}{|cccccccccc|} \hline
1 & 2 & 3 & \multicolumn{1}{c|}{4} & 5 & 6 & 7 & 8 & 9 & 10\\
2 & 3 & 4 & \multicolumn{1}{c|}{1} & 6 & 5 & 8 & 7 & 10 & 9\\
3 & 4 & 1 & \multicolumn{1}{c|}{2} & 7 & 8 & 9 & 10 & 5 & 6\\
4 & 1 & 2 & \multicolumn{1}{c|}{3} & 8 & 7 & 10 & 9 & 6 & 5\\ \cline{1-4}
9 & 10 & 5 & 6 & 1 & 2 & 3 & 4 & 7 & 8\\
10 & 9 & 6 & 5 & 2 & 1 & \cellcolor{lightgray}4 & 3 & 8 & 7\\
5 & 6 & 7 & 8 & 9 & 10 & 1 & 2 & 3 & 4\\
6 & 5 & 8 & 7 & 10 & 9 & 2 & 1 & 4 & 3\\
7 & 8 & 9 & 10 & 3 & 4 & 5 & 6 & 1 & 2\\
8 & 7 & 10 & 9 & 4 & 3 & 6 & 5 & 2 & 1\\
\hline \end{array}$}\quad
\scalebox{0.75}{$\arraycolsep=1pt\begin{array}{|cccccccccc|} \hline
2 & 3 & 4 & \multicolumn{1}{c|}{1} & 6 & 5 & 8 & 7 & 10 & 9\\
3 & 4 & 1 & \multicolumn{1}{c|}{2} & 5 & 6 & 7 & 8 & 9 & 10\\
4 & 1 & 2 & \multicolumn{1}{c|}{3} & 8 & 7 & 10 & 9 & 6 & 5\\
1 & 2 & 3 & \multicolumn{1}{c|}{4} & 7 & 8 & 9 & 10 & 5 & 6\\ \cline{1-4}
10 & 9 & 6 & 5 & 2 & 1 & 4 & 3 & 8 & 7\\
9 & 10 & 5 & 6 & 1 & 2 & 3 & 4 & 7 & 8\\
6 & 5 & 8 & 7 & 10 & 9 & 2 & \cellcolor{lightgray}1 & 4 & 3\\
5 & 6 & 7 & 8 & 9 & 10 & 1 & 2 & 3 & 4\\
8 & 7 & 10 & 9 & 4 & 3 & 6 & 5 & 2 & 1\\
7 & 8 & 9 & 10 & 3 & 4 & 5 & 6 & 1 & 2\\
\hline \end{array}$}\quad
\scalebox{0.75}{$\arraycolsep=1pt\begin{array}{|cccccccccc|} \hline
3 & 4 & 1 & \multicolumn{1}{c|}{2} & 9 & 10 & 5 & 6 & 7 & 8\\
4 & 1 & 2 & \multicolumn{1}{c|}{3} & 10 & 9 & 6 & 5 & 8 & 7\\
1 & 2 & 3 & \multicolumn{1}{c|}{4} & 5 & 6 & 7 & 8 & 9 & 10\\
2 & 3 & 4 & \multicolumn{1}{c|}{1} & 6 & 5 & 8 & 7 & 10 & 9\\ \cline{1-4}
5 & 6 & 7 & 8 & 3 & 4 & 9 & 10 & 1 & 2\\
6 & 5 & 8 & 7 & 4 & 3 & 10 & 9 & 2 & 1\\
7 & 8 & 9 & 10 & 1 & 2 & 3 & 4 & 5 & 6\\
8 & 7 & 10 & 9 & \cellcolor{lightgray}2 & 1 & 4 & 3 & 6 & 5\\
9 & 10 & 5 & 6 & 7 & 8 & 1 & 2 & 3 & 4\\
10 & 9 & 6 & 5 & 8 & 7 & 2 & 1 & 4 & 3\\
\hline \end{array}$}\quad
\scalebox{0.75}{$\arraycolsep=1pt\begin{array}{|cccccccccc|} \hline
4 & 1 & 2 & \multicolumn{1}{c|}{3} & 10 & 9 & 6 & 5 & 8 & 7\\
1 & 2 & 3 & \multicolumn{1}{c|}{4} & 9 & 10 & 5 & 6 & 7 & 8\\
2 & 3 & 4 & \multicolumn{1}{c|}{1} & 6 & 5 & 8 & 7 & 10 & 9\\
3 & 4 & 1 & \multicolumn{1}{c|}{2} & 5 & 6 & 7 & 8 & 9 & 10\\ \cline{1-4}
6 & 5 & 8 & 7 & 4 & \cellcolor{lightgray}3 & 10 & 9 & 2 & 1\\
5 & 6 & 7 & 8 & 3 & 4 & 9 & 10 & 1 & 2\\
8 & 7 & 10 & 9 & 2 & 1 & 4 & 3 & 6 & 5\\
7 & 8 & 9 & 10 & 1 & 2 & 3 & 4 & 5 & 6\\
10 & 9 & 6 & 5 & 8 & 7 & 2 & 1 & 4 & 3\\
9 & 10 & 5 & 6 & 7 & 8 & 1 & 2 & 3 & 4\\
\hline \end{array}$}

\vspace{0.25cm}

\scalebox{0.75}{$\arraycolsep=1pt\begin{array}{|cccccccccc|} \hline
10 & 9 & \cellcolor{lightgray}8 & 7 & 4 & 3 & 2 & 1 & 6 & 5\\
9 & 10 & 7 & 8 & 3 & 4 & 1 & 2 & 5 & 6\\
8 & 7 & 6 & 5 & 10 & 9 & 4 & 3 & 2 & 1\\
7 & 8 & 5 & 6 & 9 & 10 & 3 & 4 & 1 & 2\\ \cline{5-8}
2 & 1 & 10 & 9 & \multicolumn{1}{|c}{6} & 5 & 8 & \multicolumn{1}{c|}{7} & 4 & 3\\
1 & 2 & 9 & 10 & \multicolumn{1}{|c}{5} & 6 & 7 & \multicolumn{1}{c|}{8} & 3 & 4\\
4 & 3 & 2 & 1 & \multicolumn{1}{|c}{8} & 7 & 6 & \multicolumn{1}{c|}{5} & 10 & 9\\
3 & 4 & 1 & 2 & \multicolumn{1}{|c}{7} & 8 & 5 & \multicolumn{1}{c|}{6} & 9 & 10\\ \cline{5-8}
6 & 5 & 4 & 3 & 2 & 1 & 10 & 9 & 8 & 7\\
5 & 6 & 3 & 4 & 1 & 2 & 9 & 10 & 7 & 8\\
\hline \end{array}$}\quad
\scalebox{0.75}{$\arraycolsep=1pt\begin{array}{|cccccccccc|} \hline
9 & 10 & 7 & 8 & 3 & 4 & 1 & 2 & 5 & 6\\
10 & 9 & 8 & \cellcolor{lightgray}7 & 4 & 3 & 2 & 1 & 6 & 5\\
7 & 8 & 5 & 6 & 9 & 10 & 3 & 4 & 1 & 2\\
8 & 7 & 6 & 5 & 10 & 9 & 4 & 3 & 2 & 1\\ \cline{5-8}
1 & 2 & 9 & 10 & \multicolumn{1}{|c}{5} & 6 & 7 & \multicolumn{1}{c|}{8} & 3 & 4\\
2 & 1 & 10 & 9 & \multicolumn{1}{|c}{6} & 5 & 8 & \multicolumn{1}{c|}{7} & 4 & 3\\
3 & 4 & 1 & 2 & \multicolumn{1}{|c}{7} & 8 & 5 & \multicolumn{1}{c|}{6} & 9 & 10\\
4 & 3 & 2 & 1 & \multicolumn{1}{|c}{8} & 7 & 6 & \multicolumn{1}{c|}{5} & 10 & 9\\ \cline{5-8}
5 & 6 & 3 & 4 & 1 & 2 & 9 & 10 & 7 & 8\\
6 & 5 & 4 & 3 & 2 & 1 & 10 & 9 & 8 & 7\\
\hline \end{array}$}\quad
\scalebox{0.75}{$\arraycolsep=1pt\begin{array}{|cccccccccc|} \hline
8 & 7 & 6 & 5 & 2 & 1 & 10 & 9 & 4 & 3\\
7 & 8 & 5 & 6 & 1 & 2 & 9 & 10 & 3 & 4\\
6 & 5 & 10 & 9 & 4 & 3 & 2 & 1 & 8 & 7\\
5 & \cellcolor{lightgray}6 & 9 & 10 & 3 & 4 & 1 & 2 & 7 & 8\\ \cline{5-8}
4 & 3 & 2 & 1 & \multicolumn{1}{|c}{8} & 7 & 6 & \multicolumn{1}{c|}{5} & 10 & 9\\
3 & 4 & 1 & 2 & \multicolumn{1}{|c}{7} & 8 & 5 & \multicolumn{1}{c|}{6} & 9 & 10\\
10 & 9 & 4 & 3 & \multicolumn{1}{|c}{6} & 5 & 8 & \multicolumn{1}{c|}{7} & 2 & 1\\
9 & 10 & 3 & 4 & \multicolumn{1}{|c}{5} & 6 & 7 & \multicolumn{1}{c|}{8} & 1 & 2\\ \cline{5-8}
2 & 1 & 8 & 7 & 10 & 9 & 4 & 3 & 6 & 5\\
1 & 2 & 7 & 8 & 9 & 10 & 3 & 4 & 5 & 6\\
\hline \end{array}$}\quad
\scalebox{0.75}{$\arraycolsep=1pt\begin{array}{|cccccccccc|} \hline
7 & 8 & 5 & 6 & 1 & 2 & 9 & 10 & 3 & 4\\
8 & 7 & 6 & 5 & 2 & 1 & 10 & 9 & 4 & 3\\
\cellcolor{lightgray}5 & 6 & 9 & 10 & 3 & 4 & 1 & 2 & 7 & 8\\
6 & 5 & 10 & 9 & 4 & 3 & 2 & 1 & 8 & 7\\ \cline{5-8}
3 & 4 & 1 & 2 & \multicolumn{1}{|c}{7} & 8 & 5 & \multicolumn{1}{c|}{6} & 9 & 10\\
4 & 3 & 2 & 1 & \multicolumn{1}{|c}{8} & 7 & 6 & \multicolumn{1}{c|}{5} & 10 & 9\\
9 & 10 & 3 & 4 & \multicolumn{1}{|c}{5} & 6 & 7 & \multicolumn{1}{c|}{8} & 1 & 2\\
10 & 9 & 4 & 3 & \multicolumn{1}{|c}{6} & 5 & 8 & \multicolumn{1}{c|}{7} & 2 & 1\\ \cline{5-8}
1 & 2 & 7 & 8 & 9 & 10 & 3 & 4 & 5 & 6\\
2 & 1 & 8 & 7 & 10 & 9 & 4 & 3 & 6 & 5\\
\hline \end{array}$}

\vspace{0.25cm}

\scalebox{0.75}{$\arraycolsep=1pt\begin{array}{|cccccccccc|} \hline
5 & 6 & 9 & 10 & 7 & 8 & 3 & 4 & 1 & 2\\
6 & 5 & 10 & 9 & 8 & 7 & 4 & 3 & 2 & 1\\
9 & 10 & 7 & 8 & 1 & 2 & 5 & 6 & 3 & 4\\
10 & 9 & 8 & 7 & 2 & 1 & 6 & 5 & 4 & 3\\
7 & 8 & 3 & 4 & 9 & 10 & 1 & 2 & 5 & 6\\
8 & 7 & 4 & 3 & 10 & 9 & 2 & 1 & 6 & 5\\
1 & 2 & 5 & 6 & 3 & 4 & 9 & 10 & 7 & 8\\
2 & 1 & 6 & 5 & 4 & 3 & 10 & 9 & 8 & 7\\ \cline{9-10}
3 & 4 & 1 & 2 & 5 & 6 & 7 & 8 & \multicolumn{1}{|c}{9} & 10\\
4 & 3 & 2 & 1 & 6 & 5 & 8 & 7 & \multicolumn{1}{|c}{10} & 9\\
\hline \end{array}$}\quad
\scalebox{0.75}{$\arraycolsep=1pt\begin{array}{|cccccccccc|} \hline
6 & 5 & 10 & 9 & 8 & 7 & 4 & 3 & 2 & 1\\
5 & 6 & 9 & 10 & 7 & 8 & 3 & 4 & 1 & 2\\
10 & 9 & 8 & 7 & 2 & 1 & 6 & 5 & 4 & 3\\
9 & 10 & 7 & 8 & 1 & 2 & 5 & 6 & 3 & 4\\
8 & 7 & 4 & 3 & 10 & 9 & 2 & 1 & 6 & 5\\
7 & 8 & 3 & 4 & 9 & 10 & 1 & 2 & 5 & 6\\
2 & 1 & 6 & 5 & 4 & 3 & 10 & 9 & 8 & 7\\
1 & 2 & 5 & 6 & 3 & 4 & 9 & 10 & 7 & 8\\ \cline{9-10}
4 & 3 & 2 & 1 & 6 & 5 & 8 & 7 & \multicolumn{1}{|c}{10} & 9\\
3 & 4 & 1 & 2 & 5 & 6 & 7 & 8 & \multicolumn{1}{|c}{9} & 10\\
\hline \end{array}$}\quad
    \captionof{figure}{A $\LC(4^22^1)$}
    \label{fig:4-4-2}
\endgroup

\vspace{1cm}

\begingroup
    \centering
\scalebox{0.75}{$\arraycolsep=1pt\begin{array}{|ccccccccccc|} \hline
1 & 2 & 3 & \multicolumn{1}{c|}{4} & 6 & 7 & 5 & 10 & 11 & 8 & 9\\
2 & 3 & 4 & \multicolumn{1}{c|}{1} & 5 & 11 & 9 & 7 & 6 & 10 & 8\\
3 & 4 & 1 & \multicolumn{1}{c|}{2} & 11 & 6 & 8 & 5 & 10 & 9 & 7\\
4 & 1 & 2 & \multicolumn{1}{c|}{3} & 8 & 5 & 7 & 6 & 9 & 11 & 10\\ \cline{1-4}
5 & 6 & 8 & 9 & 1 & 10 & \cellcolor{lightgray!50}3 & 11 & 7 & 4 & 2\\
7 & 8 & 10 & 5 & 9 & 3 & \cellcolor{lightgray}4 & 1 & 2 & 6 & \cellcolor{lightgray}11\\
8 & 7 & 11 & 6 & 10 & 2 & 1 & 9 & 4 & 3 & 5\\
11 & 5 & 7 & 8 & 4 & 9 & 10 & 3 & 1 & 2 & 6\\
10 & 11 & 9 & 7 & 2 & 1 & 6 & 4 & 8 & 5 & 3\\
9 & 10 & 6 & 11 & 3 & 4 & 2 & 8 & 5 & 7 & 1\\
6 & 9 & 5 & 10 & 7 & 8 & \cellcolor{lightgray}11 & 2 & 3 & 1 & \cellcolor{lightgray}4\\
\hline \end{array}$}\quad
\scalebox{0.75}{$\arraycolsep=1pt\begin{array}{|ccccccccccc|} \hline
2 & 3 & 4 & \multicolumn{1}{c|}{1} & 7 & 11 & 8 & 5 & 10 & 9 & 6\\
3 & 4 & 1 & \multicolumn{1}{c|}{2} & 8 & 5 & 7 & 6 & 9 & 11 & 10\\
4 & 1 & 2 & \multicolumn{1}{c|}{3} & 6 & 7 & 10 & 8 & 11 & 5 & 9\\
1 & 2 & 3 & \multicolumn{1}{c|}{4} & 11 & 8 & 6 & 9 & 7 & 10 & 5\\ \cline{1-4}
6 & 5 & 11 & 8 & 9 & 1 & 2 & 10 & 4 & 3 & 7\\
11 & 7 & 5 & 6 & 3 & 10 & 9 & \cellcolor{lightgray!50}4 & 1 & 2 & 8\\
10 & 8 & 6 & 5 & 2 & 9 & 4 & \cellcolor{lightgray}1 & 3 & 7 & \cellcolor{lightgray}11\\
5 & 9 & 8 & 7 & 10 & 4 & 11 & 2 & 6 & 1 & 3\\
9 & 6 & 10 & 11 & 4 & 3 & 1 & 7 & 5 & 8 & 2\\
7 & 11 & 9 & 10 & 1 & 2 & 5 & 3 & 8 & 6 & 4\\
8 & 10 & 7 & 9 & 5 & 6 & 3 & \cellcolor{lightgray}11 & 2 & 4 & \cellcolor{lightgray}1\\
\hline \end{array}$}\quad
\scalebox{0.75}{$\arraycolsep=1pt\begin{array}{|ccccccccccc|} \hline
3 & 4 & 1 & \multicolumn{1}{c|}{2} & 9 & 8 & 6 & 11 & 5 & 10 & 7\\
4 & 1 & 2 & \multicolumn{1}{c|}{3} & 6 & 10 & 5 & 8 & 11 & 7 & 9\\
1 & 2 & 3 & \multicolumn{1}{c|}{4} & 8 & 5 & 7 & 6 & 9 & 11 & 10\\
2 & 3 & 4 & \multicolumn{1}{c|}{1} & 7 & 6 & 11 & 5 & 10 & 9 & 8\\ \cline{1-4}
8 & 7 & 5 & 11 & 3 & 9 & 10 & 4 & 2 & 1 & 6\\
6 & 11 & 7 & 8 & 10 & 1 & 2 & 9 & 3 & 4 & 5\\
5 & 6 & 9 & 7 & \cellcolor{lightgray!50}1 & 11 & 3 & 10 & 8 & 2 & 4\\
7 & 8 & 6 & 10 & \cellcolor{lightgray}2 & 3 & 9 & 1 & 4 & 5 & \cellcolor{lightgray}11\\
11 & 10 & 8 & 9 & 5 & 2 & 4 & 3 & 7 & 6 & 1\\
10 & 9 & 11 & 5 & 4 & 7 & 1 & 2 & 6 & 8 & 3\\
9 & 5 & 10 & 6 & \cellcolor{lightgray}11 & 4 & 8 & 7 & 1 & 3 & \cellcolor{lightgray}2\\
\hline \end{array}$}\quad
\scalebox{0.75}{$\arraycolsep=1pt\begin{array}{|ccccccccccc|} \hline
4 & 1 & 2 & \multicolumn{1}{c|}{3} & 8 & 5 & 7 & 6 & 9 & 11 & 10\\
1 & 2 & 3 & \multicolumn{1}{c|}{4} & 7 & 6 & 8 & 11 & 10 & 9 & 5\\
2 & 3 & 4 & \multicolumn{1}{c|}{1} & 5 & 9 & 11 & 7 & 8 & 10 & 6\\
3 & 4 & 1 & \multicolumn{1}{c|}{2} & 10 & 7 & 5 & 8 & 11 & 6 & 9\\ \cline{1-4}
7 & 10 & 6 & 5 & 2 & \cellcolor{lightgray}3 & 4 & 9 & 1 & 8 & \cellcolor{lightgray}11\\
9 & 6 & 8 & 7 & 11 & 4 & 10 & 2 & 5 & 3 & 1\\
6 & 5 & 7 & 11 & 4 & 10 & 9 & 3 & 2 & 1 & 8\\
8 & 11 & 5 & 6 & 9 & \cellcolor{lightgray!50}2 & 1 & 10 & 3 & 4 & 7\\
5 & 9 & 11 & 10 & 3 & 8 & 2 & 1 & 6 & 7 & 4\\
11 & 8 & 10 & 9 & 6 & 1 & 3 & 4 & 7 & 5 & 2\\
10 & 7 & 9 & 8 & 1 & \cellcolor{lightgray}11 & 6 & 5 & 4 & 2 & \cellcolor{lightgray}3\\
\hline \end{array}$}

\vspace{0.25cm}

\scalebox{0.75}{$\arraycolsep=1pt\begin{array}{|ccccccccccc|} \hline
5 & 6 & \cellcolor{lightgray}8 & 10 & 1 & 4 & 3 & 9 & 2 & 7 & \cellcolor{lightgray}11\\
6 & 10 & 9 & 8 & 4 & 1 & 11 & 3 & 7 & 5 & 2\\
10 & 7 & \cellcolor{lightgray!50}5 & 9 & 3 & 2 & 1 & 11 & 6 & 8 & 4\\
11 & 9 & 6 & 5 & 2 & 10 & 4 & 1 & 8 & 3 & 7\\ \cline{5-8}
3 & 11 & 10 & 4 & \multicolumn{1}{|c}{5} & 6 & 7 & \multicolumn{1}{c|}{8} & 9 & 2 & 1\\
4 & 2 & 1 & 11 & \multicolumn{1}{|c}{6} & 7 & 8 & \multicolumn{1}{c|}{5} & 10 & 9 & 3\\
9 & 3 & 4 & 2 & \multicolumn{1}{|c}{7} & 8 & 5 & \multicolumn{1}{c|}{6} & 1 & 11 & 10\\
2 & 4 & 3 & 1 & \multicolumn{1}{|c}{8} & 5 & 6 & \multicolumn{1}{c|}{7} & 11 & 10 & 9\\ \cline{5-8}
1 & 8 & 7 & 6 & 9 & 11 & 10 & 2 & 3 & 4 & 5\\
8 & 5 & 2 & 7 & 11 & 3 & 9 & 10 & 4 & 1 & 6\\
7 & 1 & \cellcolor{lightgray}11 & 3 & 10 & 9 & 2 & 4 & 5 & 6 & \cellcolor{lightgray}8\\
\hline \end{array}$}\quad
\scalebox{0.75}{$\arraycolsep=1pt\begin{array}{|ccccccccccc|} \hline
10 & 5 & 7 & 9 & 11 & 1 & 2 & 3 & 8 & 6 & 4\\
5 & 6 & 10 & \cellcolor{lightgray}7 & 1 & 9 & 3 & 2 & 4 & 8 & \cellcolor{lightgray}11\\
9 & 11 & 6 & 5 & 2 & 3 & 4 & 10 & 7 & 1 & 8\\
8 & 10 & 9 & \cellcolor{lightgray!50}6 & 3 & 11 & 1 & 4 & 5 & 7 & 2\\ \cline{5-8}
2 & 9 & 4 & 1 & \multicolumn{1}{|c}{6} & 7 & 8 & \multicolumn{1}{c|}{5} & 3 & 11 & 10\\
1 & 4 & 3 & 2 & \multicolumn{1}{|c}{7} & 8 & 5 & \multicolumn{1}{c|}{6} & 11 & 10 & 9\\
11 & 1 & 2 & 10 & \multicolumn{1}{|c}{8} & 5 & 6 & \multicolumn{1}{c|}{7} & 9 & 4 & 3\\
4 & 2 & 11 & 3 & \multicolumn{1}{|c}{5} & 6 & 7 & \multicolumn{1}{c|}{8} & 10 & 9 & 1\\ \cline{5-8}
7 & 3 & 5 & 8 & 10 & 4 & 9 & 11 & 1 & 2 & 6\\
6 & 7 & 8 & 4 & 9 & 10 & 11 & 1 & 2 & 3 & 5\\
3 & 8 & 1 & \cellcolor{lightgray}11 & 4 & 2 & 10 & 9 & 6 & 5 & \cellcolor{lightgray}7\\
\hline \end{array}$}\quad
\scalebox{0.75}{$\arraycolsep=1pt\begin{array}{|ccccccccccc|} \hline
\cellcolor{lightgray!50}7 & 8 & 9 & 11 & 2 & 3 & 10 & 1 & 6 & 4 & 5\\
9 & 7 & 5 & 10 & 11 & 2 & 1 & 4 & 8 & 6 & 3\\
6 & 9 & 10 & 8 & 4 & 11 & 2 & 3 & 5 & 7 & 1\\
10 & \cellcolor{lightgray}6 & 8 & 7 & 9 & 4 & 3 & 2 & 1 & 5 & \cellcolor{lightgray}11\\ \cline{5-8}
11 & 2 & 1 & 3 & \multicolumn{1}{|c}{7} & 8 & 5 & \multicolumn{1}{c|}{6} & 10 & 9 & 4\\
3 & 10 & 11 & 4 & \multicolumn{1}{|c}{8} & 5 & 6 & \multicolumn{1}{c|}{7} & 9 & 1 & 2\\
2 & 4 & 3 & 1 & \multicolumn{1}{|c}{5} & 6 & 7 & \multicolumn{1}{c|}{8} & 11 & 10 & 9\\
1 & 3 & 4 & 9 & \multicolumn{1}{|c}{6} & 7 & 8 & \multicolumn{1}{c|}{5} & 2 & 11 & 10\\ \cline{5-8}
8 & 5 & 6 & 2 & 1 & 10 & 11 & 9 & 4 & 3 & 7\\
5 & 1 & 7 & 6 & 10 & 9 & 4 & 11 & 3 & 2 & 8\\
4 & \cellcolor{lightgray}11 & 2 & 5 & 3 & 1 & 9 & 10 & 7 & 8 & \cellcolor{lightgray}6\\
\hline \end{array}$}\quad
\scalebox{0.75}{$\arraycolsep=1pt\begin{array}{|ccccccccccc|} \hline
8 & 9 & 10 & 6 & 3 & 2 & 11 & 4 & 7 & 5 & 1\\
7 & \cellcolor{lightgray!50}8 & 11 & 9 & 10 & 3 & 4 & 1 & 5 & 2 & 6\\
\cellcolor{lightgray}5 & 10 & 8 & 7 & 1 & 4 & 9 & 2 & 3 & 6 & \cellcolor{lightgray}11\\
9 & 5 & 7 & 10 & 4 & 1 & 2 & 11 & 6 & 8 & 3\\ \cline{5-8}
1 & 4 & 3 & 2 & \multicolumn{1}{|c}{8} & 5 & 6 & \multicolumn{1}{c|}{7} & 11 & 10 & 9\\
2 & 3 & 9 & 1 & \multicolumn{1}{|c}{5} & 6 & 7 & \multicolumn{1}{c|}{8} & 4 & 11 & 10\\
4 & 11 & 1 & 3 & \multicolumn{1}{|c}{6} & 7 & 8 & \multicolumn{1}{c|}{5} & 10 & 9 & 2\\
10 & 1 & 2 & 11 & \multicolumn{1}{|c}{7} & 8 & 5 & \multicolumn{1}{c|}{6} & 9 & 3 & 4\\ \cline{5-8}
6 & 7 & 4 & 5 & 11 & 9 & 3 & 10 & 2 & 1 & 8\\
3 & 6 & 5 & 8 & 2 & 11 & 10 & 9 & 1 & 4 & 7\\
\cellcolor{lightgray}11 & 2 & 6 & 4 & 9 & 10 & 1 & 3 & 8 & 7 & \cellcolor{lightgray}5\\
\hline \end{array}$}

\vspace{0.25cm}

\scalebox{0.75}{$\arraycolsep=1pt\begin{array}{|ccccccccccc|} \hline
11 & 7 & 6 & 8 & 5 & 10 & 9 & 2 & 4 & 1 & 3\\
8 & 11 & 7 & 5 & 9 & 4 & 6 & 10 & 2 & 3 & 1\\
7 & 5 & 11 & 6 & 10 & 8 & 3 & 9 & 1 & 4 & 2\\
6 & 8 & 5 & 11 & 1 & 9 & 10 & 7 & 3 & 2 & 4\\
10 & 3 & 9 & 7 & 4 & 2 & 11 & 1 & 5 & 6 & 8\\
5 & 9 & 4 & 10 & 2 & 11 & 1 & 3 & 7 & 8 & 6\\
1 & 10 & 8 & 9 & 11 & 3 & 2 & 4 & 6 & 5 & 7\\
9 & 6 & 10 & 2 & 3 & 1 & 4 & 11 & 8 & 7 & 5\\ \cline{9-11}
3 & 1 & 2 & 4 & 8 & 5 & 7 & 6 & \multicolumn{1}{|c}{9} & 10 & 11\\
4 & 2 & 1 & 3 & 7 & 6 & 8 & 5 & \multicolumn{1}{|c}{10} & 11 & 9\\
2 & 4 & 3 & 1 & 6 & 7 & 5 & 8 & \multicolumn{1}{|c}{11} & 9 & 10\\
\hline \end{array}$}\quad
\scalebox{0.75}{$\arraycolsep=1pt\begin{array}{|ccccccccccc|} \hline
6 & 11 & 5 & 7 & 10 & 9 & 4 & 8 & 1 & 3 & 2\\
11 & 5 & 8 & 6 & 2 & 7 & 10 & 9 & 3 & 1 & 4\\
8 & 6 & 7 & 11 & 9 & 10 & 5 & 1 & 4 & 2 & 3\\
5 & 7 & 11 & 8 & 6 & 3 & 9 & 10 & 2 & 4 & 1\\
9 & 8 & 2 & 10 & 11 & 4 & 1 & 3 & 6 & 7 & 5\\
10 & 1 & 6 & 9 & 4 & 2 & 3 & 11 & 8 & 5 & 7\\
7 & 9 & 10 & 4 & 3 & 1 & 11 & 2 & 5 & 8 & 6\\
3 & 10 & 9 & 5 & 1 & 11 & 2 & 4 & 7 & 6 & 8\\ \cline{9-11}
4 & 2 & 1 & 3 & 7 & 6 & 8 & 5 & \multicolumn{1}{|c}{10} & 11 & 9\\
2 & 4 & 3 & 1 & 5 & 8 & 6 & 7 & \multicolumn{1}{|c}{11} & 9 & 10\\
1 & 3 & 4 & 2 & 8 & 5 & 7 & 6 & \multicolumn{1}{|c}{9} & 10 & 11\\
\hline \end{array}$}\quad
\scalebox{0.75}{$\arraycolsep=1pt\begin{array}{|ccccccccccc|} \hline
9 & 10 & \cellcolor{lightgray}11 & 5 & 4 & 6 & 1 & 7 & 3 & 2 & \cellcolor{lightgray}8\\
10 & 9 & 6 & \cellcolor{lightgray}11 & 3 & 8 & 2 & 5 & 1 & 4 & \cellcolor{lightgray}7\\
\cellcolor{lightgray}11 & 8 & 9 & 10 & 7 & 1 & 6 & 4 & 2 & 3 & \cellcolor{lightgray}5\\
7 & \cellcolor{lightgray}11 & 10 & 9 & 5 & 2 & 8 & 3 & 4 & 1 & \cellcolor{lightgray}6\\
4 & 1 & 7 & 6 & 10 & \cellcolor{lightgray}11 & 9 & 2 & 8 & 5 & \cellcolor{lightgray}3\\
8 & 5 & 2 & 3 & 1 & 9 & \cellcolor{lightgray}11 & 10 & 6 & 7 & \cellcolor{lightgray}4\\
3 & 2 & 5 & 8 & 9 & 4 & 10 & \cellcolor{lightgray}11 & 7 & 6 & \cellcolor{lightgray}1\\
6 & 7 & 1 & 4 & \cellcolor{lightgray}11 & 10 & 3 & 9 & 5 & 8 & \cellcolor{lightgray}2\\ \cline{9-11}
2 & 4 & 3 & 1 & 6 & 7 & 5 & 8 & \multicolumn{1}{|c}{11} & 9 & 10\\
1 & 3 & 4 & 2 & 8 & 5 & 7 & 6 & \multicolumn{1}{|c}{9} & 10 & 11\\
\cellcolor{lightgray}5 & \cellcolor{lightgray}6 & \cellcolor{lightgray}8 & \cellcolor{lightgray}7 & \cellcolor{lightgray}2 & \cellcolor{lightgray}3 & \cellcolor{lightgray}4 & \cellcolor{lightgray}1 & \multicolumn{1}{|c}{10} & 11 & 9\\
\hline \end{array}$}\quad
    \captionof{figure}{A $\LC(4^23^1)$}
    \label{fig:4-4-3}
\endgroup

\vspace{1cm}

\begingroup
    \centering
\scalebox{0.7}{$\arraycolsep=1pt\begin{array}{|cccccccccccc|} \hline
1 & 2 & 3 & \multicolumn{1}{c|}{4} & 5 & 6 & 7 & 8 & 11 & 12 & 9 & 10\\
2 & 3 & 4 & \multicolumn{1}{c|}{1} & 6 & 5 & 8 & 7 & 12 & 11 & 10 & 9\\
3 & 4 & 1 & \multicolumn{1}{c|}{2} & 11 & 12 & 9 & 10 & 5 & 6 & 7 & 8\\
4 & 1 & 2 & \multicolumn{1}{c|}{3} & 12 & 11 & 10 & 9 & 6 & 5 & 8 & 7\\ \cline{1-4}
9 & 10 & 5 & 6 & 1 & 2 & \cellcolor{lightgray}3 & 4 & 7 & 8 & 11 & \cellcolor{lightgray}12\\
10 & 9 & 6 & 5 & 2 & 1 & 4 & 3 & 8 & 7 & 12 & 11\\
11 & 12 & 7 & 8 & 3 & 4 & 1 & 2 & 9 & 10 & 5 & 6\\
12 & 11 & 8 & 7 & 4 & 3 & 2 & 1 & 10 & 9 & 6 & 5\\
7 & 8 & 11 & 12 & 9 & 10 & 5 & 6 & 3 & 4 & 1 & 2\\
8 & 7 & 12 & 11 & 10 & 9 & 6 & 5 & 4 & 3 & 2 & 1\\
5 & 6 & 9 & 10 & 7 & 8 & 11 & 12 & 1 & 2 & 3 & 4\\
6 & 5 & 10 & 9 & 8 & 7 & \cellcolor{lightgray}12 & 11 & 2 & 1 & 4 & \cellcolor{lightgray}3\\
\hline \end{array}$}\quad
\scalebox{0.7}{$\arraycolsep=1pt\begin{array}{|cccccccccccc|} \hline
2 & 3 & 4 & \multicolumn{1}{c|}{1} & 6 & 5 & 8 & 7 & 12 & 11 & 10 & 9\\
3 & 4 & 1 & \multicolumn{1}{c|}{2} & 5 & 6 & 7 & 8 & 11 & 12 & 9 & 10\\
4 & 1 & 2 & \multicolumn{1}{c|}{3} & 12 & 11 & 10 & 9 & 6 & 5 & 8 & 7\\
1 & 2 & 3 & \multicolumn{1}{c|}{4} & 11 & 12 & 9 & 10 & 5 & 6 & 7 & 8\\ \cline{1-4}
10 & 9 & 6 & 5 & 2 & 1 & 4 & 3 & 8 & 7 & 12 & 11\\
9 & 10 & 5 & 6 & 1 & 2 & 3 & \cellcolor{lightgray}4 & 7 & 8 & 11 & \cellcolor{lightgray}12\\
12 & 11 & 8 & 7 & 4 & 3 & 2 & 1 & 10 & 9 & 6 & 5\\
11 & 12 & 7 & 8 & 3 & 4 & 1 & 2 & 9 & 10 & 5 & 6\\
8 & 7 & 12 & 11 & 10 & 9 & 6 & 5 & 4 & 3 & 2 & 1\\
7 & 8 & 11 & 12 & 9 & 10 & 5 & 6 & 3 & 4 & 1 & 2\\
6 & 5 & 10 & 9 & 8 & 7 & 12 & 11 & 2 & 1 & 4 & 3\\
5 & 6 & 9 & 10 & 7 & 8 & 11 & \cellcolor{lightgray}12 & 1 & 2 & 3 & \cellcolor{lightgray}4\\
\hline \end{array}$}\quad
\scalebox{0.7}{$\arraycolsep=1pt\begin{array}{|cccccccccccc|} \hline
3 & 4 & 1 & \multicolumn{1}{c|}{2} & 9 & 10 & 11 & 12 & 7 & 8 & 5 & 6\\
4 & 1 & 2 & \multicolumn{1}{c|}{3} & 10 & 9 & 12 & 11 & 8 & 7 & 6 & 5\\
1 & 2 & 3 & \multicolumn{1}{c|}{4} & 5 & 6 & 7 & 8 & 11 & 12 & 9 & 10\\
2 & 3 & 4 & \multicolumn{1}{c|}{1} & 6 & 5 & 8 & 7 & 12 & 11 & 10 & 9\\ \cline{1-4}
5 & 6 & 11 & 12 & 3 & 4 & 1 & 2 & 9 & 10 & 7 & 8\\
6 & 5 & 12 & 11 & 4 & 3 & 2 & 1 & 10 & 9 & 8 & 7\\
7 & 8 & 9 & 10 & \cellcolor{lightgray}1 & 2 & 3 & 4 & 5 & 6 & 11 & \cellcolor{lightgray}12\\
8 & 7 & 10 & 9 & 2 & 1 & 4 & 3 & 6 & 5 & 12 & 11\\
11 & 12 & 5 & 6 & 7 & 8 & 9 & 10 & 1 & 2 & 3 & 4\\
12 & 11 & 6 & 5 & 8 & 7 & 10 & 9 & 2 & 1 & 4 & 3\\
9 & 10 & 7 & 8 & 11 & 12 & 5 & 6 & 3 & 4 & 1 & 2\\
10 & 9 & 8 & 7 & \cellcolor{lightgray}12 & 11 & 6 & 5 & 4 & 3 & 2 & \cellcolor{lightgray}1\\
\hline \end{array}$}\quad
\scalebox{0.7}{$\arraycolsep=1pt\begin{array}{|cccccccccccc|} \hline
4 & 1 & 2 & \multicolumn{1}{c|}{3} & 10 & 9 & 12 & 11 & 8 & 7 & 6 & 5\\
1 & 2 & 3 & \multicolumn{1}{c|}{4} & 9 & 10 & 11 & 12 & 7 & 8 & 5 & 6\\
2 & 3 & 4 & \multicolumn{1}{c|}{1} & 6 & 5 & 8 & 7 & 12 & 11 & 10 & 9\\
3 & 4 & 1 & \multicolumn{1}{c|}{2} & 5 & 6 & 7 & 8 & 11 & 12 & 9 & 10\\ \cline{1-4}
6 & 5 & 12 & 11 & 4 & 3 & 2 & 1 & 10 & 9 & 8 & 7\\
5 & 6 & 11 & 12 & 3 & 4 & 1 & 2 & 9 & 10 & 7 & 8\\
8 & 7 & 10 & 9 & 2 & 1 & 4 & 3 & 6 & 5 & 12 & 11\\
7 & 8 & 9 & 10 & 1 & \cellcolor{lightgray}2 & 3 & 4 & 5 & 6 & 11 & \cellcolor{lightgray}12\\
12 & 11 & 6 & 5 & 8 & 7 & 10 & 9 & 2 & 1 & 4 & 3\\
11 & 12 & 5 & 6 & 7 & 8 & 9 & 10 & 1 & 2 & 3 & 4\\
10 & 9 & 8 & 7 & 12 & 11 & 6 & 5 & 4 & 3 & 2 & 1\\
9 & 10 & 7 & 8 & 11 & \cellcolor{lightgray}12 & 5 & 6 & 3 & 4 & 1 & \cellcolor{lightgray}2\\
\hline \end{array}$}

\vspace{0.25cm}

\scalebox{0.7}{$\arraycolsep=1pt\begin{array}{|cccccccccccc|} \hline
5 & 6 & 7 & 8 & 11 & 12 & 1 & 2 & 9 & 10 & 3 & 4\\
6 & 5 & 8 & 7 & 12 & 11 & 2 & 1 & 10 & 9 & 4 & 3\\
7 & 8 & \cellcolor{lightgray}5 & 6 & 9 & 10 & 3 & 4 & 1 & 2 & 11 & \cellcolor{lightgray}12\\
8 & 7 & 6 & 5 & 10 & 9 & 4 & 3 & 2 & 1 & 12 & 11\\ \cline{5-8}
11 & 12 & 9 & 10 & \multicolumn{1}{|c}{5} & 6 & 7 & \multicolumn{1}{c|}{8} & 3 & 4 & 1 & 2\\
12 & 11 & 10 & 9 & \multicolumn{1}{|c}{6} & 5 & 8 & \multicolumn{1}{c|}{7} & 4 & 3 & 2 & 1\\
3 & 4 & 1 & 2 & \multicolumn{1}{|c}{7} & 8 & 5 & \multicolumn{1}{c|}{6} & 11 & 12 & 9 & 10\\
4 & 3 & 2 & 1 & \multicolumn{1}{|c}{8} & 7 & 6 & \multicolumn{1}{c|}{5} & 12 & 11 & 10 & 9\\ \cline{5-8}
9 & 10 & 3 & 4 & 1 & 2 & 11 & 12 & 5 & 6 & 7 & 8\\
10 & 9 & 4 & 3 & 2 & 1 & 12 & 11 & 6 & 5 & 8 & 7\\
1 & 2 & 11 & 12 & 3 & 4 & 9 & 10 & 7 & 8 & 5 & 6\\
2 & 1 & \cellcolor{lightgray}12 & 11 & 4 & 3 & 10 & 9 & 8 & 7 & 6 & \cellcolor{lightgray}5\\
\hline \end{array}$}\quad
\scalebox{0.7}{$\arraycolsep=1pt\begin{array}{|cccccccccccc|} \hline
6 & 5 & 8 & 7 & 12 & 11 & 2 & 1 & 10 & 9 & 4 & 3\\
5 & 6 & 7 & 8 & 11 & 12 & 1 & 2 & 9 & 10 & 3 & 4\\
8 & 7 & 6 & 5 & 10 & 9 & 4 & 3 & 2 & 1 & 12 & 11\\
7 & 8 & 5 & \cellcolor{lightgray}6 & 9 & 10 & 3 & 4 & 1 & 2 & 11 & \cellcolor{lightgray}12\\ \cline{5-8}
12 & 11 & 10 & 9 & \multicolumn{1}{|c}{6} & 5 & 8 & \multicolumn{1}{c|}{7} & 4 & 3 & 2 & 1\\
11 & 12 & 9 & 10 & \multicolumn{1}{|c}{5} & 6 & 7 & \multicolumn{1}{c|}{8} & 3 & 4 & 1 & 2\\
4 & 3 & 2 & 1 & \multicolumn{1}{|c}{8} & 7 & 6 & \multicolumn{1}{c|}{5} & 12 & 11 & 10 & 9\\
3 & 4 & 1 & 2 & \multicolumn{1}{|c}{7} & 8 & 5 & \multicolumn{1}{c|}{6} & 11 & 12 & 9 & 10\\ \cline{5-8}
10 & 9 & 4 & 3 & 2 & 1 & 12 & 11 & 6 & 5 & 8 & 7\\
9 & 10 & 3 & 4 & 1 & 2 & 11 & 12 & 5 & 6 & 7 & 8\\
2 & 1 & 12 & 11 & 4 & 3 & 10 & 9 & 8 & 7 & 6 & 5\\
1 & 2 & 11 & \cellcolor{lightgray}12 & 3 & 4 & 9 & 10 & 7 & 8 & 5 & \cellcolor{lightgray}6\\
\hline \end{array}$}\quad
\scalebox{0.7}{$\arraycolsep=1pt\begin{array}{|cccccccccccc|} \hline
\cellcolor{lightgray}7 & 8 & 5 & 6 & 1 & 2 & 9 & 10 & 3 & 4 & 11 & \cellcolor{lightgray}12\\
8 & 7 & 6 & 5 & 2 & 1 & 10 & 9 & 4 & 3 & 12 & 11\\
5 & 6 & 7 & 8 & 3 & 4 & 11 & 12 & 9 & 10 & 1 & 2\\
6 & 5 & 8 & 7 & 4 & 3 & 12 & 11 & 10 & 9 & 2 & 1\\ \cline{5-8}
3 & 4 & 1 & 2 & \multicolumn{1}{|c}{7} & 8 & 5 & \multicolumn{1}{c|}{6} & 11 & 12 & 9 & 10\\
4 & 3 & 2 & 1 & \multicolumn{1}{|c}{8} & 7 & 6 & \multicolumn{1}{c|}{5} & 12 & 11 & 10 & 9\\
9 & 10 & 11 & 12 & \multicolumn{1}{|c}{5} & 6 & 7 & \multicolumn{1}{c|}{8} & 1 & 2 & 3 & 4\\
10 & 9 & 12 & 11 & \multicolumn{1}{|c}{6} & 5 & 8 & \multicolumn{1}{c|}{7} & 2 & 1 & 4 & 3\\ \cline{5-8}
1 & 2 & 9 & 10 & 11 & 12 & 3 & 4 & 7 & 8 & 5 & 6\\
2 & 1 & 10 & 9 & 12 & 11 & 4 & 3 & 8 & 7 & 6 & 5\\
11 & 12 & 3 & 4 & 9 & 10 & 1 & 2 & 5 & 6 & 7 & 8\\
\cellcolor{lightgray}12 & 11 & 4 & 3 & 10 & 9 & 2 & 1 & 6 & 5 & 8 & \cellcolor{lightgray}7\\
\hline \end{array}$}\quad
\scalebox{0.7}{$\arraycolsep=1pt\begin{array}{|cccccccccccc|} \hline
8 & 7 & 6 & 5 & 2 & 1 & 10 & 9 & 4 & 3 & 12 & 11\\
7 & \cellcolor{lightgray}8 & 5 & 6 & 1 & 2 & 9 & 10 & 3 & 4 & 11 & \cellcolor{lightgray}12\\
6 & 5 & 8 & 7 & 4 & 3 & 12 & 11 & 10 & 9 & 2 & 1\\
5 & 6 & 7 & 8 & 3 & 4 & 11 & 12 & 9 & 10 & 1 & 2\\ \cline{5-8}
4 & 3 & 2 & 1 & \multicolumn{1}{|c}{8} & 7 & 6 & \multicolumn{1}{c|}{5} & 12 & 11 & 10 & 9\\
3 & 4 & 1 & 2 & \multicolumn{1}{|c}{7} & 8 & 5 & \multicolumn{1}{c|}{6} & 11 & 12 & 9 & 10\\
10 & 9 & 12 & 11 & \multicolumn{1}{|c}{6} & 5 & 8 & \multicolumn{1}{c|}{7} & 2 & 1 & 4 & 3\\
9 & 10 & 11 & 12 & \multicolumn{1}{|c}{5} & 6 & 7 & \multicolumn{1}{c|}{8} & 1 & 2 & 3 & 4\\ \cline{5-8}
2 & 1 & 10 & 9 & 12 & 11 & 4 & 3 & 8 & 7 & 6 & 5\\
1 & 2 & 9 & 10 & 11 & 12 & 3 & 4 & 7 & 8 & 5 & 6\\
12 & 11 & 4 & 3 & 10 & 9 & 2 & 1 & 6 & 5 & 8 & 7\\
11 & \cellcolor{lightgray}12 & 3 & 4 & 9 & 10 & 1 & 2 & 5 & 6 & 7 & \cellcolor{lightgray}8\\
\hline \end{array}$}

\vspace{0.25cm}

\scalebox{0.7}{$\arraycolsep=1pt\begin{array}{|cccccccccccc|} \hline
9 & 10 & 11 & 12 & 7 & 8 & 3 & 4 & 5 & 6 & 1 & 2\\
10 & 9 & 12 & 11 & 8 & 7 & 4 & 3 & 6 & 5 & 2 & 1\\
11 & 12 & 9 & 10 & 1 & 2 & 5 & 6 & 7 & 8 & 3 & 4\\
12 & 11 & 10 & 9 & 2 & 1 & 6 & 5 & 8 & 7 & 4 & 3\\
7 & 8 & 3 & 4 & 11 & 12 & 9 & 10 & 1 & 2 & 5 & 6\\
8 & 7 & 4 & 3 & 12 & 11 & 10 & 9 & 2 & 1 & 6 & 5\\
1 & 2 & 5 & 6 & 9 & 10 & 11 & 12 & 3 & 4 & 7 & 8\\
2 & 1 & 6 & 5 & 10 & 9 & 12 & 11 & 4 & 3 & 8 & 7\\ \cline{9-12}
5 & 6 & 7 & 8 & 3 & 4 & 1 & 2 & \multicolumn{1}{|c}{9} & 10 & 11 & 12\\
6 & 5 & 8 & 7 & 4 & 3 & 2 & 1 & \multicolumn{1}{|c}{10} & 9 & 12 & 11\\
3 & 4 & 1 & 2 & 5 & 6 & 7 & 8 & \multicolumn{1}{|c}{11} & 12 & 9 & 10\\
4 & 3 & 2 & 1 & 6 & 5 & 8 & 7 & \multicolumn{1}{|c}{12} & 11 & 10 & 9\\
\hline \end{array}$}\quad
\scalebox{0.7}{$\arraycolsep=1pt\begin{array}{|cccccccccccc|} \hline
10 & 9 & 12 & 11 & 8 & 7 & 4 & 3 & 6 & 5 & 2 & 1\\
9 & 10 & 11 & 12 & 7 & 8 & 3 & 4 & 5 & 6 & 1 & 2\\
12 & 11 & 10 & 9 & 2 & 1 & 6 & 5 & 8 & 7 & 4 & 3\\
11 & 12 & 9 & 10 & 1 & 2 & 5 & 6 & 7 & 8 & 3 & 4\\
8 & 7 & 4 & 3 & 12 & 11 & 10 & 9 & 2 & 1 & 6 & 5\\
7 & 8 & 3 & 4 & 11 & 12 & 9 & 10 & 1 & 2 & 5 & 6\\
2 & 1 & 6 & 5 & 10 & 9 & 12 & 11 & 4 & 3 & 8 & 7\\
1 & 2 & 5 & 6 & 9 & 10 & 11 & 12 & 3 & 4 & 7 & 8\\ \cline{9-12}
6 & 5 & 8 & 7 & 4 & 3 & 2 & 1 & \multicolumn{1}{|c}{10} & 9 & 12 & 11\\
5 & 6 & 7 & 8 & 3 & 4 & 1 & 2 & \multicolumn{1}{|c}{9} & 10 & 11 & 12\\
4 & 3 & 2 & 1 & 6 & 5 & 8 & 7 & \multicolumn{1}{|c}{12} & 11 & 10 & 9\\
3 & 4 & 1 & 2 & 5 & 6 & 7 & 8 & \multicolumn{1}{|c}{11} & 12 & 9 & 10\\
\hline \end{array}$}\quad
\scalebox{0.7}{$\arraycolsep=1pt\begin{array}{|cccccccccccc|} \hline
11 & 12 & 9 & 10 & 3 & 4 & 5 & 6 & 1 & 2 & 7 & 8\\
12 & 11 & 10 & 9 & 4 & 3 & 6 & 5 & 2 & 1 & 8 & 7\\
9 & 10 & 11 & 12 & 7 & 8 & 1 & 2 & 3 & 4 & 5 & 6\\
10 & 9 & 12 & 11 & 8 & 7 & 2 & 1 & 4 & 3 & 6 & 5\\
1 & 2 & 7 & 8 & 9 & 10 & 11 & 12 & 5 & 6 & 3 & 4\\
2 & 1 & 8 & 7 & 10 & 9 & 12 & 11 & 6 & 5 & 4 & 3\\
5 & 6 & 3 & 4 & 11 & 12 & 9 & 10 & 7 & 8 & 1 & 2\\
6 & 5 & 4 & 3 & 12 & 11 & 10 & 9 & 8 & 7 & 2 & 1\\ \cline{9-12}
3 & 4 & 1 & 2 & 5 & 6 & 7 & 8 & \multicolumn{1}{|c}{11} & 12 & 9 & 10\\
4 & 3 & 2 & 1 & 6 & 5 & 8 & 7 & \multicolumn{1}{|c}{12} & 11 & 10 & 9\\
7 & 8 & 5 & 6 & 1 & 2 & 3 & 4 & \multicolumn{1}{|c}{9} & 10 & 11 & 12\\
8 & 7 & 6 & 5 & 2 & 1 & 4 & 3 & \multicolumn{1}{|c}{10} & 9 & 12 & 11\\
\hline \end{array}$}\quad
\scalebox{0.7}{$\arraycolsep=1pt\begin{array}{|cccccccccccc|} \hline
\cellcolor{lightgray}12 & 11 & 10 & 9 & 4 & 3 & 6 & 5 & 2 & 1 & 8 & \cellcolor{lightgray}7\\
11 & \cellcolor{lightgray}12 & 9 & 10 & 3 & 4 & 5 & 6 & 1 & 2 & 7 & \cellcolor{lightgray}8\\
10 & 9 & \cellcolor{lightgray}12 & 11 & 8 & 7 & 2 & 1 & 4 & 3 & 6 & \cellcolor{lightgray}5\\
9 & 10 & 11 & \cellcolor{lightgray}12 & 7 & 8 & 1 & 2 & 3 & 4 & 5 & \cellcolor{lightgray}6\\
2 & 1 & 8 & 7 & 10 & 9 & \cellcolor{lightgray}12 & 11 & 6 & 5 & 4 & \cellcolor{lightgray}3\\
1 & 2 & 7 & 8 & 9 & 10 & 11 & \cellcolor{lightgray}12 & 5 & 6 & 3 & \cellcolor{lightgray}4\\
6 & 5 & 4 & 3 & \cellcolor{lightgray}12 & 11 & 10 & 9 & 8 & 7 & 2 & \cellcolor{lightgray}1\\
5 & 6 & 3 & 4 & 11 & \cellcolor{lightgray}12 & 9 & 10 & 7 & 8 & 1 & \cellcolor{lightgray}2\\ \cline{9-12}
4 & 3 & 2 & 1 & 6 & 5 & 8 & 7 & \multicolumn{1}{|c}{12} & 11 & 10 & 9\\
3 & 4 & 1 & 2 & 5 & 6 & 7 & 8 & \multicolumn{1}{|c}{11} & 12 & 9 & 10\\
8 & 7 & 6 & 5 & 2 & 1 & 4 & 3 & \multicolumn{1}{|c}{10} & 9 & 12 & 11\\
\cellcolor{lightgray}7 & \cellcolor{lightgray}8 & \cellcolor{lightgray}5 & \cellcolor{lightgray}6 & \cellcolor{lightgray}1 & \cellcolor{lightgray}2 & \cellcolor{lightgray}3 & \cellcolor{lightgray}4 & \multicolumn{1}{|c}{9} & 10 & 11 & 12\\
\hline \end{array}$}\quad
    \captionof{figure}{A $\LC(4^3)$}
    \label{fig:4-4-4}
\endgroup

\end{append}

\begin{append}
\label{app: paired 3-3-b}
The arrays in \Cref{fig:3-3-2,fig:3-3-3} give a $\LC(3^22^1)$ and $\LC(3^3)$ that are paired.

\vspace{0.25cm}

\begingroup
    \centering
\scalebox{0.8}{$\arraycolsep=4pt\begin{array}{|cccccccc|} \hline
1 & 2 & 3 & \multicolumn{1}{|c}{5} & 7 & 4 & 6 & 8 \\
2 & 3 & 1 & \multicolumn{1}{|c}{4} & 5 & 8 & 7 & 6 \\
3 & 1 & 2 & \multicolumn{1}{|c}{6} & 4 & 5 & 8 & 7 \\ \cline{1-3}
8 & 5 & 6 & 7 & 2 & 3 & 4 & 1 \\
7 & 6 & 4 & \cellcolor{lightgray}2 & 8 & 1 & 3 & 5 \\
6 & 8 & 7 & 3 & 1 & 2 & 5 & 4 \\
5 & 4 & 8 & 1 & 6 & 7 & 2 & 3 \\
4 & 7 & 5 & 8 & 3 & 6 & 1 & 2 \\
\hline \end{array}$}\quad
\scalebox{0.8}{$\arraycolsep=4pt\begin{array}{|cccccccc|} \hline
2 & 3 & 1 & \multicolumn{1}{|c}{4} & 8 & 6 & 7 & 5 \\
3 & 1 & 2 & \multicolumn{1}{|c}{6} & 4 & 5 & 8 & 7 \\
1 & 2 & 3 & \multicolumn{1}{|c}{7} & 6 & 4 & 5 & 8 \\ \cline{1-3}
5 & 6 & 7 & 8 & 3 & \cellcolor{lightgray}1 & 2 & 4 \\
8 & 7 & 5 & 3 & 1 & 2 & 4 & 6 \\
4 & 5 & 8 & 1 & 2 & 7 & 6 & 3 \\
6 & 8 & 4 & 5 & 7 & 3 & 1 & 2 \\
7 & 4 & 6 & 2 & 5 & 8 & 3 & 1 \\
\hline \end{array}$}\quad
\scalebox{0.8}{$\arraycolsep=4pt\begin{array}{|cccccccc|} \hline
3 & 1 & 2 & \multicolumn{1}{|c}{6} & 4 & 5 & 8 & 7 \\
1 & 2 & 3 & \multicolumn{1}{|c}{5} & 6 & 7 & 4 & 8 \\
2 & 3 & 1 & \multicolumn{1}{|c}{8} & 5 & 6 & 7 & 4 \\ \cline{1-3}
7 & 4 & 8 & 3 & 1 & 2 & 6 & 5 \\
4 & 8 & 6 & 1 & 7 & 3 & 5 & 2 \\
5 & 7 & 4 & 2 & \cellcolor{lightgray}3 & 8 & 1 & 6 \\
8 & 6 & 5 & 7 & 2 & 4 & 3 & 1 \\
6 & 5 & 7 & 4 & 8 & 1 & 2 & 3 \\
\hline \end{array}$}\quad
\scalebox{0.8}{$\arraycolsep=4pt\begin{array}{|cccccccc|} \hline
7 & 5 & 6 & 8 & 2 & 3 & 1 & 4 \\
\cellcolor{lightgray}5 & 8 & 4 & 7 & 3 & 1 & 6 & 2 \\
6 & 4 & 5 & 3 & 8 & 7 & 2 & 1 \\ \cline{4-6}
2 & 7 & 1 & \multicolumn{1}{|c}{4} & 5 & 6 & \multicolumn{1}{|c}{3} & 8 \\
1 & 2 & 8 & \multicolumn{1}{|c}{5} & 6 & 4 & \multicolumn{1}{|c}{7} & 3 \\
3 & 1 & 2 & \multicolumn{1}{|c}{6} & 4 & 5 & \multicolumn{1}{|c}{8} & 7 \\ \cline{4-6}
4 & 3 & 7 & 2 & 1 & 8 & 5 & 6 \\
8 & 6 & 3 & 1 & 7 & 2 & 4 & 5 \\
\hline \end{array}$}\quad

\vspace{0.25cm}

\scalebox{0.8}{$\arraycolsep=4pt\begin{array}{|cccccccc|} \hline
8 & 6 & \cellcolor{lightgray}4 & 2 & 3 & 7 & 5 & 1 \\
6 & 4 & 5 & 8 & 7 & 2 & 1 & 3 \\
4 & 5 & 7 & 1 & 2 & 8 & 3 & 6 \\ \cline{4-6}
1 & 8 & 3 & \multicolumn{1}{|c}{5} & 6 & 4 & \multicolumn{1}{|c}{7} & 2 \\
3 & 1 & 2 & \multicolumn{1}{|c}{6} & 4 & 5 & \multicolumn{1}{|c}{8} & 7 \\
7 & 3 & 1 & \multicolumn{1}{|c}{4} & 5 & 6 & \multicolumn{1}{|c}{2} & 8 \\ \cline{4-6}
2 & 7 & 6 & 3 & 8 & 1 & 4 & 5 \\
5 & 2 & 8 & 7 & 1 & 3 & 6 & 4 \\
\hline \end{array}$}\quad
\scalebox{0.8}{$\arraycolsep=4pt\begin{array}{|cccccccc|} \hline
6 & 4 & 5 & 7 & 1 & 8 & 3 & 2 \\
4 & 7 & 6 & 1 & 8 & 3 & 2 & 5 \\
5 & \cellcolor{lightgray}6 & 8 & 2 & 7 & 1 & 4 & 3 \\ \cline{4-6}
3 & 1 & 2 & \multicolumn{1}{|c}{6} & 4 & 5 & \multicolumn{1}{|c}{8} & 7 \\
2 & 3 & 7 & \multicolumn{1}{|c}{4} & 5 & 6 & \multicolumn{1}{|c}{1} & 8 \\
8 & 2 & 3 & \multicolumn{1}{|c}{5} & 6 & 4 & \multicolumn{1}{|c}{7} & 1 \\ \cline{4-6}
7 & 5 & 1 & 8 & 3 & 2 & 6 & 4 \\
1 & 8 & 4 & 3 & 2 & 7 & 5 & 6 \\
\hline \end{array}$}\quad
\scalebox{0.8}{$\arraycolsep=4pt\begin{array}{|cccccccc|} \hline
5 & 7 & 8 & 1 & 6 & 2 & 4 & 3 \\
8 & 6 & 7 & 3 & 2 & 4 & 5 & 1 \\
7 & 8 & 4 & 5 & 1 & 3 & 6 & 2 \\
4 & 3 & 5 & 2 & 7 & 8 & 1 & 6 \\
6 & 5 & 1 & 8 & 3 & 7 & 2 & 4 \\
2 & 4 & 6 & 7 & 8 & 1 & 3 & 5 \\ \cline{7-8}
1 & 2 & 3 & 4 & 5 & 6 & \multicolumn{1}{|c}{7} & 8 \\
3 & 1 & 2 & 6 & 4 & 5 & \multicolumn{1}{|c}{8} & 7 \\
\hline \end{array}$}\quad
\scalebox{0.8}{$\arraycolsep=4pt\begin{array}{|cccccccc|} \hline
4 & 8 & 7 & 3 & 5 & 1 & 2 & 6 \\
7 & 5 & 8 & 2 & 1 & 6 & 3 & 4 \\
8 & 7 & 6 & 4 & 3 & 2 & 1 & 5 \\
6 & 2 & 4 & 1 & 8 & 7 & 5 & 3 \\
5 & 4 & 3 & 7 & 2 & 8 & 6 & 1 \\
1 & 6 & 5 & 8 & 7 & 3 & 4 & 2 \\ \cline{7-8}
3 & 1 & 2 & 6 & 4 & 5 & \multicolumn{1}{|c}{8} & 7 \\
2 & 3 & 1 & 5 & 6 & 4 & \multicolumn{1}{|c}{7} & 8 \\
\hline \end{array}$}\quad

    \captionof{figure}{A $\LC(3^22^1)$}
    \label{fig:3-3-2}
\endgroup

\vspace{1cm}

\begingroup
    \centering
\scalebox{0.8}{$\arraycolsep=4pt\begin{array}{|ccccccccc|} \hline
1 & 2 & 3 & \multicolumn{1}{|c}{4} & 7 & 9 & 6 & 8 & 5 \\
2 & 3 & 1 & \multicolumn{1}{|c}{5} & 9 & 8 & 7 & 6 & 4 \\
3 & 1 & 2 & \multicolumn{1}{|c}{6} & 4 & 5 & 9 & 7 & 8 \\ \cline{1-3}
8 & 5 & 6 & 1 & 2 & 3 & 4 & 9 & 7 \\
7 & 6 & 4 & \cellcolor{lightgray}2 & 3 & 1 & 8 & 5 & \cellcolor{lightgray}9 \\
9 & 8 & 7 & 3 & 1 & 2 & 5 & 4 & 6 \\
5 & 4 & 9 & 8 & 6 & 7 & 2 & 3 & 1 \\
4 & 9 & 5 & 7 & 8 & 6 & 1 & 2 & 3 \\
6 & 7 & 8 & \cellcolor{lightgray}9 & 5 & 4 & 3 & 1 & \cellcolor{lightgray}2 \\
\hline \end{array}$}\quad
\scalebox{0.8}{$\arraycolsep=4pt\begin{array}{|ccccccccc|} \hline
2 & 3 & 1 & \multicolumn{1}{|c}{9} & 8 & 4 & 7 & 5 & 6 \\
3 & 1 & 2 & \multicolumn{1}{|c}{6} & 4 & 5 & 9 & 7 & 8 \\
1 & 2 & 3 & \multicolumn{1}{|c}{7} & 9 & 6 & 5 & 8 & 4 \\ \cline{1-3}
5 & 6 & 7 & 2 & 3 & \cellcolor{lightgray}1 & 8 & 4 & \cellcolor{lightgray}9 \\
8 & 7 & 9 & 3 & 1 & 2 & 4 & 6 & 5 \\
4 & 5 & 8 & 1 & 2 & 3 & 6 & 9 & 7 \\
6 & 9 & 4 & 5 & 7 & 8 & 1 & 2 & 3 \\
9 & 4 & 6 & 8 & 5 & 7 & 3 & 1 & 2 \\
7 & 8 & 5 & 4 & 6 & \cellcolor{lightgray}9 & 2 & 3 & \cellcolor{lightgray}1 \\
\hline \end{array}$}\quad
\scalebox{0.8}{$\arraycolsep=4pt\begin{array}{|ccccccccc|} \hline
3 & 1 & 2 & \multicolumn{1}{|c}{6} & 4 & 5 & 9 & 7 & 8 \\
1 & 2 & 3 & \multicolumn{1}{|c}{9} & 5 & 7 & 4 & 8 & 6 \\
2 & 3 & 1 & \multicolumn{1}{|c}{8} & 6 & 9 & 7 & 4 & 5 \\ \cline{1-3}
7 & 9 & 8 & 3 & 1 & 2 & 6 & 5 & 4 \\
4 & 8 & 6 & 1 & 2 & 3 & 5 & 9 & 7 \\
5 & 7 & 4 & 2 & \cellcolor{lightgray}3 & 1 & 8 & 6 & \cellcolor{lightgray}9 \\
9 & 6 & 5 & 7 & 8 & 4 & 3 & 1 & 2 \\
6 & 5 & 9 & 4 & 7 & 8 & 2 & 3 & 1 \\
8 & 4 & 7 & 5 & \cellcolor{lightgray}9 & 6 & 1 & 2 & \cellcolor{lightgray}3 \\
\hline \end{array}$}\quad

\vspace{0.25cm}

\scalebox{0.8}{$\arraycolsep=4pt\begin{array}{|ccccccccc|} \hline
4 & 5 & 6 & 8 & 2 & 3 & 1 & 9 & 7 \\
\cellcolor{lightgray}5 & 6 & 4 & 7 & 3 & 1 & 8 & 2 & \cellcolor{lightgray}9 \\
6 & 4 & 5 & 9 & 8 & 7 & 2 & 1 & 3 \\ \cline{4-6}
1 & 7 & 9 & \multicolumn{1}{|c}{4} & 5 & 6 & \multicolumn{1}{|c}{3} & 8 & 2 \\
2 & 9 & 8 & \multicolumn{1}{|c}{5} & 6 & 4 & \multicolumn{1}{|c}{7} & 3 & 1 \\
3 & 1 & 2 & \multicolumn{1}{|c}{6} & 4 & 5 & \multicolumn{1}{|c}{9} & 7 & 8 \\ \cline{4-6}
8 & 3 & 7 & 2 & 1 & 9 & 5 & 6 & 4 \\
7 & 8 & 3 & 1 & 9 & 2 & 4 & 5 & 6 \\
\cellcolor{lightgray}9 & 2 & 1 & 3 & 7 & 8 & 6 & 4 & \cellcolor{lightgray}5 \\
\hline \end{array}$}\quad
\scalebox{0.8}{$\arraycolsep=4pt\begin{array}{|ccccccccc|} \hline
5 & 6 & \cellcolor{lightgray}4 & 2 & 3 & 7 & 8 & 1 & \cellcolor{lightgray}9 \\
6 & 4 & 5 & 8 & 7 & 9 & 1 & 3 & 2 \\
4 & 5 & 6 & 1 & 2 & 8 & 3 & 9 & 7 \\ \cline{4-6}
9 & 8 & 1 & \multicolumn{1}{|c}{5} & 6 & 4 & \multicolumn{1}{|c}{7} & 2 & 3 \\
3 & 1 & 2 & \multicolumn{1}{|c}{6} & 4 & 5 & \multicolumn{1}{|c}{9} & 7 & 8 \\
7 & 9 & 3 & \multicolumn{1}{|c}{4} & 5 & 6 & \multicolumn{1}{|c}{2} & 8 & 1 \\ \cline{4-6}
2 & 7 & 8 & 3 & 9 & 1 & 4 & 5 & 6 \\
8 & 2 & 7 & 9 & 1 & 3 & 6 & 4 & 5 \\
1 & 3 & \cellcolor{lightgray}9 & 7 & 8 & 2 & 5 & 6 & \cellcolor{lightgray}4 \\
\hline \end{array}$}\quad
\scalebox{0.8}{$\arraycolsep=4pt\begin{array}{|ccccccccc|} \hline
6 & 4 & 5 & 7 & 9 & 8 & 3 & 2 & 1 \\
4 & 5 & 6 & 1 & 8 & 3 & 2 & 9 & 7 \\
5 & \cellcolor{lightgray}6 & 4 & 2 & 7 & 1 & 8 & 3 & \cellcolor{lightgray}9 \\ \cline{4-6}
3 & 1 & 2 & \multicolumn{1}{|c}{6} & 4 & 5 & \multicolumn{1}{|c}{9} & 7 & 8 \\
9 & 2 & 7 & \multicolumn{1}{|c}{4} & 5 & 6 & \multicolumn{1}{|c}{1} & 8 & 3 \\
8 & 3 & 9 & \multicolumn{1}{|c}{5} & 6 & 4 & \multicolumn{1}{|c}{7} & 1 & 2 \\ \cline{4-6}
7 & 8 & 1 & 9 & 3 & 2 & 6 & 4 & 5 \\
1 & 7 & 8 & 3 & 2 & 9 & 5 & 6 & 4 \\
2 & \cellcolor{lightgray}9 & 3 & 8 & 1 & 7 & 4 & 5 & \cellcolor{lightgray}6 \\
\hline \end{array}$}\quad

\vspace{0.25cm}

\scalebox{0.8}{$\arraycolsep=4pt\begin{array}{|ccccccccc|} \hline
7 & 9 & 8 & 1 & 6 & 2 & 5 & 4 & 3 \\
8 & 7 & 9 & 3 & 2 & 4 & 6 & 5 & 1 \\
9 & 8 & 7 & 5 & 1 & 3 & 4 & 6 & 2 \\
4 & 3 & 5 & 7 & 9 & 8 & 2 & 1 & 6 \\
6 & 5 & 1 & 8 & 7 & 9 & 3 & 2 & 4 \\
2 & 4 & 6 & 9 & 8 & 7 & 1 & 3 & 5 \\ \cline{7-9}
1 & 2 & 3 & 4 & 5 & 6 & \multicolumn{1}{|c}{7} & 8 & 9 \\
5 & 6 & 4 & 2 & 3 & 1 & \multicolumn{1}{|c}{8} & 9 & 7 \\
3 & 1 & 2 & 6 & 4 & 5 & \multicolumn{1}{|c}{9} & 7 & 8 \\
\hline \end{array}$}\quad
\scalebox{0.8}{$\arraycolsep=4pt\begin{array}{|ccccccccc|} \hline
9 & 8 & 7 & 3 & 5 & 1 & 4 & 6 & 2 \\
7 & 9 & 8 & 2 & 1 & 6 & 5 & 4 & 3 \\
8 & 7 & 9 & 4 & 3 & 2 & 6 & 5 & 1 \\
6 & 2 & 4 & 9 & 8 & 7 & 1 & 3 & 5 \\
5 & 4 & 3 & 7 & 9 & 8 & 2 & 1 & 6 \\
1 & 6 & 5 & 8 & 7 & 9 & 3 & 2 & 4 \\ \cline{7-9}
3 & 1 & 2 & 6 & 4 & 5 & \multicolumn{1}{|c}{8} & 9 & 7 \\
2 & 3 & 1 & 5 & 6 & 4 & \multicolumn{1}{|c}{9} & 7 & 8 \\
4 & 5 & 6 & 1 & 2 & 3 & \multicolumn{1}{|c}{7} & 8 & 9 \\
\hline \end{array}$}\quad
\scalebox{0.8}{$\arraycolsep=4pt\begin{array}{|ccccccccc|} \hline
8 & 7 & \cellcolor{lightgray}9 & 5 & 1 & 6 & 2 & 3 & \cellcolor{lightgray}4 \\
\cellcolor{lightgray}9 & 8 & 7 & 4 & 6 & 2 & 3 & 1 & \cellcolor{lightgray}5 \\
7 & \cellcolor{lightgray}9 & 8 & 3 & 5 & 4 & 1 & 2 & \cellcolor{lightgray}6 \\
2 & 4 & 3 & 8 & 7 & \cellcolor{lightgray}9 & 5 & 6 & \cellcolor{lightgray}1 \\
1 & 3 & 5 & \cellcolor{lightgray}9 & 8 & 7 & 6 & 4 & \cellcolor{lightgray}2 \\
6 & 2 & 1 & 7 & \cellcolor{lightgray}9 & 8 & 4 & 5 & \cellcolor{lightgray}3 \\ \cline{7-9}
4 & 5 & 6 & 1 & 2 & 3 & \multicolumn{1}{|c}{9} & 7 & 8 \\
3 & 1 & 2 & 6 & 4 & 5 & \multicolumn{1}{|c}{7} & 8 & 9 \\
\cellcolor{lightgray}5 & \cellcolor{lightgray}6 & \cellcolor{lightgray}4 & \cellcolor{lightgray}2 & \cellcolor{lightgray}3 & \cellcolor{lightgray}1 & \multicolumn{1}{|c}{8} & 9 & 7 \\
\hline \end{array}$}\quad

    \captionof{figure}{A $\LC(3^3)$}
    \label{fig:3-3-3}
\endgroup

\end{append}

\section{Exceptions}

\begin{append}
\label{app: 6-6-5}
The following arrays form the layers of a $\LC(6^25^1)$.

\vspace{0.25cm}

\begingroup
\centering
\scalebox{0.6}{$\arraycolsep=2pt\begin{array}{|ccccccccccccccccc|} \hline
3 & 1 & 2 & 6 & 4 & \multicolumn{1}{c|}{5} & 9 & 7 & 8 & 12 & 10 & 11 & 13 & 17 & 16 & 15 & 14 \\
2 & 3 & 1 & 5 & 6 & \multicolumn{1}{c|}{4} & 8 & 9 & 7 & 11 & 12 & 10 & 17 & 16 & 15 & 14 & 13 \\
1 & 2 & 3 & 4 & 5 & \multicolumn{1}{c|}{6} & 7 & 8 & 9 & 10 & 11 & 12 & 16 & 14 & 17 & 13 & 15 \\
6 & 4 & 5 & 3 & 1 & \multicolumn{1}{c|}{2} & 10 & 16 & 17 & 15 & 13 & 14 & 9 & 7 & 8 & 11 & 12 \\
5 & 6 & 4 & 2 & 3 & \multicolumn{1}{c|}{1} & 17 & 11 & 16 & 14 & 15 & 13 & 8 & 9 & 7 & 12 & 10 \\
4 & 5 & 6 & 1 & 2 & \multicolumn{1}{c|}{3} & 16 & 17 & 12 & 13 & 14 & 15 & 7 & 8 & 9 & 10 & 11 \\ \cline{1-6}
15 & 13 & 14 & 9 & 7 & 8 & 3 & 1 & 2 & 4 & 5 & 6 & 12 & 10 & 11 & 16 & 17 \\
14 & 15 & 13 & 8 & 9 & 7 & 2 & 3 & 1 & 17 & 4 & 16 & 11 & 12 & 10 & 6 & 5 \\
13 & 14 & 15 & 7 & 8 & 9 & 1 & 2 & 3 & 5 & 6 & 4 & 10 & 11 & 12 & 17 & 16 \\
7 & 16 & 17 & 12 & 10 & 11 & 5 & 14 & 6 & 3 & 1 & 2 & 15 & 4 & 13 & 8 & 9 \\
17 & 8 & 16 & 11 & 12 & 10 & 15 & 6 & 5 & 2 & 3 & 1 & 14 & 13 & 4 & 9 & 7 \\
16 & 17 & 9 & 10 & 11 & 12 & 6 & 5 & 13 & 1 & 2 & 3 & 4 & 15 & 14 & 7 & 8 \\
12 & 10 & 11 & 13 & 16 & 17 & 14 & 15 & 4 & 9 & 7 & 8 & 1 & 5 & 6 & 3 & 2 \\
11 & 12 & 10 & 16 & 17 & 14 & 13 & 4 & 15 & 8 & 9 & 7 & 6 & 3 & 5 & 2 & 1 \\
10 & 11 & 12 & 17 & 15 & 16 & 4 & 13 & 14 & 7 & 8 & 9 & 5 & 6 & 2 & 1 & 3 \\
8 & 9 & 7 & 14 & 13 & 15 & 11 & 12 & 10 & 16 & 17 & 5 & 3 & 2 & 1 & 4 & 6 \\
9 & 7 & 8 & 15 & 14 & 13 & 12 & 10 & 11 & 6 & 16 & 17 & 2 & 1 & 3 & 5 & 4 \\
\hline \end{array}$}\quad
\scalebox{0.6}{$\arraycolsep=2pt\begin{array}{|ccccccccccccccccc|} \hline
1 & 2 & 3 & 4 & 5 & \multicolumn{1}{c|}{6} & 7 & 8 & 9 & 10 & 11 & 12 & 17 & 16 & 15 & 14 & 13 \\
3 & 1 & 2 & 6 & 4 & \multicolumn{1}{c|}{5} & 9 & 7 & 8 & 12 & 10 & 11 & 16 & 14 & 17 & 13 & 15 \\
2 & 3 & 1 & 5 & 6 & \multicolumn{1}{c|}{4} & 8 & 9 & 7 & 11 & 12 & 10 & 13 & 17 & 16 & 15 & 14 \\
4 & 5 & 6 & 1 & 2 & \multicolumn{1}{c|}{3} & 16 & 17 & 11 & 13 & 14 & 15 & 7 & 8 & 9 & 12 & 10 \\
6 & 4 & 5 & 3 & 1 & \multicolumn{1}{c|}{2} & 12 & 16 & 17 & 15 & 13 & 14 & 9 & 7 & 8 & 10 & 11 \\
5 & 6 & 4 & 2 & 3 & \multicolumn{1}{c|}{1} & 17 & 10 & 16 & 14 & 15 & 13 & 8 & 9 & 7 & 11 & 12 \\ \cline{1-6}
13 & 14 & 15 & 7 & 8 & 9 & 1 & 2 & 3 & 5 & 6 & 4 & 10 & 11 & 12 & 17 & 16 \\
15 & 13 & 14 & 9 & 7 & 8 & 3 & 1 & 2 & 4 & 5 & 6 & 12 & 10 & 11 & 16 & 17 \\
14 & 15 & 13 & 8 & 9 & 7 & 2 & 3 & 1 & 16 & 17 & 5 & 11 & 12 & 10 & 4 & 6 \\
16 & 17 & 8 & 10 & 11 & 12 & 6 & 15 & 4 & 1 & 2 & 3 & 14 & 13 & 5 & 9 & 7 \\
9 & 16 & 17 & 12 & 10 & 11 & 13 & 4 & 6 & 3 & 1 & 2 & 5 & 15 & 14 & 7 & 8 \\
17 & 7 & 16 & 11 & 12 & 10 & 4 & 6 & 14 & 2 & 3 & 1 & 15 & 5 & 13 & 8 & 9 \\
10 & 11 & 12 & 16 & 17 & 13 & 5 & 14 & 15 & 7 & 8 & 9 & 3 & 6 & 4 & 2 & 1 \\
12 & 10 & 11 & 17 & 14 & 16 & 15 & 13 & 5 & 9 & 7 & 8 & 4 & 2 & 6 & 1 & 3 \\
11 & 12 & 10 & 15 & 16 & 17 & 14 & 5 & 13 & 8 & 9 & 7 & 6 & 4 & 1 & 3 & 2 \\
7 & 8 & 9 & 13 & 15 & 14 & 10 & 11 & 12 & 6 & 16 & 17 & 2 & 1 & 3 & 5 & 4 \\
8 & 9 & 7 & 14 & 13 & 15 & 11 & 12 & 10 & 17 & 4 & 16 & 1 & 3 & 2 & 6 & 5 \\
\hline \end{array}$}\quad
\scalebox{0.6}{$\arraycolsep=2pt\begin{array}{|ccccccccccccccccc|} \hline
2 & 3 & 1 & 5 & 6 & \multicolumn{1}{c|}{4} & 8 & 9 & 7 & 11 & 12 & 10 & 16 & 14 & 17 & 13 & 15 \\
1 & 2 & 3 & 4 & 5 & \multicolumn{1}{c|}{6} & 7 & 8 & 9 & 10 & 11 & 12 & 13 & 17 & 16 & 15 & 14 \\
3 & 1 & 2 & 6 & 4 & \multicolumn{1}{c|}{5} & 9 & 7 & 8 & 12 & 10 & 11 & 17 & 16 & 15 & 14 & 13 \\
5 & 6 & 4 & 2 & 3 & \multicolumn{1}{c|}{1} & 17 & 12 & 16 & 14 & 15 & 13 & 8 & 9 & 7 & 10 & 11 \\
4 & 5 & 6 & 1 & 2 & \multicolumn{1}{c|}{3} & 16 & 17 & 10 & 13 & 14 & 15 & 7 & 8 & 9 & 11 & 12 \\
6 & 4 & 5 & 3 & 1 & \multicolumn{1}{c|}{2} & 11 & 16 & 17 & 15 & 13 & 14 & 9 & 7 & 8 & 12 & 10 \\ \cline{1-6}
14 & 15 & 13 & 8 & 9 & 7 & 2 & 3 & 1 & 6 & 16 & 17 & 11 & 12 & 10 & 5 & 4 \\
13 & 14 & 15 & 7 & 8 & 9 & 1 & 2 & 3 & 5 & 6 & 4 & 10 & 11 & 12 & 17 & 16 \\
15 & 13 & 14 & 9 & 7 & 8 & 3 & 1 & 2 & 4 & 5 & 6 & 12 & 10 & 11 & 16 & 17 \\
17 & 9 & 16 & 11 & 12 & 10 & 4 & 13 & 5 & 2 & 3 & 1 & 6 & 15 & 14 & 7 & 8 \\
16 & 17 & 7 & 10 & 11 & 12 & 14 & 5 & 4 & 1 & 2 & 3 & 15 & 6 & 13 & 8 & 9 \\
8 & 16 & 17 & 12 & 10 & 11 & 5 & 4 & 15 & 3 & 1 & 2 & 14 & 13 & 6 & 9 & 7 \\
11 & 12 & 10 & 17 & 13 & 16 & 15 & 6 & 14 & 8 & 9 & 7 & 2 & 4 & 5 & 1 & 3 \\
10 & 11 & 12 & 14 & 16 & 17 & 6 & 15 & 13 & 7 & 8 & 9 & 5 & 1 & 4 & 3 & 2 \\
12 & 10 & 11 & 16 & 17 & 15 & 13 & 14 & 6 & 9 & 7 & 8 & 4 & 5 & 3 & 2 & 1 \\
9 & 7 & 8 & 15 & 14 & 13 & 12 & 10 & 11 & 17 & 4 & 16 & 1 & 3 & 2 & 6 & 5 \\
7 & 8 & 9 & 13 & 15 & 14 & 10 & 11 & 12 & 16 & 17 & 5 & 3 & 2 & 1 & 4 & 6 \\
\hline \end{array}$}%\quad

\vspace{0.25cm}

\scalebox{0.6}{$\arraycolsep=2pt\begin{array}{|ccccccccccccccccc|} \hline
6 & 4 & 5 & 3 & 1 & \multicolumn{1}{c|}{2} & 15 & 13 & 14 & 7 & 16 & 17 & 12 & 10 & 11 & 8 & 9 \\
5 & 6 & 4 & 2 & 3 & \multicolumn{1}{c|}{1} & 14 & 15 & 13 & 17 & 8 & 16 & 11 & 12 & 10 & 9 & 7 \\
4 & 5 & 6 & 1 & 2 & \multicolumn{1}{c|}{3} & 13 & 14 & 15 & 16 & 17 & 9 & 10 & 11 & 12 & 7 & 8 \\
3 & 1 & 2 & 6 & 4 & \multicolumn{1}{c|}{5} & 9 & 7 & 8 & 12 & 10 & 11 & 13 & 17 & 16 & 15 & 14 \\
2 & 3 & 1 & 5 & 6 & \multicolumn{1}{c|}{4} & 8 & 9 & 7 & 11 & 12 & 10 & 17 & 16 & 15 & 14 & 13 \\
1 & 2 & 3 & 4 & 5 & \multicolumn{1}{c|}{6} & 7 & 8 & 9 & 10 & 11 & 12 & 16 & 14 & 17 & 13 & 15 \\ \cline{1-6}
9 & 7 & 8 & 10 & 16 & 17 & 6 & 4 & 5 & 1 & 3 & 2 & 14 & 15 & 13 & 11 & 12 \\
8 & 9 & 7 & 17 & 11 & 16 & 5 & 6 & 4 & 13 & 15 & 14 & 2 & 3 & 1 & 12 & 10 \\
7 & 8 & 9 & 16 & 17 & 12 & 4 & 5 & 6 & 3 & 2 & 1 & 15 & 13 & 14 & 10 & 11 \\
12 & 10 & 11 & 15 & 13 & 14 & 1 & 2 & 3 & 6 & 4 & 5 & 9 & 7 & 8 & 16 & 17 \\
11 & 12 & 10 & 14 & 15 & 13 & 17 & 1 & 16 & 5 & 6 & 4 & 8 & 9 & 7 & 3 & 2 \\
10 & 11 & 12 & 13 & 14 & 15 & 2 & 3 & 1 & 4 & 5 & 6 & 7 & 8 & 9 & 17 & 16 \\
17 & 13 & 16 & 9 & 7 & 8 & 12 & 10 & 11 & 2 & 14 & 15 & 4 & 1 & 3 & 6 & 5 \\
14 & 16 & 17 & 8 & 9 & 7 & 11 & 12 & 10 & 15 & 13 & 3 & 1 & 6 & 2 & 5 & 4 \\
16 & 17 & 15 & 7 & 8 & 9 & 10 & 11 & 12 & 14 & 1 & 13 & 3 & 2 & 5 & 4 & 6 \\
15 & 14 & 13 & 11 & 12 & 10 & 16 & 17 & 2 & 8 & 9 & 7 & 6 & 5 & 4 & 1 & 3 \\
13 & 15 & 14 & 12 & 10 & 11 & 3 & 16 & 17 & 9 & 7 & 8 & 5 & 4 & 6 & 2 & 1 \\
\hline \end{array}$}\quad
\scalebox{0.6}{$\arraycolsep=2pt\begin{array}{|ccccccccccccccccc|} \hline
4 & 5 & 6 & 1 & 2 & \multicolumn{1}{c|}{3} & 13 & 14 & 15 & 16 & 17 & 8 & 10 & 11 & 12 & 9 & 7 \\
6 & 4 & 5 & 3 & 1 & \multicolumn{1}{c|}{2} & 15 & 13 & 14 & 9 & 16 & 17 & 12 & 10 & 11 & 7 & 8 \\
5 & 6 & 4 & 2 & 3 & \multicolumn{1}{c|}{1} & 14 & 15 & 13 & 17 & 7 & 16 & 11 & 12 & 10 & 8 & 9 \\
1 & 2 & 3 & 4 & 5 & \multicolumn{1}{c|}{6} & 7 & 8 & 9 & 10 & 11 & 12 & 17 & 16 & 15 & 14 & 13 \\
3 & 1 & 2 & 6 & 4 & \multicolumn{1}{c|}{5} & 9 & 7 & 8 & 12 & 10 & 11 & 16 & 14 & 17 & 13 & 15 \\
2 & 3 & 1 & 5 & 6 & \multicolumn{1}{c|}{4} & 8 & 9 & 7 & 11 & 12 & 10 & 13 & 17 & 16 & 15 & 14 \\ \cline{1-6}
7 & 8 & 9 & 16 & 17 & 11 & 4 & 5 & 6 & 15 & 14 & 13 & 3 & 1 & 2 & 12 & 10 \\
9 & 7 & 8 & 12 & 16 & 17 & 6 & 4 & 5 & 3 & 2 & 1 & 15 & 13 & 14 & 10 & 11 \\
8 & 9 & 7 & 17 & 10 & 16 & 5 & 6 & 4 & 1 & 3 & 2 & 14 & 15 & 13 & 11 & 12 \\
10 & 11 & 12 & 13 & 14 & 15 & 2 & 3 & 1 & 4 & 5 & 6 & 7 & 8 & 9 & 17 & 16 \\
12 & 10 & 11 & 15 & 13 & 14 & 1 & 2 & 3 & 6 & 4 & 5 & 9 & 7 & 8 & 16 & 17 \\
11 & 12 & 10 & 14 & 15 & 13 & 16 & 17 & 2 & 5 & 6 & 4 & 8 & 9 & 7 & 1 & 3 \\
13 & 16 & 17 & 7 & 8 & 9 & 10 & 11 & 12 & 14 & 15 & 3 & 6 & 2 & 1 & 5 & 4 \\
16 & 17 & 14 & 9 & 7 & 8 & 12 & 10 & 11 & 13 & 1 & 15 & 2 & 5 & 3 & 4 & 6 \\
17 & 15 & 16 & 8 & 9 & 7 & 11 & 12 & 10 & 2 & 13 & 14 & 1 & 3 & 4 & 6 & 5 \\
14 & 13 & 15 & 10 & 11 & 12 & 3 & 16 & 17 & 7 & 8 & 9 & 5 & 4 & 6 & 2 & 1 \\
15 & 14 & 13 & 11 & 12 & 10 & 17 & 1 & 16 & 8 & 9 & 7 & 4 & 6 & 5 & 3 & 2 \\
\hline \end{array}$}\quad
\scalebox{0.6}{$\arraycolsep=2pt\begin{array}{|ccccccccccccccccc|} \hline
5 & 6 & 4 & 2 & 3 & \multicolumn{1}{c|}{1} & 14 & 15 & 13 & 17 & 9 & 16 & 11 & 12 & 10 & 7 & 8 \\
4 & 5 & 6 & 1 & 2 & \multicolumn{1}{c|}{3} & 13 & 14 & 15 & 16 & 17 & 7 & 10 & 11 & 12 & 8 & 9 \\
6 & 4 & 5 & 3 & 1 & \multicolumn{1}{c|}{2} & 15 & 13 & 14 & 8 & 16 & 17 & 12 & 10 & 11 & 9 & 7 \\
2 & 3 & 1 & 5 & 6 & \multicolumn{1}{c|}{4} & 8 & 9 & 7 & 11 & 12 & 10 & 16 & 14 & 17 & 13 & 15 \\
1 & 2 & 3 & 4 & 5 & \multicolumn{1}{c|}{6} & 7 & 8 & 9 & 10 & 11 & 12 & 13 & 17 & 16 & 15 & 14 \\
3 & 1 & 2 & 6 & 4 & \multicolumn{1}{c|}{5} & 9 & 7 & 8 & 12 & 10 & 11 & 17 & 16 & 15 & 14 & 13 \\ \cline{1-6}
8 & 9 & 7 & 17 & 12 & 16 & 5 & 6 & 4 & 3 & 2 & 1 & 15 & 13 & 14 & 10 & 11 \\
7 & 8 & 9 & 16 & 17 & 10 & 4 & 5 & 6 & 1 & 3 & 2 & 14 & 15 & 13 & 11 & 12 \\
9 & 7 & 8 & 11 & 16 & 17 & 6 & 4 & 5 & 14 & 13 & 15 & 1 & 2 & 3 & 12 & 10 \\
11 & 12 & 10 & 14 & 15 & 13 & 3 & 16 & 17 & 5 & 6 & 4 & 8 & 9 & 7 & 2 & 1 \\
10 & 11 & 12 & 13 & 14 & 15 & 2 & 3 & 1 & 4 & 5 & 6 & 7 & 8 & 9 & 17 & 16 \\
12 & 10 & 11 & 15 & 13 & 14 & 1 & 2 & 3 & 6 & 4 & 5 & 9 & 7 & 8 & 16 & 17 \\
16 & 17 & 13 & 8 & 9 & 7 & 11 & 12 & 10 & 15 & 1 & 14 & 5 & 3 & 2 & 4 & 6 \\
17 & 14 & 16 & 7 & 8 & 9 & 10 & 11 & 12 & 2 & 15 & 13 & 3 & 4 & 1 & 6 & 5 \\
15 & 16 & 17 & 9 & 7 & 8 & 12 & 10 & 11 & 13 & 14 & 3 & 2 & 1 & 6 & 5 & 4 \\
13 & 15 & 14 & 12 & 10 & 11 & 17 & 1 & 16 & 9 & 7 & 8 & 4 & 6 & 5 & 3 & 2 \\
14 & 13 & 15 & 10 & 11 & 12 & 16 & 17 & 2 & 7 & 8 & 9 & 6 & 5 & 4 & 1 & 3 \\
\hline \end{array}$}%\quad

\vspace{0.25cm}

\scalebox{0.6}{$\arraycolsep=2pt\begin{array}{|ccccccccccccccccc|} \hline
14 & 7 & 9 & 12 & 10 & 11 & 4 & 16 & 17 & 3 & 1 & 2 & 15 & 13 & 8 & 5 & 6 \\
7 & 15 & 8 & 11 & 12 & 10 & 17 & 5 & 16 & 2 & 3 & 1 & 14 & 9 & 13 & 6 & 4 \\
9 & 8 & 13 & 10 & 11 & 12 & 16 & 17 & 6 & 1 & 2 & 3 & 7 & 15 & 14 & 4 & 5 \\
12 & 10 & 11 & 7 & 8 & 9 & 15 & 13 & 14 & 6 & 4 & 5 & 3 & 1 & 2 & 16 & 17 \\
11 & 12 & 10 & 17 & 7 & 16 & 14 & 15 & 13 & 5 & 6 & 4 & 2 & 3 & 1 & 9 & 8 \\
10 & 11 & 12 & 8 & 9 & 7 & 13 & 14 & 15 & 4 & 5 & 6 & 1 & 2 & 3 & 17 & 16 \\ \cline{7-12}
1 & 16 & 17 & 15 & 13 & 14 & \multicolumn{1}{|c}{9} & 7 & 8 & 12 & 10 & \multicolumn{1}{c|}{11} & 6 & 4 & 5 & 2 & 3 \\
17 & 2 & 16 & 14 & 15 & 13 & \multicolumn{1}{|c}{8} & 9 & 7 & 11 & 12 & \multicolumn{1}{c|}{10} & 5 & 6 & 4 & 3 & 1 \\
16 & 17 & 3 & 13 & 14 & 15 & \multicolumn{1}{|c}{7} & 8 & 9 & 10 & 11 & \multicolumn{1}{c|}{12} & 4 & 5 & 6 & 1 & 2 \\
6 & 4 & 5 & 3 & 1 & 2 & \multicolumn{1}{|c}{12} & 10 & 11 & 9 & 7 & \multicolumn{1}{c|}{8} & 13 & 17 & 16 & 15 & 14 \\
5 & 6 & 4 & 2 & 3 & 1 & \multicolumn{1}{|c}{11} & 12 & 10 & 8 & 9 & \multicolumn{1}{c|}{7} & 17 & 16 & 15 & 14 & 13 \\
4 & 5 & 6 & 1 & 2 & 3 & \multicolumn{1}{|c}{10} & 11 & 12 & 7 & 8 & \multicolumn{1}{c|}{9} & 16 & 14 & 17 & 13 & 15 \\ \cline{7-12}
15 & 9 & 14 & 6 & 4 & 5 & 3 & 1 & 2 & 13 & 17 & 16 & 8 & 10 & 7 & 12 & 11 \\
8 & 13 & 15 & 5 & 6 & 4 & 2 & 3 & 1 & 17 & 16 & 14 & 12 & 7 & 9 & 11 & 10 \\
13 & 14 & 7 & 4 & 5 & 6 & 1 & 2 & 3 & 16 & 15 & 17 & 9 & 8 & 11 & 10 & 12 \\
2 & 3 & 1 & 16 & 17 & 8 & 5 & 6 & 4 & 15 & 14 & 13 & 11 & 12 & 10 & 7 & 9 \\
3 & 1 & 2 & 9 & 16 & 17 & 6 & 4 & 5 & 14 & 13 & 15 & 10 & 11 & 12 & 8 & 7 \\
\hline \end{array}$}\quad
\scalebox{0.6}{$\arraycolsep=2pt\begin{array}{|ccccccccccccccccc|} \hline
9 & 8 & 15 & 10 & 11 & 12 & 16 & 17 & 5 & 1 & 2 & 3 & 14 & 7 & 13 & 6 & 4 \\
13 & 7 & 9 & 12 & 10 & 11 & 6 & 16 & 17 & 3 & 1 & 2 & 8 & 15 & 14 & 4 & 5 \\
7 & 14 & 8 & 11 & 12 & 10 & 17 & 4 & 16 & 2 & 3 & 1 & 15 & 13 & 9 & 5 & 6 \\
10 & 11 & 12 & 8 & 9 & 7 & 13 & 14 & 15 & 4 & 5 & 6 & 1 & 2 & 3 & 17 & 16 \\
12 & 10 & 11 & 7 & 8 & 9 & 15 & 13 & 14 & 6 & 4 & 5 & 3 & 1 & 2 & 16 & 17 \\
11 & 12 & 10 & 16 & 17 & 8 & 14 & 15 & 13 & 5 & 6 & 4 & 2 & 3 & 1 & 7 & 9 \\ \cline{7-12}
16 & 17 & 2 & 13 & 14 & 15 & \multicolumn{1}{|c}{7} & 8 & 9 & 10 & 11 & 12 & \multicolumn{1}{|c}{4} & 5 & 6 & 3 & 1 \\
3 & 16 & 17 & 15 & 13 & 14 & \multicolumn{1}{|c}{9} & 7 & 8 & 12 & 10 & 11 & \multicolumn{1}{|c}{6} & 4 & 5 & 1 & 2 \\
17 & 1 & 16 & 14 & 15 & 13 & \multicolumn{1}{|c}{8} & 9 & 7 & 11 & 12 & 10 & \multicolumn{1}{|c}{5} & 6 & 4 & 2 & 3 \\
4 & 5 & 6 & 1 & 2 & 3 & \multicolumn{1}{|c}{10} & 11 & 12 & 7 & 8 & 9 & \multicolumn{1}{|c}{17} & 16 & 15 & 14 & 13 \\
6 & 4 & 5 & 3 & 1 & 2 & \multicolumn{1}{|c}{12} & 10 & 11 & 9 & 7 & 8 & \multicolumn{1}{|c}{16} & 14 & 17 & 13 & 15 \\
5 & 6 & 4 & 2 & 3 & 1 & \multicolumn{1}{|c}{11} & 12 & 10 & 8 & 9 & 7 & \multicolumn{1}{|c}{13} & 17 & 16 & 15 & 14 \\ \cline{7-12}
14 & 15 & 7 & 4 & 5 & 6 & 1 & 2 & 3 & 17 & 16 & 13 & 9 & 12 & 8 & 11 & 10 \\
15 & 9 & 13 & 6 & 4 & 5 & 3 & 1 & 2 & 16 & 14 & 17 & 11 & 8 & 7 & 10 & 12 \\
8 & 13 & 14 & 5 & 6 & 4 & 2 & 3 & 1 & 15 & 17 & 16 & 7 & 9 & 10 & 12 & 11 \\
1 & 2 & 3 & 9 & 16 & 17 & 4 & 5 & 6 & 14 & 13 & 15 & 10 & 11 & 12 & 8 & 7 \\
2 & 3 & 1 & 17 & 7 & 16 & 5 & 6 & 4 & 13 & 15 & 14 & 12 & 10 & 11 & 9 & 8 \\
\hline \end{array}%
$}

\endgroup
\end{append}

\begin{append}
\label{app: 9-9-7}
The following arrays can be used to construct a $\LC(9^27^1)$.

\vspace{0.25cm}

\begingroup
\centering
\scalebox{0.6}{$\arraycolsep=1pt%
$}

\endgroup

\vspace{0.25cm}

Let $L(i,j)$ for all $i,j\in[25]$ be the entries of the first array, let $M(i,j)$ be the entries of the second array, and let $C$ be a $25\times25\times25$ array.

For $n\in[25]$ and $m\in\mathbb{Z}$, let $$f(n,m) = \begin{cases}
    n-m\pmod9, & \text{if $n\in[9]$,}\\
    9+ (n-m\pmod9), & \text{if $n\in 9+[9]$,}\\
    n, & \text{otherwise.}
\end{cases}$$
For all $i,j\in[25]$ and $k\in[9]$, let
$$C(i,j,k) = f\Big(L\big(f(i,k-1),f(j,k-1)\big),1-k\Big).$$

For $n\in[25]$, let $$g(n) = \begin{cases}
    n + 9 \pmod{18}, & \text{if $n\in[18]$,}\\
    n, & \text{otherwise,}
\end{cases}$$
and for all $i,j\in[25]$ and $k\in9+[9]$, let
$$C(i,j,k) = g\Big(C\big(g(i),g(j),k-9\big)\Big).$$

For $n\in[25]$ and $m\in\mathbb{Z}$, let $$h(n,m) = \begin{cases}
    18+ (n-m\pmod7), & \text{if $n\in 18+[7]$,}\\
    n, & \text{otherwise.}
\end{cases}$$
Then, for all $i,j\in[25]$ and $k\in[7]$, let
$$C(i,j,18+k) = h\Big(f\Big(M\big(h(i,k-1),j\big),1-k\Big),1-k\Big).$$

\end{append}

\begin{append}
\label{app: 9-9-8}
The following arrays can be used to construct a $\LC(9^28^1)$.

\vspace{0.25cm}

\begingroup
\centering
\scalebox{0.6}{$\arraycolsep=1pt\begin{array}{|cccccccccccccccccccccccccc|} \hline
1 & 9 & 8 & 7 & 6 & 5 & 4 & 3 & 2 & \multicolumn{1}{|c}{24} & 21 & 15 & 25 & 14 & 19 & 26 & 20 & 12 & 16 & 10 & 11 & 17 & 13 & 18 & 23 & 22 \\
2 & 1 & 9 & 8 & 7 & 6 & 5 & 4 & 3 & \multicolumn{1}{|c}{16} & 18 & 19 & 13 & 10 & 11 & 24 & 22 & 23 & 14 & 21 & 17 & 25 & 12 & 20 & 26 & 15 \\
3 & 2 & 1 & 9 & 8 & 7 & 6 & 5 & 4 & \multicolumn{1}{|c}{19} & 13 & 22 & 15 & 16 & 23 & 10 & 21 & 20 & 17 & 25 & 24 & 26 & 18 & 14 & 12 & 11 \\
4 & 3 & 2 & 1 & 9 & 8 & 7 & 6 & 5 & \multicolumn{1}{|c}{25} & 22 & 16 & 20 & 23 & 13 & 21 & 10 & 18 & 12 & 15 & 26 & 14 & 11 & 17 & 19 & 24 \\
5 & 4 & 3 & 2 & 1 & 9 & 8 & 7 & 6 & \multicolumn{1}{|c}{23} & 19 & 21 & 18 & 25 & 24 & 17 & 26 & 16 & 15 & 13 & 14 & 11 & 22 & 12 & 20 & 10 \\
6 & 5 & 4 & 3 & 2 & 1 & 9 & 8 & 7 & \multicolumn{1}{|c}{18} & 25 & 23 & 24 & 13 & 26 & 20 & 12 & 19 & 10 & 11 & 22 & 15 & 17 & 21 & 14 & 16 \\
7 & 6 & 5 & 4 & 3 & 2 & 1 & 9 & 8 & \multicolumn{1}{|c}{22} & 12 & 20 & 21 & 26 & 16 & 23 & 25 & 15 & 18 & 17 & 10 & 19 & 24 & 13 & 11 & 14 \\
8 & 7 & 6 & 5 & 4 & 3 & 2 & 1 & 9 & \multicolumn{1}{|c}{15} & 20 & 18 & 26 & 24 & 25 & 16 & 19 & 22 & 21 & 14 & 12 & 10 & 23 & 11 & 13 & 17 \\
9 & 8 & 7 & 6 & 5 & 4 & 3 & 2 & 1 & \multicolumn{1}{|c}{26} & 10 & 24 & 17 & 22 & 18 & 11 & 15 & 21 & 19 & 16 & 20 & 13 & 14 & 23 & 25 & 12 \\ \cline{1-9}
14 & 15 & 13 & 17 & 12 & 18 & 16 & 11 & 23 & 8 & 1 & 6 & 19 & 9 & 4 & 2 & 5 & 26 & 20 & 22 & 3 & 21 & 25 & 24 & 10 & 7 \\
25 & 17 & 14 & 13 & 11 & 16 & 18 & 12 & 15 & 21 & 6 & 1 & 9 & 20 & 3 & 7 & 8 & 24 & 22 & 26 & 23 & 5 & 10 & 2 & 4 & 19 \\
12 & 26 & 15 & 14 & 18 & 17 & 11 & 16 & 10 & 20 & 24 & 4 & 5 & 7 & 21 & 3 & 9 & 1 & 13 & 6 & 2 & 8 & 19 & 25 & 22 & 23 \\
16 & 11 & 20 & 15 & 17 & 12 & 10 & 14 & 13 & 9 & 26 & 5 & 7 & 2 & 6 & 22 & 1 & 3 & 24 & 8 & 18 & 23 & 4 & 19 & 21 & 25 \\
10 & 13 & 16 & 21 & 14 & 15 & 12 & 17 & 11 & 7 & 9 & 25 & 22 & 5 & 2 & 4 & 24 & 8 & 3 & 23 & 19 & 18 & 20 & 1 & 6 & 26 \\
15 & 12 & 17 & 11 & 24 & 14 & 13 & 10 & 16 & 2 & 4 & 9 & 23 & 19 & 5 & 1 & 7 & 25 & 6 & 18 & 21 & 20 & 26 & 22 & 3 & 8 \\
11 & 14 & 10 & 16 & 15 & 19 & 17 & 13 & 12 & 4 & 5 & 8 & 6 & 21 & 20 & 9 & 3 & 7 & 23 & 24 & 25 & 22 & 2 & 26 & 1 & 18 \\
18 & 16 & 12 & 10 & 13 & 11 & 14 & 15 & 17 & 5 & 23 & 3 & 8 & 1 & 22 & 25 & 4 & 2 & 26 & 19 & 6 & 7 & 21 & 9 & 24 & 20 \\
17 & 18 & 11 & 12 & 10 & 13 & 15 & 22 & 14 & 1 & 3 & 26 & 2 & 4 & 9 & 19 & 23 & 6 & 25 & 20 & 5 & 24 & 8 & 16 & 7 & 21 \\
20 & 19 & 26 & 22 & 25 & 24 & 23 & 21 & 18 & 3 & 7 & 17 & 12 & 15 & 10 & 14 & 11 & 4 & 5 & 1 & 16 & 9 & 6 & 8 & 2 & 13 \\
13 & 23 & 25 & 19 & 20 & 21 & 22 & 26 & 24 & 10 & 8 & 11 & 4 & 18 & 17 & 12 & 14 & 5 & 7 & 2 & 15 & 6 & 1 & 3 & 16 & 9 \\
21 & 24 & 18 & 23 & 26 & 22 & 19 & 25 & 20 & 13 & 16 & 10 & 14 & 11 & 1 & 8 & 6 & 17 & 9 & 12 & 4 & 2 & 7 & 5 & 15 & 3 \\
26 & 25 & 19 & 18 & 22 & 23 & 24 & 20 & 21 & 17 & 11 & 14 & 3 & 6 & 12 & 5 & 16 & 13 & 2 & 9 & 7 & 1 & 15 & 10 & 8 & 4 \\
22 & 10 & 23 & 24 & 21 & 20 & 26 & 19 & 25 & 6 & 2 & 12 & 11 & 17 & 8 & 15 & 18 & 14 & 1 & 3 & 13 & 4 & 16 & 7 & 9 & 5 \\
23 & 22 & 24 & 20 & 19 & 25 & 21 & 18 & 26 & 12 & 15 & 2 & 1 & 8 & 14 & 13 & 17 & 10 & 11 & 7 & 9 & 16 & 3 & 4 & 5 & 6 \\
24 & 21 & 22 & 25 & 16 & 26 & 20 & 23 & 19 & 11 & 14 & 7 & 10 & 3 & 15 & 18 & 13 & 9 & 8 & 4 & 1 & 12 & 5 & 6 & 17 & 2 \\
19 & 20 & 21 & 26 & 23 & 10 & 25 & 24 & 22 & 14 & 17 & 13 & 16 & 12 & 7 & 6 & 2 & 11 & 4 & 5 & 8 & 3 & 9 & 15 & 18 & 1 \\
\hline \end{array}$}\quad
\scalebox{0.6}{$\arraycolsep=1pt\begin{array}{|cccccccccccccccccccccccccc|} \hline
19 & 20 & 22 & 17 & 26 & 21 & 23 & 15 & 12 & 18 & 14 & 25 & 10 & 24 & 16 & 11 & 13 & 8 & 3 & 5 & 7 & 6 & 1 & 9 & 2 & 4 \\
13 & 22 & 26 & 21 & 18 & 24 & 19 & 25 & 16 & 9 & 10 & 15 & 20 & 11 & 23 & 17 & 12 & 14 & 4 & 6 & 8 & 7 & 2 & 1 & 3 & 5 \\
17 & 14 & 24 & 23 & 19 & 10 & 22 & 26 & 21 & 15 & 1 & 11 & 16 & 20 & 12 & 25 & 18 & 13 & 5 & 7 & 9 & 8 & 3 & 2 & 4 & 6 \\
26 & 18 & 15 & 19 & 21 & 22 & 11 & 23 & 25 & 14 & 16 & 2 & 12 & 17 & 20 & 13 & 24 & 10 & 6 & 8 & 1 & 9 & 4 & 3 & 5 & 7 \\
25 & 26 & 10 & 16 & 22 & 19 & 21 & 12 & 23 & 11 & 15 & 17 & 3 & 13 & 18 & 20 & 14 & 24 & 7 & 9 & 2 & 1 & 5 & 4 & 6 & 8 \\
23 & 24 & 20 & 11 & 17 & 25 & 26 & 21 & 13 & 22 & 12 & 16 & 18 & 4 & 14 & 10 & 19 & 15 & 8 & 1 & 3 & 2 & 6 & 5 & 7 & 9 \\
14 & 25 & 23 & 26 & 12 & 18 & 24 & 20 & 22 & 16 & 21 & 13 & 17 & 10 & 5 & 15 & 11 & 19 & 9 & 2 & 4 & 3 & 7 & 6 & 8 & 1 \\
24 & 15 & 19 & 25 & 23 & 13 & 10 & 22 & 26 & 20 & 17 & 21 & 14 & 18 & 11 & 6 & 16 & 12 & 1 & 3 & 5 & 4 & 8 & 7 & 9 & 2 \\
21 & 23 & 16 & 24 & 25 & 26 & 14 & 11 & 20 & 13 & 19 & 18 & 22 & 15 & 10 & 12 & 7 & 17 & 2 & 4 & 6 & 5 & 9 & 8 & 1 & 3 \\
9 & 5 & 25 & 1 & 24 & 7 & 2 & 4 & 17 & 19 & 20 & 22 & 8 & 26 & 21 & 23 & 6 & 3 & 12 & 14 & 16 & 15 & 10 & 18 & 11 & 13 \\
18 & 1 & 6 & 20 & 2 & 23 & 8 & 3 & 5 & 4 & 22 & 26 & 21 & 9 & 24 & 19 & 25 & 7 & 13 & 15 & 17 & 16 & 11 & 10 & 12 & 14 \\
6 & 10 & 2 & 7 & 20 & 3 & 25 & 9 & 4 & 8 & 5 & 24 & 23 & 19 & 1 & 22 & 26 & 21 & 14 & 16 & 18 & 17 & 12 & 11 & 13 & 15 \\
5 & 7 & 11 & 3 & 8 & 20 & 4 & 24 & 1 & 26 & 9 & 6 & 19 & 21 & 22 & 2 & 23 & 25 & 15 & 17 & 10 & 18 & 13 & 12 & 14 & 16 \\
2 & 6 & 8 & 12 & 4 & 9 & 20 & 5 & 24 & 25 & 26 & 1 & 7 & 22 & 19 & 21 & 3 & 23 & 16 & 18 & 11 & 10 & 14 & 13 & 15 & 17 \\
22 & 3 & 7 & 9 & 13 & 5 & 1 & 19 & 6 & 23 & 24 & 20 & 2 & 8 & 25 & 26 & 21 & 4 & 17 & 10 & 12 & 11 & 15 & 14 & 16 & 18 \\
7 & 21 & 4 & 8 & 1 & 14 & 6 & 2 & 19 & 5 & 25 & 23 & 26 & 3 & 9 & 24 & 20 & 22 & 18 & 11 & 13 & 12 & 16 & 15 & 17 & 10 \\
20 & 8 & 21 & 5 & 9 & 2 & 15 & 7 & 3 & 24 & 6 & 19 & 25 & 23 & 4 & 1 & 22 & 26 & 10 & 12 & 14 & 13 & 17 & 16 & 18 & 11 \\
4 & 19 & 9 & 22 & 6 & 1 & 3 & 16 & 8 & 21 & 23 & 7 & 24 & 25 & 26 & 5 & 2 & 20 & 11 & 13 & 15 & 14 & 18 & 17 & 10 & 12 \\ \cline{19-26}
3 & 4 & 5 & 6 & 7 & 8 & 9 & 1 & 2 & 12 & 13 & 14 & 15 & 16 & 17 & 18 & 10 & 11 & \multicolumn{1}{|c}{19} & 26 & 25 & 24 & 23 & 22 & 21 & 20 \\
11 & 12 & 13 & 14 & 15 & 16 & 17 & 18 & 10 & 2 & 3 & 4 & 5 & 6 & 7 & 8 & 9 & 1 & \multicolumn{1}{|c}{20} & 19 & 26 & 25 & 24 & 23 & 22 & 21 \\
1 & 2 & 3 & 4 & 5 & 6 & 7 & 8 & 9 & 10 & 11 & 12 & 13 & 14 & 15 & 16 & 17 & 18 & \multicolumn{1}{|c}{21} & 20 & 19 & 26 & 25 & 24 & 23 & 22 \\
12 & 13 & 14 & 15 & 16 & 17 & 18 & 10 & 11 & 3 & 4 & 5 & 6 & 7 & 8 & 9 & 1 & 2 & \multicolumn{1}{|c}{22} & 21 & 20 & 19 & 26 & 25 & 24 & 23 \\
16 & 17 & 18 & 10 & 11 & 12 & 13 & 14 & 15 & 7 & 8 & 9 & 1 & 2 & 3 & 4 & 5 & 6 & \multicolumn{1}{|c}{23} & 22 & 21 & 20 & 19 & 26 & 25 & 24 \\
10 & 11 & 12 & 13 & 14 & 15 & 16 & 17 & 18 & 1 & 2 & 3 & 4 & 5 & 6 & 7 & 8 & 9 & \multicolumn{1}{|c}{24} & 23 & 22 & 21 & 20 & 19 & 26 & 25 \\
8 & 9 & 1 & 2 & 3 & 4 & 5 & 6 & 7 & 17 & 18 & 10 & 11 & 12 & 13 & 14 & 15 & 16 & \multicolumn{1}{|c}{25} & 24 & 23 & 22 & 21 & 20 & 19 & 26 \\
15 & 16 & 17 & 18 & 10 & 11 & 12 & 13 & 14 & 6 & 7 & 8 & 9 & 1 & 2 & 3 & 4 & 5 & \multicolumn{1}{|c}{26} & 25 & 24 & 23 & 22 & 21 & 20 & 19 \\
\hline \end{array}$}

\endgroup

\vspace{0.25cm}

Use the same functions, $f$, $g$ and $h$, and the same method as with the $\LC(9^27^1)$, replacing $[25]$ and $[7]$ with $[26]$ and $[8]$ respectively.

\end{append}

\begin{append}
\label{app: 12-12-7}
As with the $\LC(9^27^1)$ and $\LC(9^28^1)$, we construct a $\LC(12^27^1)$, $C$, from a subset of the layers. Of the following thirteen arrays, let the first 6 be the first layers $(\cdot,\cdot,[6])$ of $C$, and let the other 7 arrays be the last layers $(\cdot,\cdot,24+[7])$ of $C$.

For $n\in[31]$, let $$f(n) = \begin{cases}
    n+6, & \text{if $n\in[6]$ or $n\in12+[6]$,}\\
    n-6, & \text{if $n\in 6+[6]$ or $n\in18+[6]$,}\\
    n, & \text{otherwise.}
\end{cases}$$
For all $i,j\in[31]$ and $k\in6+[6]$, let
$$C(i,j,k) = f\Big(C\big(f(i),f(j),k-6\big)\Big).$$

For $n\in[31]$, let $$g(n) = \begin{cases}
    n + 12, & \text{if $n\in[12]$,}\\
    n - 12, & \text{if $n\in12+[12]$,}\\
    n, & \text{otherwise,}
\end{cases}$$
and let $$h(n) = \begin{cases}
    n + 18, & \text{if $n\in[6]$,}\\
    n + 6, & \text{if $n\in6+[6]$,}\\
    n - 6, & \text{if $n\in12+[6]$,}\\
    n - 18, & \text{if $n\in18+[6]$,}\\
    n, & \text{otherwise.}
\end{cases}$$
Then for all $i,j\in[31]$ and $k\in12+[12]$, let
$$C(i,j,k) = h\Big(C\big(h(i),g(j),k-12\big)\Big).$$

\begingroup
\centering
\scalebox{0.6}{$\arraycolsep=1pt
$}\quad

\endgroup
\end{append}

\end{document}